\documentclass[a4paper,11pt,oneside,openright,reqno]{amsart}
\usepackage[english]{babel}
\usepackage[colorlinks,citecolor=green,linkcolor=red]{hyperref}
\usepackage{amsthm}
\usepackage{color}
\usepackage{mathrsfs}
\usepackage{amsmath}
\usepackage{mathtools}
\usepackage{amsfonts}
\usepackage{amssymb}
\usepackage{accents}
\usepackage{bm}
\usepackage{physics}
\usepackage{enumitem}
\usepackage[margin=0.7in]{geometry}
\usepackage{cleveref}
\usepackage{esint}
\usepackage{nicefrac}
\usepackage{yfonts}
\usepackage{xcolor}

\numberwithin{equation}{section}

\newtheorem{thm}{Theorem}[section]
\newtheorem*{thm*}{Theorem}
\newtheorem{lem}[thm]{Lemma}
\newtheorem{prop}[thm]{Proposition}

\theoremstyle{definition}
\newtheorem{defn}[thm]{Definition}
\theoremstyle{remark}
\newtheorem{rem}[thm]{Remark}
\newtheorem{ex}[thm]{Example}

\newcommand{\fr}{}

\def\Xint#1{\mathchoice
	{\XXint\displaystyle\textstyle{#1}}%
	{\XXint\textstyle\scriptstyle{#1}}%
	{\XXint\scriptstyle\scriptscriptstyle{#1}}%
	{\XXint\scriptscriptstyle\scriptscriptstyle{#1}}%
	\!\int}
\def\XXint#1#2#3{{\setbox0=\hbox{$#1{#2#3}{\int}$}
		\vcenter{\hbox{$#2#3$}}\kern-.5\wd0}}
\def\dashint{\Xint-}

\newcommand{\mres}{\mathbin{\vrule height 1.6ex depth 0pt width 0.13ex\vrule height 0.13ex depth 0pt width 1.3ex}}
\newcommand{\mressmall}{\mathbin{\vrule height 1.2ex depth 0pt width 0.11ex\vrule height 0.10ex depth 0pt width 1.0ex}}

\let\mathrm\operatorname
\newcommand{\RCD}{{\mathrm {RCD}}}

\newcommand{\XX}{X}
\newcommand{\ZZ}{Z}
\newcommand{\mass}{{\mathsf{m}}}

\newcommand{\LIP}{{\mathrm {LIP}}}
\newcommand{\BV}{{\mathrm {BV}}}

\newcommand{\RR}{\mathbb{R}}
\newcommand{\NN}{\mathbb{N}}

\newcommand{\HH}{\mathcal{H}}
\newcommand{\LL}{\mathcal{L}}

\newcommand{\sign}{\mathrm {sign}}
\newcommand{\hess}{\mathrm{Hess}}

\newcommand{\defeq}{\vcentcolon=}
\newcommand{\eqdef}{=\vcentcolon}

\let\oldchi\chi
\renewcommand{\chi}{\text{\raisebox{\depth}{\(\oldchi\)}}}
\let\phi\varphi
\let\epsilon\varepsilon
\let\bar\overline
\renewcommand{\det}{\operatorname{det}}
\let\big\relax
\let\bigg\Big

\let\lft\Big
\let\rgt\Big

\DeclareMathOperator*{\esssup}{ess\,sup}

\DeclareMathOperator\supp{supp}
\title{Lower Ricci Curvature Bounds and the Orientability of Spaces}
\author{Camillo Brena}
\address{Institute for Advanced Study. Einstein Drive 1, 08540 Princeton, New Jersey}
\email{cbrena@ias.edu}

\author{Elia Bruè}
\address{Bocconi University, Department of Decision Sciences. Via Sarfatti 25, 20136 Milano}
\email{elia.brue@unibocconi.it}

\author{Alessandro  Pigati}
\address{Bocconi University, Department of Decision Sciences. Via Sarfatti 25, 20136 Milano}
\email{alessandro.pigati@unibocconi.it}

\begin{document}

\raggedbottom

\begin{abstract}
    We study orientability in spaces with Ricci curvature bounded below. Building on the theory developed by Honda in \cite{HondaOrient}, we establish equivalent characterizations of orientability for Ricci limit and $\RCD$ spaces in terms of the orientability of their manifold part. We prove a new stability theorem and, as a corollary, we deduce that four-manifolds with Ricci curvature bounded below and volume non-collapsing  are uniformly locally orientable. As a global counterpart of the latter, we show that four-manifolds with nonnegative Ricci curvature and Euclidean volume growth are orientable.
\end{abstract}
\maketitle

	\tableofcontents

\section{Introduction}

We extend the theory of orientability on nonsmooth spaces with Ricci curvature bounded below, originally introduced by Honda in \cite{HondaOrient} for Ricci limit spaces. Our contributions advance the theory in three key directions:
\begin{itemize}
    \item[(i)] a characterization of orientability in terms of the \emph{effective manifold part} of the space;
    \item[(ii)] an analysis of the \emph{stability of non-orientability} with respect to Gromov--Hausdorff convergence;
    \item[(iii)] the construction of a \emph{ramified orientable double cover} for non-orientable spaces, along with a discussion of its main properties, within the setting of Ricci limit spaces.
\end{itemize}

As an application of our theory and the recent results in \cite{brupisem}, we prove that four-dimensional manifolds with a uniform lower bound on Ricci curvature and non-collapsing volume are uniformly locally orientable. More precisely, we establish the following statements.

\begin{thm}[Uniform Local Orientability]\label{thm:UniformLocalOrientability}
    Let $(M^4,g)$ satisfy ${\rm Ric}_g \ge -3$ and ${\rm Vol}_g(B_1(p)) \ge v>0$ for some $p\in M^4$. There exists $r(v)>0$ such that $B_{r(v)}(x)$ is orientable for every $x\in B_1(p)$.
\end{thm}

\begin{thm}[Maximal Volume Growth and Orientability]\label{thm:MaximalVolumeGrowth}
    Let $(M^4,g)$ be an open manifold with ${\rm Ric}_g \ge 0$ and Euclidean volume growth. Then $M^4$ is orientable.
\end{thm}

In dimensions two and three, a stronger version of the previous statement holds: uniform lower bound on the Ricci curvature and volume non-collapsing imply {\it uniform local contractibility} of the space \cite{Zhu93}. Moreover, in these dimensions manifolds with nonnegative Ricci curvature and Euclidean volume growth are known to be diffeomorphic to $\RR^3$ \cite{Liu13, SchoenYau82}.
However, these facts do not hold in dimension at least four, as it is readily seen by looking at the Eguchi-Hanson space and its rescalings.

\begin{rem}[Orientability in Higher Dimension]
Theorems \ref{thm:UniformLocalOrientability} and \ref{thm:MaximalVolumeGrowth} admit higher-dimensional extensions under the assumption of an (almost) split structure of the space. By applying the same argument, it follows that a manifold $(M^n, g)$ with ${\rm Ric}_g \geq 0$ and Euclidean volume growth is orientable, provided there exists a tangent cone at infinity that splits isometrically an $\mathbb{R}^{n-4}$ factor.
\end{rem}

Without additional splitting or symmetry assumptions, our results are sharp, as demonstrated by the following example of Otsu \cite{Otsu91}.

\begin{ex}[Sharpness of Theorem \ref{thm:UniformLocalOrientability}]\label{ex:Otsu}
There exists a sequence of smooth metrics $g_k$ on $S^3 \times \mathbb{RP}^2$ with  ${\rm Ric}_{g_k} \ge 1$ and ${\rm Vol}_{g_k}(S^3 \times \mathbb{RP}^2)\ge v>0$  such that 
\begin{equation}
    (S^3 \times \mathbb{RP}^2, g_k) \to (\Sigma(S^2 \times \mathbb{RP}^2), d) ,
    \quad 
    \text{as $k\to \infty$},
\end{equation}
in the Gromov--Hausdorff topology. Here, $\Sigma(S^2 \times \mathbb{RP}^2)$ denotes the spherical suspension over $S^2 \times \mathbb{RP}^2$, and $d$ is the limit distance. The limit space has two singular points $x, y$, where the tangent cone is $C(S^2 \times \mathbb{RP}^2)$, which is non-orientable according to our notion of orientability (see Definition \ref{dfn:orientable}).
Moreover, it is not hard to check that arbitrarily small balls around the points $x_k, y_k \to x, y$ in $S^3 \times \mathbb{RP}^2$, approximating these singular points, are not orientable for sufficiently large $k$.
\end{ex}

\begin{ex}[Sharpness of Theorem \ref{thm:MaximalVolumeGrowth}]
A construction similar to the previous example allows us to build a metric with nonnegative Ricci curvature and Euclidean volume growth on \( \mathbb{R}^3 \times \mathbb{RP}^2 \), showing that Theorem \ref{thm:MaximalVolumeGrowth} fails in dimension five. More precisely,
using polar coordinates for the $\mathbb{R}^3$ factor,
one considers a doubly warped product metric 
\begin{equation}\label{ex:metric}
    dr^2 + f_1(r)^2 g_{S^2} + f_2(r)^2 g_{\mathbb{RP}^2} ,
\end{equation}
where $f_1(r)$ is concave and satisfies $f_1(r) = r$ for $0 \le r \le 10^{-1}$, and $f_1(r) = 10^{-1}(r + 100)$ for $r \ge 1000$, while $f_2(r)$ is convex and given by $f_2(r) = \delta$ for $0 \le r \le 10^{-1}$, and $f_2(r) = \delta'(r + 100)$ for $r \ge 1000$. By selecting $0 < \delta < 1$ sufficiently small and a suitable $0 < \delta' < \delta$ depending on $\delta$, one can smoothly define $f_1(r)$ and $f_2(r)$ in the transition region $10^{-1} \leq r \leq 1000$ to ensure that the Ricci curvature is nonnegative. For further details on the construction of $f_1(r)$ and $f_2(r)$, we refer the reader to \cite[Section 7]{BrueNaberSemola23Milnor}.
Note that the metric \eqref{ex:metric} is smooth and isometric to $C(S^2_{10^{-1}} \times \mathbb{RP}^2_{\delta'})$ for $r > 1000$. Consequently, it exhibits Euclidean volume growth.

In dimension \( n = 10 \), Dancer and Wang have constructed an Einstein metric with Euclidean volume growth on \( \mathbb{R}^4 \times \mathbb{RP}^6 \) \cite{DancerWang03}; see also the discussion in \cite{HondaOrient}.
\end{ex}

\subsection{Orientable RCD Spaces}

We propose a notion of orientability for non-collapsed $\RCD(-(n-1),n)$ spaces without boundary, based on the {\it orientability of their manifold part}. Our notion turns out to be equivalent to the one introduced by Honda \cite{HondaOrient} for Ricci limit spaces, which involves the existence of a (weak notion of) volume form. We refer to Section \ref{intro:VolumeForemCurrents} for the detailed statement.

\begin{defn}[Orientable $\RCD$ Spaces]\label{dfn:orientable}
	Let $(\XX, d, \mathcal{H}^n)$ be a non-collapsed $\RCD(-(n-1),n)$ space with no boundary. We say that $(\XX, d, \mathcal{H}^n)$ is \emph{orientable} if every open set $A \subseteq \XX$ which is a topological manifold is orientable.
\end{defn}

Definition \ref{dfn:orientable} is obviously consistent with the standard notion of orientability for smooth Riemannian manifolds. However, in practical applications, it is not necessary to verify the orientability of every subset $A \subseteq \XX$ as described. It suffices to identify one sufficiently large subset with the required properties.

\begin{prop}\label{thm:char1}  
	A non-collapsed $\RCD(-(n-1),n)$ space $(\XX, d, \mathcal{H}^n)$ with no boundary is orientable according to Definition \ref{dfn:orientable} if and only if there exists an open subset $A \subseteq \XX$ which is an orientable topological manifold and satisfies $\HH^{n-1}(\XX \setminus A) =0$.
\end{prop}


\begin{rem}[Local Orientability] \label{rmk.local.ori}
	Both Definition \ref{dfn:orientable} and Proposition \ref{thm:char1} can be applied locally. Specifically, in a non-collapsed $\RCD$ space without boundary, we say that the ball $B_1(x)$ is {\it orientable} if $A \cap B_1(x)$ is orientable for any $A \subseteq \XX$ as in Definition \ref{dfn:orientable}. Furthermore, to establish orientability of $B_1(x)$, it suffices to find an open set $A$ which is an orientable manifold with $\mathcal{H}^{n-1}(B_1(x) \setminus A) =0$.
\end{rem}

The characterization of orientability in Proposition \ref{thm:char1} is particularly convenient for our setting of non-collapsed $\RCD$ spaces without boundary, where the $\epsilon$-regularity theorem \cite{Cheeger-Colding97I,DPG17} ensures that, if $0<\epsilon < \epsilon(n)$, the open subset
\begin{equation}\label{eq:Aeps}
	A_{\epsilon}(\XX)\defeq \lft\lbrace x\in \XX \, : \, d_{GH}(B_r(x), B_r(0^n)) < \epsilon r \  \text{for some $r\in (0,\epsilon)$}\rgt\rbrace 
\end{equation}
is a connected topological manifold without boundary  whose complement has Hausdorff dimension  smaller than or equal to $n-2$, making it a suitable candidate to test orientability.

\begin{ex}[Cones and Spherical Suspensions]\label{ex:cones}
	Let $(\XX, d, \mathcal{H}^n)$ be a non-collapsed $\RCD(n-1,n)$ space with no boundary. It is well-known that the cone $C(\XX)$ and the spherical suspension $\Sigma(\XX)$ are $\RCD(0,n+1)$ spaces. By Theorem \ref{thm:char1}, it is easy to see that both $C(\XX)$ and $\Sigma(\XX)$ are orientable if and only if $\XX$ is orientable. 
	Indeed, consider the subset $A_{\epsilon}(\XX) \subseteq \XX$ for $\epsilon<\epsilon(n)$ small enough. For the cone, the set $(0, \infty) \times A_\epsilon(\XX)$ serves as a suitable open subset for testing orientability, and it is orientable if and only if $\XX$ is orientable. A similar argument holds for spherical suspensions.
\end{ex}

\begin{ex}[Three-Dimensional Cones]
Non-collapsed $\RCD$ spaces of dimension two  are topological manifolds. In particular, in the case of empty boundary, they are orientable according to Definition \ref{dfn:orientable} if and only if they are orientable manifolds. A three-dimensional cone $C(\XX^2)$ where $\XX^2$ has no boundary is orientable if and only if $\XX^2$ is homeomorphic to $S^2$. Indeed, $\XX^2$ is homeomorphic either to $S^2$ or to $\mathbb{RP}^2$, as a consequence of \cite{LytSta18} and standard topological arguments: see e.g.\ \cite[Corollary 3.3]{brupisem}. In particular, $C(\mathbb{RP}^2)$ is an example of non-orientable (but simply connected) $\RCD$ space with no boundary.
\end{ex}

\begin{ex}[$\RCD(2,3)$ Spaces]\label{ex:RCD(2,3)or}
Let $(Z^3,d, \HH^3)$ be a non-collapsed $\RCD(2,3)$ space without boundary whose tangent cones are all homeomorphic to $C(S^2)$. As a consequence of \cite{brupisem}, $Z^3$ is a topological manifold covered by $S^3$, and hence an orientable manifold. Indeed, \cite[Theorem 1.8]{brupisem} shows that $Z^3$ is a topological $3$-manifold without boundary and, by the solution to the Poincaré conjecture, its universal cover (which is compact as it is $\RCD(2,3)$) is homeomorphic to $S^3$.
Moreover, a homeomorphism $h:S^3\to S^3$ without fixed points cannot reverse the orientation: if this happened, the induced map
$h_*:H_3(S^3)\to H_3(S^3)$ would be given by $h_*([\sigma])=-[\sigma]$, and hence $h$ would have Lefschetz number $2\neq0$,
contradicting the fixed-point theorem of Lefschetz.
This class of spaces is significant, as they arise as cross-sections of tangent cones (also at infinity) to non-collapsed four-dimensional Ricci limit spaces \cite{brupisem}.

\end{ex}

\begin{ex}[Four-Dimensional Cones]\label{ex:4cones}
In light of Examples \ref{ex:cones} and \ref{ex:RCD(2,3)or}, any tangent cone to a non-collapsed four-dimensional Ricci limit space $C(\XX^3)$ is orientable. This property is a key ingredient in the proof of Theorems \ref{thm:UniformLocalOrientability} and \ref{thm:MaximalVolumeGrowth}, along with the stability result presented in the following section.
\end{ex}

\subsection{Volume Form and Currents}
\label{intro:VolumeForemCurrents}
Our notion of orientability is equivalent to the existence of a volume form $\omega \in L^\infty(\Lambda^n T^*X)$, and agrees with the one proposed by Honda in \cite{HondaOrient}.

\begin{thm}[Volume Form vs Orientability]
	\label{thm:OrientabilityVolumeForm}
	A non-collapsed $\RCD(-(n-1),n)$ space $(\XX, d, \mathcal{H}^n)$ with no boundary is orientable according to Definition \ref{dfn:orientable} if and only if there exists $\omega \in L^\infty(\Lambda^n T^*\XX)$ such that $|\omega| =1$ a.e.\ and $\omega \,\cdot\, \eta \in W^{1,2}(\XX)$ for every $\eta = f_0\, df_1 \wedge \ldots \wedge df_n$ with $f_i \in \mathrm{Test}(\XX)$ with compact support. Moreover, the volume form $\omega$ is unique up to scalar multiplication by $-1$.
\end{thm}

Actually, to ensure that $\XX$ is orientable it is enough to build a less rigid form: see Theorem \ref{vefdacscdcsd}, where all the detailed equivalences are collected.



\begin{rem}[Harmonic Volume Form]
	 It is possible to show that the volume form $\omega \in L^\infty(\Lambda^n T^*\XX)$ in Theorem \ref{thm:OrientabilityVolumeForm} is harmonic. Specifically, $\delta \omega = 0$ in the sense of distributions, i.e.\ $\int_\XX \omega\,\cdot\,d\eta\,\dd\HH^n=0$ for every  $\eta\in \mathrm{TestForms}_{n-1}(\XX)$ with compact support (see Theorem \ref{vefdacscdcsd}). 
\end{rem}

For the sake of completeness, we  investigate also the link between  the existence of a non-vanishing harmonic top form, and the existence of a non-vanishing top dimensional metric current with no boundary \cite{AmbrosioKirchheim00,lan11}. This provides us with another equivalent notion of orientability.
As in the rest of the paper, the investigation is limited to non-collapsed $\RCD$ spaces with \emph{no boundary}. For Ricci limit spaces, an analogue investigation has been conducted in \cite{HondaOrient}. Our main result in this direction is basically a corollary of Proposition \ref{formsandcurrents}, in which the link between forms and currents is established, and is as follows
(see also Theorems \ref{brfvaed} and \ref{vfedacadc}).

\begin{thm}[Metric Currents vs Orientability]
		A non-collapsed $\RCD(-(n-1),n)$ space $(\XX, d, \mathcal{H}^n)$ with no boundary is orientable according to Definition \ref{dfn:orientable} if and only if there exists a nonzero metric $n$-current $T$, with $|T|\in L^\infty$, with no boundary. Moreover, the current $T$ is unique up to scalar multiplication by $c\in \RR\setminus \{0\}$.
\end{thm}

\subsection{Stability of (Non-)Orientability}\label{intro:stability}
In this section, we introduce two stability results:
\begin{itemize}
	\item[(i)] GH-limits of non-collapsed, orientable $\RCD$ spaces are still orientable;
	
	\item[(ii)] GH-limits of non-collapsed, uniformly bounded, non-orientable Ricci limit spaces are still non-orientable.
\end{itemize}

The first stability result, (i), in the context of non-collapsed Ricci limits, was originally proved in \cite{HondaOrient}. With our topological definition of orientability (Definition \ref{dfn:orientable}), it follows almost immediately.

\begin{thm}[Stability of Orientability]\label{vfeadadcsd}
	Let $(\XX_k,d_k,\HH^n,p_k)\xrightarrow{GH}(\XX,d,\HH^n,p)$ be a sequence of non-collapsed $\RCD(-(n-1),n)$ spaces with no boundary.
	If $(\XX_k,d_k,\HH^n)$ is orientable for every $k$, then  $(\XX,d,\HH^n)$ is orientable. 
\end{thm}

\begin{rem}
Although we will provide a topological proof, this fact could also be proved by passing to the limit the metric $n$-currents $T_k$ associated with each space $\XX_k$. The integral current spaces $(\XX_k,d_k,T_k)$ can be shown to converge to the limit $(\XX,d,T)$
in the intrinsic flat metric. For the terminology, we refer the reader to \cite{MatPor}, where intrinsic flat convergence is shown to be essentially equivalent to Gromov--Hausdorff convergence.
\end{rem}

The second statement, (ii), is significantly more challenging and requires an analysis of the \textit{ramified double cover} introduced in the next section. Currently, we can establish this result for $\RCD$ spaces only under additional assumptions (see Theorem \ref{nonstability} below). The clearest statement of this result is available within the class of Ricci limit spaces.

\begin{thm}[Stability of Non-Orientable Ricci Limits]\label{thm:stabnonorRicciLimit}
	Let $(M^n_k, g_k, p_k) \xrightarrow{GH} (\XX, d, p)$ be a sequence satisfying the uniform bounds $\mathrm{Vol}(B_1(p_k)) \ge v > 0$ and $\mathrm{Ric}_{g_k} \ge -(n-1)$. Assume that, for some $R>0$, $B_R(p_k)$ is non-orientable for every $k$.  Then $(\XX,d,p)$ is a non-orientable Ricci limit space. 
\end{thm}

As noted by Honda in \cite{HondaOrient}, the blow-up at a point $x \in \mathbb{RP}^2$ (or, more generally, at any point of a non-orientable manifold) provides an example of a pointed GH sequence of non-collapsed, non-orientable spaces converging to $\mathbb{R}^2$, which is orientable.
This does not contradict our Theorem \ref{thm:stabnonorRicciLimit}, as $B_r(x)$ is orientable for sufficiently small $r > 0$. Roughly speaking, to preserve non-orientability in the limit, it is necessary to have a uniformly bounded non-orientable subset along the sequence.

A corollary of Theorem \ref{thm:stabnonorRicciLimit} is that (non-)orientability is detected by the {\it effective regular part} of the manifold.
\begin{thm}
    Let $(M^n, g)$ satisfy  $\mathrm{Vol}(B_1(p)) \ge v > 0$ and $\mathrm{Ric}_{g} \ge -(n-1)$. Assume that $B_1(p)$ is non-orientable. For $\epsilon<\epsilon(n)$, if $r<r(n,v,\epsilon)$, then $M^n \setminus \overline{B_r(S_{\epsilon,r}^{n-2})}$ is non-orientable, where
    \begin{equation}
    S_{\epsilon,r}^{n-2}\defeq \{x\in M^n \, :\, \text{for no $r\le s<1$ the ball $B_s(x)$ is $(n-1,\epsilon)$-symmetric}\}.
\end{equation}
\end{thm}

 \begin{proof}[Proof of Theorems \ref{thm:UniformLocalOrientability} and \ref{thm:MaximalVolumeGrowth} given Theorem \ref{thm:stabnonorRicciLimit}]
 	
If the statement were false,  we could find a sequence of smooth $4$-manifolds $(M^4_k,g_k,p_k) \xrightarrow{GH} ( C(Z^3), d, p)$ as $k\to \infty$, with ${\rm Ric}_{g_k} \ge - 1/k$ and uniform volume non-collapsing such that $B_1(p_k)$ is not orientable. As a consequence of \cite{brupisem}, $Z^3$ is orientable (see Examples \ref{ex:RCD(2,3)or}, \ref{ex:4cones}). Our stability result Theorem \ref{thm:stabnonorRicciLimit} implies that $Z^3$ is not orientable, a contradiction.
 
 The proof of Theorem \ref{thm:MaximalVolumeGrowth} is analogous, but the contradicting sequence is obtained by blowing down $(M^4,g)$. Indeed, the blow-down of a manifold with Euclidean volume growth is a metric cone and, again by \cite{brupisem}, the cross-section of the latter is orientable.
 \end{proof}

\subsection{Ramified Orientable Double Cover}\label{intro:cover}

It is well-known that any non-orientable Riemannian manifold $(M^n, g)$ admits a degree-two Riemannian cover $\pi : \widehat{M} \to M$ which is orientable. We can think of $M$ as the quotient of $\widehat{M}$ by an isometric, free involution $\Gamma$, where $\widehat{M}$ is orientable.

However, in the context of singular spaces with Ricci curvature bounded below, the previous statement is too strong to hold in full generality. It becomes necessary to consider \textit{ramified covers} with singular points, corresponding to the fixed points of the involution.

\begin{ex}[Spherical Suspension]
    Consider $\XX^3 \defeq \Sigma(\mathbb{RP}^2)$, the spherical suspension over the two-dimensional projective space endowed with the standard metric. It turns out that $\XX^3$ is a non-collapsed $\RCD(0,3)$ space which is non-orientable according to our definition. A natural double cover is the map $\pi : \Sigma(S^2) \to \Sigma(\mathbb{RP}^2)$, which acts on $S^2$ in the obvious way. Away from the tips, this map serves as a Riemannian double cover; however, the tips are fixed points of the involution.
\end{ex}

Building on the previous example, we propose a natural construction to exhibit a ramified double cover of a  non-orientable non-collapsed $\RCD(-(n-1),n)$ space $(\XX, d)$ without boundary. We begin by considering the effective manifold part $A_\epsilon(X)$ for sufficiently small $\epsilon < \epsilon(n)$. This set is a non-orientable topological manifold, with a complement of Hausdorff dimension smaller than or equal to $n-2$.

Next, we take the standard double cover of $A\defeq A_\epsilon(X)$ and complete it to obtain a complete metric space. As a result, we obtain a new geodesic metric measure space $(\widehat{\XX}, \widehat{d}, \mathcal{H}^n)$, along with a map $\pi: \widehat{\XX} \to \XX$, which serves as a double cover when restricted to $\widehat{A}\defeq\pi^{-1}(A)$, and an involution $\Gamma:\widehat\XX\rightarrow\widehat\XX$ such that $\pi\circ\Gamma=\pi$.

In this construction, singular points may emerge when completing the double cover of the manifold part $A_\epsilon(\XX)$. This phenomenon is illustrated by the spherical suspension over $\mathbb{R P}^2$, where the manifold part corresponds to the complement of the tip. In the absence of singular points, it is straightforward to verify that $\widehat{\XX}$ is an $\RCD(-(n-1),n)$ space locally isometric to $(\XX, d)$. However, the presence of singular points complicates this verification (see Section \ref{intro:Open} below). At present, we can only establish the $\RCD$ property of $\widehat{\XX}$ when $\XX$ is smoothable, relying on an approximating sequence $(M^n_k, g_k, p_k) \xrightarrow{GH} (\XX, d, p)$ and the associated sequence of double covers $(\widehat{M}^n_k, \widehat{g}_k, \hat{p}_k) \xrightarrow{GH} (\widehat{\XX}, \widehat{d}, \hat{p})$. Below, we focus on Ricci limit spaces, with further details on the general $\RCD$ case provided in Theorem \ref{defndouble}.

\begin{thm}[Ramified Double Cover for Ricci Limit Spaces]\label{defndoubleIntro}
	Let $(M^n_k, g_k, p_k) \xrightarrow{GH} (\XX, d, p)$ be a sequence satisfying the uniform bounds $\mathrm{Vol}(B_1(p_k)) \ge v > 0$ and $\mathrm{Ric}_{g_k} \ge -(n-1)$. Assume that $(\XX, d, \mathcal{H}^n)$ is non-orientable (according to Definition \ref{dfn:orientable}). Then there exists a non-collapsed $\RCD(-(n-1), n)$ space $(\widehat{\XX}, \widehat{d}, \mathcal{H}^n)$, along with an isometric involution $\Gamma$ acting on $\widehat{\XX}$, such that the following properties hold.
	\begin{itemize}
		\item[(i)] $\widehat{\XX}$ is orientable and $\pi: \widehat{\XX} \to \XX$ is the projection map with respect to the action  $\Gamma$.
		
		\item[(ii)] For every open set $A \subseteq \XX$ which is a connected topological manifold, $\pi: \widehat{A} \defeq \pi^{-1}(A) \to A$ is a double cover. Moreover, $\widehat{A}$ is connected if and only if $A$ is non-orientable.
	\end{itemize}
	The pair $((\widehat{\XX}, \widehat{d}, \mathcal{H}^n), \pi)$ is unique up to isomorphism, in the sense that if $((\widehat{\XX}', \widehat{d}', \mathcal{H}^n), \pi')$ is any other pair, then there exists an isometry $\Phi: (\widehat{\XX}, \widehat{d}) \rightarrow (\widehat{\XX}', \widehat{d}')$ such that $\pi' \circ \Phi = \pi$.
\end{thm}

Beyond its theoretical significance, Theorem \ref{defndoubleIntro} is a crucial component in establishing the stability result given in Theorem \ref{thm:stabnonorRicciLimit}. A similar result is well-known in the context of Alexandrov spaces \cite{HarveySearle17}.

\subsection{Locally (Non-)Orientable Points}

Let $(\XX, d, \mathcal{H}^n)$ be an $\RCD(-(n-1),n)$ space without boundary. We say that $x \in \XX$ is \textit{locally orientable} if there exists an $r > 0$ such that $B_r(x)$ is orientable according to Definition \ref{dfn:orientable}. Otherwise, we say that $x \in \XX$ is \textit{locally non-orientable}.

\begin{ex}[Locally Non-Orientable Points]
	The simplest example of a locally non-orientable point is the tip of $\XX^3 = C(\mathbb{RP}^2)$. In the setting of Ricci limit spaces, an example is provided by the singular points of $\XX^5 = \Sigma(S^2 \times \mathbb{RP}^2)$, as discussed in Example \ref{ex:Otsu}. 
\end{ex}

In the setting of smoothable $\RCD$ spaces we have the following structure theorem for locally non-orientable points.

\begin{thm}[Local Non-Orientable Points]\label{thm:localnonorientable}
	Let $(M^n_k, g_k, p_k) \xrightarrow{GH} (\XX, d, p)$ be a sequence satisfying the uniform bounds $\mathrm{Vol}(B_1(p_k)) \ge v > 0$ and $\mathrm{Ric}_{g_k} \ge -(n-1)$. Then the following are equivalent:
	\begin{enumerate}
		\item $x \in \XX$ is locally non-orientable;
		\item $X$ is non-orientable and $\pi^{-1}(x)$ is a singleton, where $\pi : \widehat \XX \to \XX$ denotes the ramified double cover;
		\item the cross-section of every tangent cone at $x$ is non-orientable.
	\end{enumerate}
	Furthermore, the set of locally non-orientable points $\XX_{\rm LNO}\subseteq \XX$ is empty for dimensions $2 \le n \le 4$, and satisfies the volume bound $\HH^{n-5}(B_1(p)\cap \XX_{LNO}) \le C(n,v)$ when $n\ge 5$.
\end{thm}

Example \ref{ex:Otsu} illustrates the sharpness of the previous statement. Consider the non-collapsed Ricci limit space $\XX^5 = \Sigma(S^2 \times \mathbb{RP}^2)$. The ramified double cover is given by $\pi : \Sigma(S^2 \times S^2) \to \Sigma(S^2 \times \mathbb{RP}^2)$, with the involution $\Gamma$ fixing two singular points. Both of these points are locally non-orientable, and the blow-up in each case is $C(S^2 \times \mathbb{RP}^2)$. The Hausdorff dimension of the set of locally non-orientable points is precisely $n - 5$.

\subsection{Open Problems and Future Developments}\label{intro:Open}

Although inherently synthetic, the orientation theory proposed in this paper is both satisfactory and complete only within the framework of Ricci limit spaces. In the general $\RCD$ setting, the primary open question concerns the $\RCD$ regularity of the ramified double cover.

\medskip
\noindent
{\bf Question.} Let $(\XX, d, \HH^n)$ be a non-collapsed $\RCD(-(n-1),n)$ space without boundary. Is the ramified double cover (see Theorem \ref{defndouble}) $(\widehat{\XX}, \widehat{d},\HH^n)$ an $\RCD(-(n-1),n)$ space?
\medskip

A positive answer to this question would yield two immediate corollaries:
\begin{itemize}
	\item[(i)] stability of the non-orientability property in the $\RCD$ framework, as stated in Theorem \ref{nonstability};
	
	\item[(ii)] a version of Theorem \ref{thm:localnonorientable}, the structure theorem for $\XX_{LNO} \subseteq \XX$, applicable to $\RCD$ spaces.
\end{itemize}

While the statement of (i) remains identical to that for Ricci limit spaces, the version of (ii) in the $\RCD$ context will be weaker. Specifically, cones such as $C(\mathbb{RP}^2)$ may appear within the class of $\RCD$ spaces, implying that the Hausdorff dimension of $\XX_{LNO}$ should be bounded by $n-3$ rather than $n-5$, and we are guaranteed that $\XX_{LNO} = \emptyset$ only when $n=2$.

\begin{rem}
The fact that the ramified double cover  $(\widehat{\XX}, \widehat{d},\HH^n)$ is still $\RCD(-(n-1),n)$
for Ricci limits is primarily used to conclude that the effective manifold part of $\widehat{\XX}$ is connected,
which is crucial in the proof of stability of non-orientability.

We briefly elaborate on the difficulties faced while attempting to prove this fact for general non-collapsed $\RCD(-(n-1),n)$ spaces $({\XX}, {d},\HH^n)$ without boundary.

First, letting $\widehat G\defeq \{\hat x:\Gamma\hat x\ne \hat x\}\supseteq\widehat A$ be the points that are not fixed by the involution $\Gamma$, it holds that $\dim_\HH(\widehat\XX\setminus\widehat G)\le n-2$.
We can prove that every point of $\widehat G$ has a neighborhood  isometric to a neighborhood of its projection in $(\XX, d, \HH^n)$. Hence, we can easily prove the $\RCD(K,N)$ condition locally on the open set $\widehat G$, in the sense that the $(K,N)$-Bochner inequality holds for test functions with support contained  in $\widehat G$. 
If we knew that  $(\widehat{\XX}, \widehat{d},\HH^n)$ is $\RCD(K',\infty)$ for some $K'$, then we would have enough regularity to extend the $(K,N)$-Bochner inequality  to the whole space. 
It is then clear that the difficulties lie in examining $(\widehat{\XX}, \widehat{d},\HH^n)$ locally around the fixed points of the involution.

We remark that the $\RCD$ condition can be seen as a convexity of suitable entropy functionals along geodesics in the space of probability measures. Hence, what prevents us from completing the proof is the fact that many geodesics starting and ending in $\widehat G$ may pass through the singular set $\widehat{\XX}\setminus\widehat G$. Excluding this possibility, i.e.\ proving that $\widehat G$ is, in a sense, quantitatively convex, would allow us to conclude (notice that this would be  implied by the  $\RCD$ property of $(\widehat{\XX}, \widehat{d},\HH^n)$).
\end{rem}

Another promising research direction involves extending the theory to the class of {\it $\RCD$ spaces with boundary}. In the setting of Ricci limit spaces with convex boundaries, we expect this extension to be relatively straightforward. However, the general $\RCD$ framework may present greater challenges, as fundamental structural questions—such as the characterization of tangent cones at boundary points and the stability of the boundary—remain open in this broader context \cite{brueboundary}.

\subsection*{Acknowledgements}{
The authors wish to thank S.\ Honda for his valuable comments on a preliminary version of the manuscript.
Most of this work was carried out while C.B.\ was a PhD student at  Scuola Normale Superiore visiting E.B.\ and A.P.\ at Bocconi University.
This material is based upon work supported by the National Science Foundation under Grant No.\ DMS-1926686.}

\section{Equivalent Definitions of Orientability}
\label{ori.equiv.sec}

In this section, we provide a more detailed discussion of the equivalent definitions of orientability within the $\RCD$ framework. Definition \ref{dfn:orientable} corresponds to statement (1) in Theorem \ref{vefdacscdcsd} below. For foundational material on calculus, forms, and currents in $\RCD$ spaces, we refer the reader to Section \ref{sec:Preliminaries}.

In the following theorem and throughout the paper, when we refer to an open subset $A$ of a non-collapsed $\RCD$ space as a topological manifold, we mean that $A$ inherits the topology of the ambient space. Moreover, the dimension of $A$ as a manifold is the essential dimension of the $\RCD$ space.

\begin{thm}[Equivalent Definitions of Orientability]\label{vefdacscdcsd}
	Let $(\XX,d,\HH^n)$ be a non-collapsed $\RCD(-(n-1),n)$ space with no boundary. Then the following are equivalent.
	\begin{enumerate}
		\item Every open $A\subseteq\XX$ which is a topological manifold is orientable.
		\item There exists an open $A\subseteq\XX$ which is an orientable topological manifold with $\HH^{n-1}(\XX\setminus A)=0$.
		\item For $\epsilon<\epsilon(n)$, $A_\epsilon(\XX)$ defined in \eqref{eq:Aeps} is orientable.
		\item There exists $0\neq \omega\in L^\infty(\Lambda^n T\XX)$ such that  $|\omega|$ is constant and  $\omega\,\cdot\,\eta\in W^{1,2}(\XX)$ for every $\eta\in \mathrm{TestForms}_{n}(\XX)$ with compact support.
		\item There exists $0\ne \omega\in L^\infty(\Lambda^n T\XX)$ such that, for every $\eta\in \mathrm{TestForms}_{n-1}(\XX)$ with compact support, it holds $\int_\XX \omega\,\cdot\,d\eta\,\dd\HH^n=0$.
		\item There exists a nonzero  metric $n$-current $T$, with $|T|\in L^\infty$ and no boundary.
	\end{enumerate}
	If any of these holds, the form $\omega$ is unique up to scalar multiplication by $c\in \RR\setminus \{0\}$, and satisfies:
	\begin{enumerate}
		\item[(i)] $|\omega|$ is constant;
		\item[(ii)] for every $\eta\in \mathrm{TestForms}_{n}(\XX)$, $\nabla (\omega\,\cdot\,\eta)=\omega\,\cdot\,\nabla \eta$ holds $\HH^n$-a.e.
	\end{enumerate} 
	Similarly, the current $T$ as above is unique up to scalar multiplication by $c\in \RR\setminus \{0\}$.
\end{thm}

We remark that, to obtain the equivalence of items (1)--(6), the assumption in item (4) can be slightly weakened to $|\omega| \ge c > 0$ in place of $|\omega|$ constant. This is discussed further in Remark \ref{remmodulo}.

\subsection{Preliminaries from Algebraic Topology}

We recall some basic notions from algebraic topology (see e.g.\ \cite[Section 3.3]{Hatcher}).
Given a topological manifold $M^n$ with $n\ge1$, an \emph{orientation} is a continuous choice of
a generator $\mu_x\in H_n(M,M\setminus\{x\})$.
Here continuity means that, for each $x$, we can find a compact neighborhood $\bar U$ homeomorphic to the closed Euclidean ball
and such that, for some generator $\mu\in H_n(M,M\setminus U)$,
we have $i_y(\mu_y)=\mu$ for all $y\in U$, where $i_y:H_n(M,M\setminus\{y\})\to H_n(M,M\setminus U)$
is the canonical isomorphism given by the deformation retraction of $M\setminus\{y\}$ onto $M\setminus U$.
Note that all these groups are isomorphic to $\mathbb Z$, since we have
$$H_n(M,M\setminus U)\cong H_n(\bar U,\partial U)\cong H_n(\bar B_1(0^n),S^{n-1})\cong\mathbb{Z}$$
by excision (however, the second isomorphism is not canonical).
We say that $M$ is \emph{orientable} if it admits an orientation. It is easy to check
that $M$ is orientable if and only if each connected component is orientable.

Regardless of whether $M$ is orientable or not, given a continuous curve $\gamma:[0,1]\to M$
and $\mu_0\in H_n(M,M\setminus\{\gamma(0)\})$,
there exists a unique continuous path $(\mu_t)_{t\in[0,1]}$ of generators $\mu_t\in H_n(M,M\setminus\{\gamma(t)\})$,
where continuity is understood as above. If $\gamma$ is a loop (i.e.\ $\gamma(0)=\gamma(1)$),
we say that $\gamma$ \emph{preserves} the orientation if $\mu_1=\mu_0$, and we say that it \emph{reverses} the orientation if instead $\mu_1=-\mu_0$; clearly, this is independent of the choice of $\mu_0$.
It is easy to check that this property of a loop is invariant under homotopies (even those moving the basepoint)
and that $M$ is orientable if and only if there is no loop reversing the orientation.

\subsection{Proof of Proposition \ref{thm:char1}}

We begin by proving Proposition \ref{thm:char1}, which establishes the equivalence $(1)\Leftrightarrow (2)$ in Theorem \ref{vefdacscdcsd}. The proof is relatively straightforward and relies on purely topological arguments, together with the fact that closed sets with vanishing $(n-1)$-Hausdorff dimension do not disconnect the space. The latter result was shown in \cite[Theorem 3.7]{Cheeger-Colding97II} for Ricci limit spaces. Following a similar argument,
it was adapted to the setting of $\RCD$ spaces in \cite[Proposition A.6]{KapMon19}.
We briefly reproduce the proof in our setting for the reader's convenience.

\begin{lem}\label{moltoconnesso}
    Let $(\XX,d,\HH^n)$ be a non-collapsed $\RCD(-(n-1),n)$ space with no boundary. Let $C\subseteq\XX$ be closed with $\HH^{n-1}(C)=0$. Then, for every $x\in \XX\setminus C$, the following holds. For $\HH^n$-a.e.\ $y\in \XX\setminus C$, there exists a geodesic $\gamma:[0,1]\rightarrow\XX$ connecting $x$ to $y$ with $\gamma([0,1])\subseteq\XX\setminus C$.
\end{lem}
\begin{proof}
    In this proof, all geodesics have to be understood to be parametrized by constant speed (close to $1$).
    Take $x,y\in\XX\setminus C$. We are going to prove that, for some $\eta>0$, there exists a set $A\subseteq B_\eta(y)$ with $\HH^n(B_\eta(y)\setminus A)=0$ such that, for every $y'\in A$, there exists a geodesic joining $x$ to $y'$ whose image is contained in $\XX\setminus C$. This will clearly be enough to conclude. For simplicity, we rescale the metric so that $d(x,y)=1$.
    
    Let $\eta\in (0,1/10)$ be such that $B_{3\eta}(x)\cup B_{3\eta}(y)\subseteq\XX\setminus C$. 
	Let $\mu_0\defeq\delta_{x}$ and let $\mu_1\defeq\frac{1}{\HH^n(B_\eta( y))}\HH^n\mres B_\eta( y)$.  
In this proof, we will denote by $D$ a constant that depends only upon the space and $x,y,\eta$ and may vary from line to line.
 By Bishop--Gromov, we see that 
	\begin{equation}
		\frac{r^n}{D}\le  \HH^n(B_r( z))\le D r^n\quad\text{for every $z\in B_{3}(x)$ and $r\in (0,3)$}.
	\end{equation}
	Now we are going to use \cite[Theorem 1.4]{Rajala12-2} (and its proof), with the same notation. Let us consider an optimal geodesic plan $\pi\in\mathrm{OptGeod}(\mu_0,\mu_1)$. Then $(e_t)_*\pi\le C(t)\HH^n$ for every $t\in (0,1]$, with $C(t)$ locally bounded in $(0,1]$. In particular, $\sup_{t\ge \eta/2} C(t)\le D$.
 
	Now take a ball $B\defeq B_{1/N}(z)\subseteq B_{3}(x)$ with $2B \cap B_{\eta}(x)=\emptyset$, for some $N\in\NN$ large, and consider the times $(t_i)_{i=0,\dots,N}$ with $t_i\defeq i/N$. Let $\sigma$ be any geodesic from $x$ to a point in $B_\eta(y)$ (note that $\pi$ is concentrated on these geodesics), and assume that $\sigma$  intersects $B$. Then $\sigma(t_i)\in 2B$ for some $i=0,\dots,N$, hence $t_i\ge \eta/2$. Therefore, denoting $\sigma_t\defeq \sigma(t)$, we have 
	\begin{equation}
		\begin{split}
			\pi(\{\sigma\,:\, \sigma_t\in B\text{ for some $t\in [0,1]$}\})&\le \sum_{i\,:\,t_i\ge \eta/2} \pi(\{\sigma\,:\, \sigma_{t_i}\in 2B\}) \\
			&\le \sum_{i\,:\,t_i\ge \eta/2} (e_{t_i})_{*}\pi (2B)\le  DN\frac{1}{N^n}\le D (1/N)^{n-1}.
		\end{split}
	\end{equation}
 Now, notice that if $\sigma$ is a geodesic joining $x$ to a point in $B_\eta(y)$ intersecting $C$, then it must intersect $C\cap B_{2}(x)$. Hence, for any $\epsilon\in (0,1)$, we  cover $C\cap \bar B_{2}(x)$ with finitely many balls $(B_j)_{j\in J}$, where we can assume also that for each $B_j=B_{r_j}(x_j)$ we have $B_j\subseteq B_{3}(x)$, $2B_j\cap B_{\eta}( x)=\emptyset$, and
	\begin{equation}
		\sum_j r_j^{n-1}\le \epsilon.
	\end{equation}
	Also, there is no loss of generality in assuming that each $r_j$ is the reciprocal of a natural number.
	By what we have shown above,
	\begin{equation}
		\begin{split}        
			\pi(\{\sigma\,:\, \sigma_t\in C\text{ for some $t\in [0,1]$}\})&\le \pi(\{\sigma\,:\, \sigma_t\in B_j\text{ for some $t\in [0,1]$ and $j\in J$}\})\\&\le \sum_j \pi(\{\sigma\,:\, \sigma_t\in B_j\text{ for some $t\in [0,1]$}\})\\
			&\le \sum_j D r_j^{n-1}\le D\epsilon.
		\end{split}
	\end{equation}
 Since $\epsilon\in (0,1)$ is arbitrary and $D$ is independent of $\epsilon$, we see that 
 \begin{equation}
     \pi(\{\sigma\,:\, \sigma_t\in C\text{ for some $t\in [0,1]$}\})=0,
 \end{equation}
 which shows that for $\HH^n$-a.e.\ $y'\in B_\eta(y)$ there exists a geodesic joining $x$ to $y'$ whose image does not intersect $C$. 
\end{proof}

\begin{proof}[Proof of Proposition \ref{thm:char1}]
    Assume that $A,A'\subseteq\XX$ are open sets and topological manifolds,
    and assume that $A$ is orientable and satisfies $\HH^{n-1}(\XX\setminus A)=0$.
    We wish to show that $A'$ is orientable as well. By possibly replacing $A'$ with $A\cup A'$,
    we can assume that $A\subseteq A'$.

    

    Both $A$ and $A'$ are pathwise-connected by Lemma \ref{moltoconnesso}.
    Assuming by contradiction that $A'$ is not orientable, let $\gamma:[0,1]\to A'$ be a loop which reverses the orientation, with $\gamma(0)=\gamma(1)\in A$,
    and let $I\subseteq(0,1)$ be a closed interval such that $\gamma([0,1]\setminus \accentset{\circ}{I})\subseteq A$.
    Since $A$ is obviously dense in $A'$, by compactness of the image of $\gamma$, we can find $\delta>0$ and $x_1,\dots,x_N\in A$
    such that $B_{10\delta}(x_i)$ is included in a topological ball $B_i\cong B_1(0^n)$, itself included in $A'$,
    and such that we can write
    $$I=[t_0,t_1]\cup\dots\cup[t_{N-1},t_N]$$
    with $\gamma([t_{i-1},t_i])\subseteq B_\delta(x_i)$.
    Clearly, we can assume that $x_N=\gamma(t_N)$, and we let $x_0\defeq\gamma(t_0)$.
    
    Thanks to Lemma \ref{moltoconnesso} again, since $d(x_{i-1},x_i)<4\delta$
    we can join the two points $x_{i-1}$ and $x_i$ with a curve $\tilde\gamma_i$ taking values in $A$
    and of length at most $4\delta$. Moreover, for each $i=0,\dots,N$ we select a geodesic $\eta_i$ from $\gamma(t_i)$ to
    $x_i$ (constant for $i=0,N$). We now take the concatenation
    $$\tilde\gamma\defeq \gamma|_{[0,t_0]}*\tilde\gamma_1*\dots*\tilde\gamma_N*\gamma|_{[t_N,1]},$$
    which is a loop in $A$. To conclude, we claim that $\tilde\gamma$ is homotopic to $\gamma$ in $A'$.
    This will yield a contradiction since then $\tilde\gamma$ must reverse the orientation, as well.

    To prove the claim, we just observe that each curve $\eta_i$ has length at most $\delta$,
    and hence the concatenation $\eta_{i-1}*\tilde\gamma_i*\eta_i^{-1}$ has length at most $6\delta$,
    and thus takes values in $B_{10\delta}(x_i)\subseteq B_i$. Since $B_i$ is a contractible subset of $A'$
    and includes the images of both $\eta_{i-1}*\tilde\gamma_i*\eta_i^{-1}$ and $\gamma|_{[t_{i-1},t_i]}$,
    we conclude that these two curves are homotopic in $A'$ (note that they have the same endpoints), and the conclusion follows.
 \end{proof}

\subsection{Proof of Theorem \ref{vefdacscdcsd}}

In this part, we address the remaining implications in the proof of Theorem \ref{vefdacscdcsd}. We will make use of the technical material on volume forms and currents, which is developed in Sections \ref{sectionforms} and \ref{sectioncurrents}, respectively.


We start by addressing the equivalence among (1), (2), and (3).
Recall first that, for $\epsilon<\epsilon(n)$, $A_\epsilon(\XX)$  is open and is a connected topological manifold. Hence,  ${(1)}\Rightarrow{(3)}\Rightarrow{(2)}$.
Proposition \ref{thm:char1} proves ${(1)}\Leftrightarrow{(2)}$.

To deal with items (4), (5) and establish the equivalence with the previous ones, we rely on the results of Section \ref{sectionforms}. 
Theorem \ref{brfvaed} proves $(2)\Rightarrow (4),(5)$ for a volume form $\omega$ satisfying (i) and (ii). Notice that we could have equivalently proved the slightly weaker implications $(3)\Rightarrow (4),(5)$; however, the proof would not have been significantly simpler. 
Theorem \ref{vfedacadc} proves $(4)\Rightarrow (3)$ and $(5)\Rightarrow (3)$. 
Theorem \ref{uniqueness} proves the uniqueness of the orientation form. Notice that, as a consequence of this and the proof of Theorem \ref{brfvaed},  any form satisfying either $(4)$ or $(5)$ satisfies  also (i) and (ii). 

Finally, to deal with item (6) and prove the equivalence with (4) and (5), we exploit the results of Section \ref{sectioncurrents}.
Theorem \ref{maincurr} proves $(5)\Rightarrow (6)$ (notice that, by Proposition \ref{formsandcurrents}, $|T_\omega|\in L^\infty$ and $T_\omega\ne 0$), and Theorem \ref{bgrvfdcassc} proves  $(6)\Rightarrow (5)$ (notice that, by Proposition \ref{formsandcurrents}, $\omega_T\ne 0$). By the uniqueness of the orientation form, it is straightforward to deduce the uniqueness of the orientation current, as the map $T\mapsto \omega_T$ of Proposition \ref{formsandcurrents} is injective.\qed

\section{Ramified Double Cover}

In this section, we examine the ramified double cover within the context of non-collapsed $\RCD$ spaces without boundary. As outlined in Sections \ref{intro:cover} and \ref{intro:Open}, we establish the existence, uniqueness, and defining properties of the ramified double cover. However, $\RCD$ regularity remains open, except in the smoothable case.

\begin{thm}[Ramified Double Cover for $\RCD$ Spaces]\label{defndouble}
	Let $(\XX,d,\HH^n)$ be a non-orientable, non-collapsed $\RCD(-(n-1),n)$ space without boundary. Then there exists a geodesic metric measure space $(\widehat \XX,\widehat d,  \HH^n)$ along with an involutive isometry $\Gamma:\widehat \XX\rightarrow\widehat \XX$ such that the following hold.
	\begin{itemize}
		\item[(i)] $\XX = \widehat \XX /\langle \Gamma \rangle$, and we denote by $\pi:\widehat\XX\rightarrow\XX$ the projection map.

		\item[(ii)] There exists $\widehat A \subseteq \widehat \XX$ open dense, which is an orientable topological manifold and length space.
		
		\item[(iii)] There is an open connected topological manifold $\mathcal{R} \subseteq A \subseteq \XX$, such that  $\pi:\widehat A = \pi^{-1}(A)  \rightarrow A$ is a local isometry and forms a double cover, where $\mathcal{R}$
        is the set of points where the tangent cone is Euclidean.
	\end{itemize}

	The pair $((\widehat \XX,\widehat d,  \HH^n),\pi)$ is unique up to isomorphism. Specifically, if $((\widehat \XX',\widehat d',  \HH^n),\pi')$ is another such pair, then there exists an isometry $\Phi: (\widehat \XX,\widehat d)\rightarrow(\widehat \XX',\widehat d')$ satisfying $\pi'\circ \Phi = \pi$.
\end{thm}

\begin{rem}\label{vfedavdac}
	From the statement above, one can easily deduce the following properties for the involutive isometry $\Gamma:\widehat \XX\rightarrow\widehat \XX$, whose proof is detailed later on:
	\begin{enumerate}
        \item $\pi$ is surjective and, if $x=\pi(\hat x)$,
        then the fiber $\pi^{-1}(x)=\{\hat x,\Gamma\hat x\}$;
        \item $d(\pi(\hat x),\pi(\hat y)) = \min \{\widehat d(\hat x,\hat y), \widehat d(\Gamma\hat x,\hat y)\}$ for every $\hat x, \hat y \in \widehat \XX$;
        \item the map $\pi$ is $1$-Lipschitz, satisfies $\pi\circ\Gamma = \pi$, and pushes forward the measure as $\pi_* \HH_{\widehat \XX}^n = 2\HH_\XX^n$;
		\item the sets $\{\hat x \,:\, \Gamma \hat x \neq \hat x\}\supseteq\hat A$ are open, dense, and have full measure in $\widehat \XX$, and on each of them $\pi$ is a local isometry;        
        \item $\dim_\HH(\widehat \XX \setminus \widehat A) \le n-2$;
        \item $\mathrm{diam}(\widehat \XX) \le 2\mathrm{diam}(\XX)$.
	\end{enumerate}
\end{rem}

As in the smooth setting, the preimage of orientable balls $B_r(p) \subseteq \XX$ through $\pi: \widehat{\XX} \to \XX$ is disconnected. More precisely, we have the following.

\begin{lem}\label{lemstupid}
	Let $(\XX,d,\HH^n)$ be a non-orientable, non-collapsed $\RCD(-(n-1),n)$ space with no boundary, and let $p\in\XX$. Let $\pi:(\widehat \XX,\widehat d,\HH^n)\rightarrow (\XX,d,\HH^n)$ be the orientable double  cover as in Theorem \ref{defndouble}, and let $\hat p \in \widehat \XX$ be such that $\pi(\hat p)=p$. 
	Then, for any $R>0$, $B_R(p)$ is non-orientable if and only if $\widehat d(\hat p,\Gamma\hat p)<2R$.
\end{lem}

We introduce a shorthand notation for the displacement of the involutive isometry $\Gamma$, as it will be central in the proof of Theorem \ref{defndouble} and in the next sections.

\begin{defn}[Displacement of the Involution]   \label{defndelta}
	Let $(\XX,d,\HH^n)$ be a non-orientable, non-collapsed $\RCD(-(n-1),n)$ space with no boundary. Let $\pi:(\widehat \XX,\widehat d,\HH^n)\rightarrow (\XX,d,\HH^n)$ be a ramified double  cover as in Theorem \ref{defndouble}. Then we define $\Delta: \XX \rightarrow \RR$ as
	\begin{equation}
		\Delta x\defeq \widehat d(\hat x,\Gamma \hat x)\quad\text{with $\pi(\hat x)=x$},
	\end{equation}
	where $\Gamma$ is the involution given by Theorem \ref{defndouble}.
\end{defn}
\begin{rem}
	Notice that $\Delta x$ is well defined and also that it depends only on $(\XX,d,\HH^n)$, rather than the orientable double cover $\pi:(\widehat\XX ,\widehat d ,\HH^n)\rightarrow (\XX,d,\HH^n)$. Indeed, let $\pi:(\widehat\XX ,\widehat d ,\HH^n)\rightarrow (\XX,d,\HH^n)$ and $\pi:(\widehat\XX ',\widehat d ',\HH^n)\rightarrow (\XX,d,\HH^n)$ be two orientable double covers as in Theorem \ref{defndouble}. Let $\Phi:(\widehat\XX ,\widehat d )\rightarrow(\widehat\XX ',\widehat d ')$ be the isometry as in Theorem \ref{defndouble}, with $\pi'\circ\Phi=\pi$. For $x\in\XX$ and $\hat x\in\widehat\XX $ with $\pi(\hat x)=x$, we have that $\pi'(\Phi(\hat x))=\pi'(\Phi(\Gamma\hat x))=x$. Hence, if $\Gamma\hat x\ne \hat x$ then $\Gamma'\Phi(\hat x)=\Phi(\Gamma\hat x)$.
    By density of $\hat A$, we deduce that
    $$\Gamma'\circ\Phi=\Phi\circ\Gamma,$$
    and hence
	\begin{equation}
		\Delta x= \widehat d (\hat x,\Gamma\hat x)= \widehat d '(\Phi(\hat x),\Phi(\Gamma\hat x))=\widehat d '(\Phi(\hat x),\Gamma'\Phi(\hat x))=\Delta' x.
	\end{equation}
	
	Also, we remark that $\Delta$ is $2$-Lipschitz. Indeed, letting $x,y\in\XX$, by Theorem \ref{defndouble} we can take $\hat x,\hat y\in\widehat\XX $ with $\pi(\hat x)=x$ and $\pi(\hat y)=y$, and $\widehat d (\Gamma\hat x,\Gamma\hat y)=\widehat d (\hat x,\hat y)=d(x,y)$. Then the conclusion is due to the triangle inequality.\fr
\end{rem}

\subsection{Proof of Theorem \ref{defndouble} and Remark \ref{vfedavdac}}

We first prove uniqueness. Take $(\widehat\XX ,\widehat d ,  \HH^n), (\widehat\XX ',\widehat d ',  \HH^n)$ with $\pi,\pi'$, as in the statement, so that we also have $A,A'$ and $\widehat A,\widehat A'$. Let $\Gamma,\Gamma'$ be the natural involutions for $\pi:\widehat A\rightarrow A$ and $\pi':\widehat A'\rightarrow A'$. Set $A''\defeq A\cap A'$, which is an open set containing $\mathcal{R}$, hence connected and non-orientable. Notice that both $\pi^{-1}(A'')$ and  $(\pi')^{-1}(A'')$ are orientable double covers of $A''$, hence there exists a homeomorphism $f:\pi^{-1}(A'')\rightarrow (\pi')^{-1}(A'')$  such that $\pi'\circ f=\pi$.  Moreover, the involutions for $\pi:\pi^{-1}(A'') \rightarrow A''$ and $\pi':(\pi')^{-1}(A'')\rightarrow A''$ are obtained by restriction of $\Gamma$ and $\Gamma'$, respectively.

	Now take $\hat x\in A''$. Let $x\defeq\pi(\hat x)$, and let $\hat x'\in \widehat A'$ be such that $\pi'(\hat x')=x$. Let $r\in (0,1)$ be such that $\pi$ restricts to an isometry both on $B_r(\hat x)$ and on $B_r(\Gamma\hat x)$, with image $B_r(x)$, and the same for the second pair. 
	Up to decreasing $r$, and possibly exchanging $\hat x'$ with $\Gamma\hat x'$, we can assume that $B_r(x)\subseteq A''$ and $f(B_r(\hat x))=B_r(\hat x')$. Thus, $f$ is a local isometry.

    Since $(\hat \XX,\hat d)$ is a length space, $f$ is 1-Lipschitz
    and, by density of $\pi^{-1}(A'')$ in $\hat\XX$ (which follows for instance from density of $A''$ and the formula relating $\hat d$ and $d$), we can extend it to a 1-Lipschitz map $f:\hat \XX\to\hat\XX'$. Since the same can be done in the reverse direction, $f$ is the desired isometry.
    
	
	Now we provide the construction. Let $A\defeq A_\epsilon(\XX)$ for $\epsilon<\epsilon(n)$ and recall that $A\subseteq\XX$ is open and a non-orientable topological manifold. Hence we can take $\widehat A$ to be the orientable double cover of $A$, which is connected, and has the covering map $\pi:\widehat A\rightarrow A$ and the natural involution $\Gamma$. We endow $\widehat A$ with the unique length metric $\widehat d $ making $\pi$ a local isometry, which is possible because $A$ is locally geodesic. Notice that, by this construction, $\pi$ is $1$-Lipschitz.  By construction, $\Gamma$ is a local isometry, so that  $\widehat d (\Gamma\hat x,\Gamma\hat y)\le \widehat d (\hat x,\hat y)$, for every $\hat x,\hat y\in\widehat A$. Since $\Gamma$ is an involution, it is an isometry. Also, $\Gamma\circ\pi=\pi$ on $\widehat A$.
	
	Finally, we take the metric completion and obtain the space $(\widehat\XX ,\widehat d )$, and we endow it with the Hausdorff measure $\HH^n$.   Notice that we can consider the extensions of $\pi:\widehat A\rightarrow \XX$ and $\Gamma:\widehat A\rightarrow\widehat A$, still denoted by $\pi:\widehat\XX \rightarrow \XX$ and $\Gamma:\widehat\XX \rightarrow\widehat\XX $. We still have $\Gamma\circ\pi=\pi$ and $\Gamma$ remains an involutive isometry (in particular, $\Gamma(\widehat A)=\widehat A$). Now we prove several properties satisfied by this space.
	
	First, notice that 
	\begin{equation}\label{alift}
		\pi(\widehat\XX \setminus \widehat A)\subseteq \XX\setminus A,
	\end{equation} 		 
	as $\pi:\hat A\rightarrow A$ is a local isometry and $(\widehat X,\widehat d )$ is the completion of  $(\widehat A,\widehat d )$.

	Now we prove that
	\begin{equation}\label{lift}
		d(\pi(\hat x),\pi(\hat y))=\min \{d(\hat x,\hat y), d(\hat x,\Gamma\hat y)\}\quad\text{for every $\hat x,\hat y\in\widehat\XX $},
	\end{equation}
    so that in particular $\pi^{-1}(x)=\{\hat x,\Gamma\hat x\}$ if $x=\pi(\hat x)$.
	We can assume $\hat x,\hat y\in\widehat A$. Since $\pi$ is $1$-Lipschitz, $d(\pi(\hat x),\pi(\hat y))$ is bounded by the right-hand side.
	Set then $x\defeq\pi(\hat x), y\defeq \pi(\hat y)$, and take $\epsilon>0$. Take then a curve $\gamma:[0,1]\rightarrow A$ joining $x$ to $y$, with length $l(\gamma)\le d(x,y)+\epsilon$. Exploiting the fact that $\pi$ is a local isometry, we can lift $\gamma$ to a curve $\hat \gamma:[0,1]\rightarrow\widehat A$ joining $\hat x$ with some $\hat z\in\widehat\XX $, such that $\pi\circ\hat \gamma=\gamma$. Hence, $l(\hat \gamma)=l(\gamma)\le d(x,y)+\epsilon$. Since $\pi(\hat z)=y$, we have that either $\hat z=\hat y$ or $\hat z=\Gamma\hat y$ (by  \eqref{alift}).
	Hence the claim, as $\epsilon>0$ was arbitrary.
	
	Now we prove that $(\widehat\XX ,\widehat d )$ is proper (and geodesic), by proving that bounded subsets of $(\widehat\XX ,\widehat d )$ are totally bounded. Take indeed $\widehat B\subseteq\widehat\XX $ bounded and, for $\epsilon\in (0,1)$, take a finite set $(x_i)_{i=1,\dots,N}\subseteq A\cap \pi(\widehat B)$ such that $\bigcup_{i=1,\dots,N} B_{\epsilon}(x_i)\supseteq \pi(\widehat B)$, which is possible as bounded sets in $(\XX,d)$ are precompact. Then we have $\bigcup_{i=1,\dots,N} \bigcup_{\hat x_i\in \pi^{-1}(x_i)} B_{\epsilon}(\hat x_i)\supseteq\widehat B$, by \eqref{lift}.
	
	Finally, we prove that 
	\begin{equation}\label{aalift}
		\pi(\widehat\XX \setminus \widehat A)= \XX\setminus A,
	\end{equation} 		 
	i.e.\ that equality holds in \eqref{alift}.
	Indeed, take $x\in\XX \setminus A$ and $(x_k)_k\subseteq A$, $x_k \rightarrow x$. Take any $\hat x_1$ such that $\pi(\hat x_1)=x_1$ and, for $k\ge 2$, take $\hat x_k$ with $\pi(\hat x_k)=x_k$ and such that $\widehat d (\hat x_k,\hat x_1)=d(x_k,x_1)$, which is possible thanks to \eqref{lift}. Since $(\widehat\XX ,\widehat d )$ is proper, for a non-relabeled subsequence we can find a limit $\hat x_k\rightarrow \hat x$. By continuity, we have that $\pi(\hat x)=x$, which is the conclusion.
	
	Now we prove that 
	\begin{equation}
		\mathrm{diam}(\widehat\XX ,\widehat d )\le 2\mathrm{diam}(\XX,d).
	\end{equation}
	Take any $\hat x,\hat y\in\widehat\XX$, let $\hat m\in\widehat\XX $ be such that $\widehat d (\hat x,\hat m)=\widehat d (\hat m,\hat y)$, and set $m\defeq\pi(\hat m)$. By \eqref{lift}, either $d(x,m)=\widehat d (\hat x,\hat m)$ or $d(x,m)=\widehat d (\hat x,\Gamma\hat m)$. In the former case, $\widehat d (\hat x,\hat y)\le \widehat d (\hat x,\hat m)+\widehat d (\hat m,\hat y)=  2\widehat d (\hat x,\hat m)=2d(x,m)\le 2\mathrm{diam}(X)$. We can conclude in a similar way if $d(m,y)=\widehat d (\hat m,\hat y)$. Hence, the only case left to consider is the one in which $d(x,m)=\widehat d (\hat x,\Gamma\hat m)$ and $d(m,y)=\widehat d (\hat m,\Gamma\hat y)$. However, in this case,
	$\widehat d (\hat x,\hat y)\le \widehat d (\hat x,\Gamma\hat m)+\widehat d (\Gamma\hat m,\hat y)= d(x,m)+d(m,y)\le 2\mathrm{diam}(X)$. 

	Finally, we prove that 
	\begin{equation}
		\dim_\HH(\widehat\XX \setminus\widehat A)\le n-2,
	\end{equation}
	which will clearly imply that $	\pi_*\HH_{\widehat\XX }^n= 2\HH_\XX^n$.
	Indeed, take $\alpha\in (0,1)$, recall \eqref{aalift} and the fact that $\HH^{n-2+\alpha}(\XX\setminus A)=0$, as $\XX\setminus A\subseteq\XX\setminus\mathcal{R}$. Write then $\XX\setminus A\subseteq \bigcup_i B_{r_i}(x_i)$, where $\sum_i r_i^{n-2+\alpha}<\epsilon$, for $\epsilon\in (0,1)$. Then $\widehat\XX \setminus\widehat A\subseteq\bigcup_{i} \bigcup_{\widehat x_i\in \pi^{-1}(x_i)} B_{r_i}(\hat x_i)$, by \eqref{lift}, so that the conclusion follows.
	
	Finally, notice that $\widehat A\subseteq\{\widehat X\,:\,\Gamma \hat x\ne \hat x\}$, so that the latter is dense and, as $\Gamma$ is continuous, it is also open. Take $\hat x$ such that $10r\defeq \widehat d (\hat x,\Gamma\hat x)>0$. We show that $B_{r}(\hat x)$ is mapped isometrically onto $B_r(x)$, for $x\defeq \pi (\hat x)$. Take $\hat y,\hat z\in B_r(\hat  x)$, and set $y\defeq \pi(\hat y),z\defeq \pi(\hat z)\in B_r(x)$. Now we have $ d( y, z)\le 2r$, whereas 
	\begin{equation}
		10r=\widehat d (\hat x,\Gamma\hat x)\le \widehat d  (\hat x,\hat y)+ \widehat d (\hat y, \Gamma \hat z)+\widehat d ( \Gamma \hat z, \Gamma\hat x)\le 2r+ \widehat d (\hat y, \Gamma\hat z),
	\end{equation}
	so by \eqref{lift} we must have $d(y,z)=\widehat d (\hat y,\hat z)$. Also, \eqref{lift} proves that $\pi:B_r(\hat x)\rightarrow B_r(x)$ is surjective.
\qed

\subsection{Proof of Lemma \ref{lemstupid}}

	We keep the notation of Theorem \ref{defndouble} and we assume, for simplicity of notation, that $R=1$.
	Assume first that $B_1(p)$ is non-orientable, so that there exists an orientation-reversing loop $\gamma:[0,1]\rightarrow A\cap B_{1}(p)$ based at some $q\in A$, with $d(q,p)<1/10$ (recall Remark \ref{rmk.local.ori}).
	Take $\hat q\in\widehat\XX $ such that $\pi(\hat q)=q$ and $\widehat d (\hat p,\hat q)<1/10$.
	Hence, by lifting $\gamma$ through the local isometry $\pi$, we obtain a curve $\hat\gamma:[0,1]\rightarrow \widehat A$ joining $\hat q$ to $\Gamma\hat q$, with $\hat\gamma([0,1])\subseteq\pi^{-1}( B_{1}(p))$. 
	If $\widehat d (\hat q,\Gamma\hat p)\le 2/10$, then
	\begin{equation}
		\widehat d (\hat p,\Gamma\hat p)\le  \widehat d (\hat p,\hat q)+\widehat d (\hat q,\Gamma\hat p)<3/10<2,
	\end{equation}
	so that we can assume $\widehat d (\hat q,\Gamma\hat p)\ge 2/10$.
	Hence, by considering the continuous function $[0,1]\ni t\mapsto \widehat d (\hat \gamma(t),\hat p)-\widehat d (\hat \gamma(t),\Gamma\hat p)\in\RR$, we see that 
	there exists $\hat m$ in the image of $\hat\gamma$ which satisfies
	\begin{equation}
		\widehat d (\hat m,\hat p)=\widehat d (\hat m,\Gamma\hat p).
	\end{equation}
	Indeed,
	\begin{equation}
		\widehat d (\hat \gamma(0),\hat p)-\widehat d (\hat\gamma(0),\Gamma \hat p)= \widehat d (\hat q,\hat p)-\widehat d (\hat q,\Gamma\hat p) \le 1/10-2/10<0,
	\end{equation}
	and also we have $\widehat d (\hat \gamma(1),\hat p)-\widehat d (\hat\gamma(1),\Gamma \hat p)=\widehat d (\Gamma\hat q,\hat p)-\widehat d (\Gamma\hat q,\Gamma \hat p)=-(\widehat d (\hat q, \hat p)-\widehat d (\hat q,\Gamma\hat p))>0$.
	Since, $m\defeq\pi(\hat m)\in B_{1}(p)$, we have either $\widehat d  (\hat m,\hat p)< 1$ or $\widehat d  (\hat m,\Gamma\hat p)<1$ (hence both).
	Hence, in any case, 
	\begin{equation}
		\widehat d (\hat p,\Gamma\hat p)\le\widehat d (\hat p,\hat m)+\widehat d (\hat m,\Gamma\hat p)<2.
	\end{equation}
	
	Conversely, assume that $\widehat d (\hat p,\Gamma\hat p)<2$ and let $\eta\in (0,1)$. Let $\hat q\in\widehat A\cap B_{\eta}(\hat p)$, so that $ \widehat d (\hat q,\Gamma\hat q)\le \widehat d (\hat p,\Gamma\hat p)+2\eta$. Since $\widehat A$ is a length space, we can find a curve $\hat \gamma:[0,1]\rightarrow\widehat A$ joining $\hat q$ to $\Gamma\hat q$ with $l(\hat \gamma)\le \widehat d (\hat p,\Gamma\hat p)+3\eta$. If we set $\gamma\defeq\pi\circ\hat\gamma$, then $\gamma$ is an orientation-reversing loop based at $q=\pi(\hat q)$ of length $l(\gamma)\le \widehat d (\hat p,\Gamma\hat p)+ 3\eta$. In particular, the image of $\gamma$ is contained in $B_{l(\gamma)/2}(q)\subseteq B_{l(\gamma)/2+\eta}(p)$. If $\eta$ is small enough, the image of $\gamma$ is contained in $B_{1}(p)$, which means that $B_1(p)$ is non-orientable.
\qed

\section{Stability of Orientability}

In this section, we discuss in more detail and prove the stability results outlined in Section \ref{intro:stability}.
For the reader's convenience, we restate below Theorem \ref{vfeadadcsd} regarding the stability of orientability according to Definition \ref{dfn:orientable} under GH-convergence.

\begin{thm}[Stability of Orientability]\label{orientstable}
	Let $(\XX_k,d_k,\HH^n,p_k)\xrightarrow{GH}(\XX,d,\HH^n,p)$ be a sequence of non-collapsed $\RCD(-(n-1),n)$ spaces with no boundary.
	If $(\XX_k,d_k,\HH^n)$ is orientable for every $k$, then  $(\XX,d,\HH^n)$ is orientable. More specifically, if, for some $R>0$, $B_R(p_k)$ is orientable for every $k$, then $B_R(p)$ is orientable.
\end{thm}

We now state the most general version of the stability of non-orientable spaces with Ricci bounded below. Theorem \ref{thm:stabnonorRicciLimit} is an immediate corollary.

\begin{thm}[Stability of Non-Orientability]\label{nonstability}
Let $(\XX_k,d_k,\HH^n,p_k)\xrightarrow{GH}(\XX,d,\HH^n,p)$ be a sequence of non-collapsed $\RCD(-(n-1),n)$ spaces with no boundary. 
Assume that, for some $R>0$, $B_R(p_k)$ is not orientable, and the ramified double cover $(\widehat \XX_k, \widehat d_k, \HH^n)$ (as in Theorem \ref{defndouble}) is an $\RCD(-(n-1), n)$ space, for every $k$. 
Then $B_R(p)$ is non-orientable. Further, we have $(\widehat \XX_k, \widehat d_k, \HH^n, \hat p_k) \xrightarrow{GH} (\widehat \XX, \widehat d, \HH^n, \hat p)$, where $\pi: (\widehat \XX, \widehat d, \HH^n) \rightarrow (\XX, d, \HH^n)$ is the orientable ramified double cover as in Theorem \ref{defndouble} and $\pi(\widehat p) = p$.
Finally (recall Definition \ref{defndelta}), $\Delta_k\rightarrow \Delta$ locally uniformly.
\end{thm}


\subsection{Proof of Theorem \ref{orientstable}}
We argue by contradiction and thus we assume that $B_R(p)$ is non-orientable. To simplify the notation, we assume that $R=1$, although this will play no difference in the proof. Therefore, by Remark \ref{rmk.local.ori} we can find an orientation-reversing loop $\gamma:[0,1]\rightarrow A_\epsilon(\XX)\cap B_1(p)$.

To conclude, we propose two distinct proofs: one shorter, which utilizes the global Reifenberg theorem for regular sets in $\RCD$ spaces, and a second one which is completely elementary.

\begin{proof}[Proof I]
By standard $\epsilon$-regularity, if $\epsilon \le \epsilon(n)$ there exists a finite covering $\gamma([0,1]) \subseteq \bigcup_{i=1}^m B_{r_i}(x_i)$ where $B_{100 r_i}(x_i)$ are $\epsilon(n)$-regular balls, bi-H\"older to $B_1(0^n) \subseteq \RR^n$. For $k$ large enough, we can find  $2\epsilon(n)$-regular balls $B_{100r_i}(x_i^k) \subseteq X_k$ which are $\epsilon(n)$-close to $B_{100r_i}(x_i)$. Hence $\bigcup_{i=1}^m B_{100r_i}(x_i^k)$ is a topological manifold, and by standard gluing (see \cite[Appendix I]{Cheeger-Colding96} and \cite[Theorem 4.6]{Kapovitch07}) a subset of it is homeomorphic to $\bigcup_{i=1}^m B_{r_i}(x_i)$, which is not orientable, a contradiction.
\end{proof}

\begin{proof}[Proof II]
Let us embed isometrically $B_2(p_k)\subseteq \XX_k$ and $B_2(p)\subseteq \XX$
into a common metric space $(\tilde \XX,\tilde d)$
realizing the GH-convergence $B_2(p_k)\to B_2(p)$. Thus, from now on,
we will view $B_2(p_k),B_2(p)\subseteq\tilde\XX$.

By taking points $\{q_{k,j}\mid j=0,\dots,2^k\}\subseteq X_k$ with $q_{k,0}=q_{k,2^k}$ and $\tilde d(q_{k,j},\gamma(2^{-k}j))\le\epsilon_k\to0$,
and joining them with $2^k$ geodesics in $\XX_k$ (parametrized over intervals of length $2^{-k}$),
we can find loops $\gamma_k:[0,1]\to \XX_k$ converging uniformly to $\gamma$.
We claim that $\gamma_k$ is also orientation-reversing for $k$ large enough, which is the desired contradiction.

Since $\gamma$ takes values in $A_\epsilon(\XX)$, by Reifenberg we can find $N\ge1$ and $\rho>0$ such that
$$\gamma\lft(\lft[\frac{\ell-1}{N},\frac{\ell+1}{N}\rgt]\rgt)\subseteq B_\rho\lft(\gamma\lft(\frac{\ell}{N}\rgt)\rgt)$$
for all $\ell=0,\dots,N$ (where intervals are taken in $\mathbb R/\mathbb Z$),
and moreover there exists a map
$$h_\ell:B_{10\rho}(0^n)\to B_{10\rho}\lft(\gamma\lft(\frac{\ell}{N}\rgt)\rgt)\subseteq\XX$$
which is a homeomorphism with its image $h_\ell(B_{10\rho}(0^n))\supseteq B_{9\rho}(\gamma(\frac{\ell}{N}))$; we can assume that $h_N=h_0$. Similarly, for $k$ large enough, we can find
$$h_{k,\ell}:B_{10\rho}(0^n)\to B_{10\rho}\lft(\gamma_k\lft(\frac{\ell}{N}\rgt)\rgt)\subseteq\XX_k$$
which is a homeomorphism with its image $h_{k,\ell}(B_{10\rho}(0^n))\supseteq B_{9\rho}(\gamma_k(\frac{\ell}{N}))$; since $h_{k,\ell}$ can be constructed by perturbing any given GH-approximation
$B_{10\rho}(0^n)\to B_{10\rho}(\gamma_k(\frac{\ell}{N}))$, and since $B_2(p_k)\to B_2(p)$ in the Hausdorff metric for $(\tilde\XX,\tilde d)$, we can also require that
$$\sup_{x\in B_{10\rho}(0^n)}\tilde d(h_\ell(x),h_{k,\ell}(x))\le\delta_k$$
for a vanishing sequence $\delta_k\to0$. We will also assume that
$$h_\ell(0)=\gamma\lft(\frac{\ell}{N}\rgt),\quad h_{k,\ell}(0)=\gamma_k\lft(\frac{\ell}{N}\rgt),$$
and that both $h_\ell,h_{k,\ell}$ and their inverses increase distances
at most by the additive error $\rho$.

We now fix an initial generator $\mu_0\in H_n(\XX,\XX\setminus\{\gamma(0)\})^*\cong\mathbb Z^*=\{\pm1\}$
and let $\mu_t\in H_n(\XX,\XX\setminus\{\gamma(t)\})^*$ be the (unique) continuous extension along the curves $\gamma$ and $\gamma_k$,
as discussed at the beginning of Section \ref{ori.equiv.sec}.
We endow the topological ball $B^{(\ell)}\defeq h_\ell(B_{10\rho}(0^n))$
with the unique orientation $\mu^{(\ell)}$ such that
$$\mu^{(\ell)}_{\gamma(t)}=\mu_t\quad\text{for }t\in\lft[\frac{\ell-1}{N},\frac{\ell+1}{N}\rgt]\cap[0,1].$$
Note that, even if $B^{(0)}=B^{(N)}$, we have opposite orientations $\mu^{(0)}=-\mu^{(N)}$
as $\mu_0=-\mu_1$.

Next, we endow the topological ball $B^{(k,\ell)}\defeq h_{k,\ell}(B_{10\rho}(0^n))$
with the image orientation
$$\mu^{(k,\ell)}\defeq (h_{k,\ell}\circ h_\ell^{-1})_*\mu^{(\ell)}.$$
For $\ell=0,\dots,N-1$ and $k$ large enough, we now claim that
$$\mu^{(k,\ell)}_{\gamma_k(t)}=\mu^{(k,\ell+1)}_{\gamma_k(t)}\quad\text{for }t\in\lft[\frac{\ell}{N},\frac{\ell+1}{N}\rgt].$$
Once this is done, the definition
$$\mu_{k,t}\defeq \mu^{(k,\ell)}_{\gamma_k(t)}\quad\text{for }t\in\lft[\frac{\ell}{N},\frac{\ell+1}{N}\rgt]$$
yields a continuous choice of a generator $\mu_{k,t}\in H_n(\XX_k,\XX_k\setminus\{\gamma_k(t)\})$
and, since $\mu^{(k,0)}=-\mu^{(k,N)}$, we have $\mu_{k,0}=-\mu_{k,N}$.
This says that $\gamma_k$ is orientation-reversing, as desired.

To check the last claim, we let $U^{(\ell)}$ denote the connected component
of $B^{(\ell)}\cap B^{(\ell+1)}$ including the curve $\gamma([\frac{\ell}{N},\frac{\ell+1}{N}])$. By definition of $\mu^{(\ell)}$ we have
$$\mu^{(\ell)}=\mu^{(\ell+1)}\quad\text{on }U^{(\ell)}.$$
Hence, given $q\in \gamma_k([\frac{\ell}{N},\frac{\ell+1}{N}])$, letting
$$x'\defeq h_{k,\ell}^{-1}(q),\quad x''\defeq h_{k,\ell+1}^{-1}(q),\quad q'\defeq h_\ell(x'),\quad q''\defeq h_{\ell+1}(x''),$$
by the closeness of the maps $h_\ell$ and $h_{k,\ell}$ we have
$$d\lft(q',\gamma\lft(\frac{\ell}{N}\rgt)\rgt)
\le \tilde d(q',q)+d_k\lft(q,\gamma_k\lft(\frac{\ell}{N}\rgt)\rgt)+\tilde d\lft(\gamma_k\lft(\frac{\ell}{N}\rgt),\gamma\lft(\frac{\ell}{N}\rgt)\rgt)
\le \delta_k+\rho+\delta_k'
\le 2\rho$$
for $k$ large enough, where $\delta_k'\to0$ is another vanishing sequence.
Similarly, we have
$$d\lft(q'',\gamma\lft(\frac{\ell+1}{N}\rgt)\rgt)\le2\rho.$$
Thus,
$$q',q''\in B_{3\rho}\lft(\gamma\lft(\frac{\ell+1/2}{N}\rgt)\rgt)\subseteq\XX.$$
This ball is connected and included in $B_{9\rho}(\gamma(\frac{\ell}{N}))\cap B_{9\rho}(\gamma(\frac{\ell+1}{N}))$, and thus it is included in $U^{(\ell)}$. We deduce that
$$q',q''\in U^{(\ell)}.$$
Hence, defining the two orientations
$$\mu'\defeq (h_\ell^{-1})_*\mu^{(\ell)},\quad\mu''\defeq (h_{\ell+1}^{-1})_*\mu^{(\ell+1)}$$
on the Euclidean ball $B_{10\rho}(0^n)$ and setting $A\defeq h_\ell^{-1}(U^{(\ell)})$
(so that $A$ and $h_{\ell+1}^{-1}\circ h_\ell(A)$ are open subsets of $B_{10\rho}(0^n)$),
we see that $A$ contains $x',x''$ and the homeomorphism
$$h_{\ell+1}^{-1}\circ h_\ell\Big|_A:A\to h_{\ell+1}^{-1}\circ h_\ell(A)$$
maps the orientation $\mu'$ to $\mu''$. To conclude, it suffices to check that
$$(h_{k,\ell+1}^{-1}\circ h_{k,\ell})_*\mu'=\mu''$$
as well; note that $x',x''\in B_{6\rho}(0^n)$, whose closure is included in the domains of both
$h_{\ell+1}^{-1}\circ h_\ell$ and $h_{k,\ell+1}^{-1}\circ h_{k,\ell}$.
These two homeomorphisms map $B_{6\rho}(0^n)$ to a superset of $B_{3\rho}(0^n)$
(for instance, we have $h_\ell(B_{6\rho}(0^n))\supseteq B_{5\rho}(\gamma(\frac{\ell}{N}))\supseteq B_{4\rho}(\gamma(\frac{\ell+1}{N}))$, and the image of the latter through $h_{\ell+1}^{-1}$ includes $B_{3\rho}(0^n)$),
and hence the previous assertion follows from the fact that
$$|h_{\ell+1}^{-1}\circ h_\ell(x)-h_{k,\ell+1}^{-1}\circ h_{k,\ell}(x)|\le 2\rho$$
for all $x\in B_{6\rho}(0^n)$, which guarantees that
$$[(h_{k,\ell+1}^{-1}\circ h_{k,\ell})_*\mu']_0=[(h_{\ell+1}^{-1}\circ h_{\ell})_*\mu']_0$$
(as the two homeomorphisms induce the same map $H_n(\bar B_{6\rho}(0^n),\partial B_{6\rho}(0^n))\to H_n(\mathbb R^n,\mathbb R^n\setminus\{0\})$).
Note that the last bound holds since
$$\tilde d(h_{k,\ell+1}\circ h_{\ell+1}^{-1}(y),h_{k,\ell}\circ h_{\ell}^{-1}(y))
\le\tilde d(h_{k,\ell+1}\circ h_{\ell+1}^{-1}(y),y)+\tilde d(y,h_{k,\ell}\circ h_{\ell}^{-1}(y))
\le2\delta_k\le\rho,$$
where we let $y\defeq h_\ell(x)$, and since $h_{k,\ell+1}^{-1}$ increases distances by at most $\rho$.
\end{proof}

\subsection{Proof of Theorem \ref{nonstability}}
To simplify the notation, we set $R=1$. This will play no real difference in the proof.
We will use the conclusions (and the notation) of Theorem \ref{defndouble} freely.
	We start by recalling the conclusion of Lemma \ref{lemstupid}, i.e.
	\begin{equation}\label{bgrfdvs}
		\widehat d _k(\hat p_k,\Gamma_k\hat p_k)<  2\quad\text{for every $k$.}
	\end{equation}
	
	Note that, since $(\XX,d,\HH^n)$ is non-collapsed, using \cite[Theorem 1.3]{DPG17}, we have $\HH^n(B_1(\hat p))>0$, and hence $\inf_k \HH^n(B_1(\hat p_k))>0$. Therefore, up to subsequences, 
	$(\widehat\XX _k,\widehat d _k,\HH^n,\hat p_k)\rightarrow (\widehat\XX ,\widehat d ,\HH^n,\hat p)$, for some pointed non-collapsed $\RCD(-(n-1) ,n)$ space with no boundary (see \cite[Theorem 1.6]{brueboundary}). 
	We are going to prove next that $(\XX,d,\HH^n)$ is non-orientable and $(\widehat\XX ,\widehat d ,\HH^n)$ is indeed the orientable double cover as in Theorem \ref{defndouble}, so that the full sequence converges without the need of extracting subsequences.
	
	Notice that the maps $\Gamma_k: \widehat X_k\rightarrow\widehat X_k$ and $\pi_k:\widehat X_k\rightarrow X_k$ are $1$-Lipschitz, so that, up to taking a subsequence, we have limit maps $\Gamma:\widehat\XX \rightarrow\widehat\XX $ and $\pi:\widehat\XX \rightarrow \XX$. Of course,  $\pi$ is $1$-Lipschitz, $\pi\circ \Gamma=\pi$, and $\Gamma$ is an involutive isometry. 
	Now the key observation is the following: if we have   $\widehat X_k\ni\hat x_k\rightarrow \hat x\in\widehat\XX $, then 
	\begin{equation}
		d_k(\pi_k(\hat x_k),p_k)=\min\{\widehat d _k(\hat x_k,\hat p_k), \widehat d _k(\hat x_k,\Gamma_k\hat p_k)\}\le \widehat d _k(\hat x_k,\hat p_k)+ \widehat d _k(\hat p_k,\Gamma_k\hat p_k)\le  \widehat d _k(\hat x_k,\hat p_k)+2
	\end{equation}
	is uniformly bounded, as well as 
	\begin{equation}
		\widehat d _k(\Gamma_k\hat x_k,\hat p_k)\le 	\widehat d _k(\Gamma_k\hat x_k,\Gamma_k\hat p_k)+ 	\widehat d _k(\Gamma_k\hat p_k,\hat p_k)\le \widehat d _k(\hat x_k,\hat p_k)+2.
	\end{equation}
	Also, we obtain
	\begin{equation}\label{vfeadsaaaaa}
		d(\pi(\hat x),\pi(\hat y))=\min\{\widehat d (\hat x,\hat y),\widehat d (\hat x,\Gamma\hat y)\}\quad\text{for every }\hat x,\hat y\in\widehat\XX .
	\end{equation}
	In particular, we have
	\begin{equation}\label{vfeadsaaaa}
		\pi^{-1}(\pi(\hat x))=\{\hat x,\Gamma\hat x\}\quad\text{for every }\hat x\in\widehat\XX .
	\end{equation}

    
	Fix $\epsilon>0$ small and let $\delta\in (0,1)$ given by Theorem \ref{prel2} below (with this choice of $\epsilon$), thus depending only upon $n$, and set $A\defeq  A_{\delta/4}(X)$.
    In particular, $A\subseteq\XX$ is dense and $\dim_\HH(\XX\setminus A)\le n-2$.
	As in the proof of Theorem \ref{defndouble}, the combination of \eqref{vfeadsaaaaa} and \eqref{vfeadsaaaa} implies that 
	\begin{equation}\label{hausa}
		\dim_\HH(\widehat\XX \setminus\widehat A)\le n-2,
	\end{equation}
	where $\widehat A\defeq\pi^{-1}(A)$ is open. 
	
	Now we claim that
	\begin{equation}\label{vrfdcaasss}
		\Gamma\hat x\ne \hat x\qquad\text{for every }\hat x\in \widehat A.
	\end{equation}
	For $\hat x\in\widehat A$, consider $x\defeq\pi(\hat x)$ and take $(x_k)_k$,  where $\XX_k\ni x_k\rightarrow x$. By Theorem \ref{prel2} in scale-invariant form, we see that, for $k$ large enough, for a suitable $r_x<\delta/4$ the ball $B_{r_{x}}(x_k)$ is homeomorphic to an open set of $\RR^n$, hence is orientable, and this forces
	$\widehat d _k (\hat x_k, \Gamma_k \hat x_k)\ge2  r_{x}$ for every $k$ large enough, by Lemma \ref{lemstupid}, with the obvious meaning for $\hat x_k$. 
	As $k\rightarrow\infty$, we obtain \eqref{vrfdcaasss}.

	By Lemma \ref{moltoconnesso} with \eqref{hausa}, $\widehat A$ is a length space. In particular, we can take a  curve of finite length $\hat \gamma:[0,1]\rightarrow\widehat A$ joining $\hat q$ to $\Gamma\hat q$, where $\hat q$ is any point in $\widehat A$.	
	We set $\gamma\defeq\pi\circ\hat\gamma:[0,1]\rightarrow A$, which is a loop based at $q\defeq\pi(\hat q)$ of finite length.  Arguing as in the first proof of Theorem \ref{orientstable}, for $k$ large enough,  a neighborhood of $\gamma$ (independent of $k$) is homeomorphic to an open subset of $\XX_k$, say $B_k$, which is also a topological manifold. 
	Take  $\widehat\XX _k\ni \hat q_k\rightarrow \hat q$ and $\hat \gamma_k:[0,1]\rightarrow \XX_k$ joining $\hat q_k$ to $\Gamma_k\hat q_k$, uniformly converging to $\hat \gamma$, and set $\gamma_k\defeq\pi_k\circ\hat\gamma_k$, uniformly converging to $\gamma$. Notice that eventually the image of $\gamma_k$ will be contained in $B_k$, so that $\gamma_k$ are orientation-reversing loops based at $q_k\defeq\pi_k(\hat q_k)$. We thus see that $\gamma$ is orientation-reversing. Hence, $(\XX,d,\HH^n)$ is a non-orientable, non-collapsed $\RCD(-(n-1) ,n)$ space with no boundary. Note that we could have concluded also by arguing as in the second proof of Theorem \ref{orientstable}.
	
	We now prove that $(\widehat\XX ,\widehat d ,\HH^n)$ coincides with the orientable double cover given by Theorem \ref{defndouble}. First, combining \eqref{vrfdcaasss}, \eqref{vfeadsaaaaa} and \eqref{vfeadsaaaa}, as in the proof of Theorem \ref{defndouble}, we have that $\pi:\widehat A\rightarrow A$ is a local isometry. From this and Theorem \ref{prel2} in scale-invariant form, it follows that $\widehat A$ is a topological manifold, and  by Theorem \ref{vfeadadcsd} 
    $\widehat A$ is orientable. We have already proved that  $\widehat A$ is a length space.   
	Thus, the conclusion follows from the  uniqueness part of 	Theorem \ref{defndouble}.

	Now we prove that $\Delta_k\rightarrow \Delta$. Recall that these maps are uniformly Lipschitz
     and it is enough to show that if $\XX_k\ni q_k\rightarrow q\in\XX$, then $\Delta_k q_k\rightarrow\Delta q$. Fix $\hat q\in\widehat\XX $ with $\pi(\hat q)=q$, and take also a sequence $(\hat q_k)_k$ with $\pi_k(\hat q_k)=q_k$. Now
	\begin{equation}
		\widehat d _k (\hat q_k,\hat p_k)\le d(q_k,p)+ \widehat d _k (\hat p_k,\Gamma_k\hat p_k)
	\end{equation}
	is uniformly  bounded by \eqref{bgrfdvs}, so that any subsequence of $(\hat q_k)_k$ has limit points in $\widehat\XX $.
	Also, by the convergence of $\pi_k$ and \eqref{vfeadsaaaa}, any such limit point is either $\hat q$ or $\Gamma\hat q$. Hence, up to exchanging some $\hat q_k$ with $\Gamma_k\hat q_k$, we can assume that $\hat q_k\rightarrow\hat q$. Hence, also $\Gamma_k\hat q_k\rightarrow\Gamma\hat q$, so that $\widehat d _k(\hat q_k,\Gamma_k\hat q_k)\rightarrow \widehat d (\hat q,\Gamma\hat q) $, which is the conclusion.\qed

\subsection{Proof of Theorem \ref{defndoubleIntro}}

In view of Theorem \ref{orientstable}, $(M^n_k, g_k, p_k)$ is not orientable for  $k$ sufficiently large. Hence, by Theorem \ref{nonstability}, we have $(\widehat M_k^n, \widehat g_k, \widehat p_k) \xrightarrow{GH}(\widehat \XX,\widehat d,\HH^n, \widehat p)$ where $\pi: (\widehat \XX, \widehat d, \HH^n) \rightarrow (\XX, d, \HH^n)$ is the orientable ramified double cover as in Theorem \ref{defndouble} and $\pi(\widehat p) = p$. By stability of the $\RCD$ condition, we obtain that $(\widehat \XX, \widehat d, \HH^n)$ is an $\RCD(-(n-1),n)$ space. Properties (i), (ii), and the uniqueness of the ramified double cover follow from Theorem \ref{defndouble}.
\qed

\subsection{Proof of Theorem \ref{thm:localnonorientable}}

If $x$ is locally non-orientable, then every tangent cone at $x$ is non-orientable, by Theorem \ref{nonstability}. Otherwise, every tangent cone at $x$ is orientable, by Theorem \ref{orientstable}. Hence, recalling Example \ref{ex:cones}, we have established the equivalence between (1) and (3).
 
The equivalence between (1) and (2) follows from Lemma \ref{lemstupid}.
	
We finally prove the volume estimate on $X_{NLO}$. Let $x \in \XX$ be a locally non-orientable point. We show that
	\begin{equation}\label{zz}
		d_{GH}(B_r(x), B_r((0^{n-4},z)) \ge \epsilon(n) r,
		\quad \text{for every $r\in (0,1)$},
	\end{equation}
    where $(0^{n-4},z) \in \RR^{n-4} \times C(Z^3)$ is a tip point, and 
    where $Z^3$ is a non-collapsed $\RCD(2,3)$ space. This together with \cite{CJNrect} would imply the volume estimate as  $x\in S^{n-5}_{\epsilon(n)}$.
    
    To prove \eqref{zz} we argue by contradiction. A diagonal and scaling argument provides a sequence of smooth uniformly non-collapsing manifolds $(M^n_k,g_k,p_k) \xrightarrow{GH} (\RR^{n-4} \times C(Z^3), d, p)$ such that $B_1(p_k)$ is not orientable. By \cite{brupisem}, $Z^3$ is a topological manifold covered by $S^3$, and hence it is orientable (see Example \ref{ex:RCD(2,3)or}). This provides a contradiction as a consequence of our stability result, Theorem \ref{thm:stabnonorRicciLimit}.\qed

\section{Volume Form on $\RCD$ Spaces}\label{sectionforms}

		In this section, we prove the main results relating the orientability of $\RCD$ spaces in the sense of Definition \ref{dfn:orientable}. As a first result, we show that the existence of an open subset which is an orientable manifold, with sufficiently small complement, induces a volume form.

			\begin{thm}[Orientability vs Volume Form I]
			\label{brfvaed}
			Let $(\XX,d,\HH^n)$ be a non-collapsed $\RCD(-(n-1),n)$ space without boundary. 
			Assume that there exists $A\subseteq\XX$ open which is an orientable topological manifold, such that $\HH^{n-1}(\XX\setminus A)=0$. 
			Then there exists $\omega_X\in L^\infty(\Lambda^n T^{*}\XX)$ with $|\omega_X|=1\ \HH^n$-a.e.\ such that the following hold.
			\begin{enumerate}[label=(\arabic*)]
				\item For every $\eta\in\mathrm{TestForms}_n(\XX)$, we have $\omega_X\,\cdot\,\eta\in W^{1,2}(\XX)$, with 
				\begin{equation}\label{vvvfsds}
					\nabla (\omega_X\,\cdot\,\eta)=\omega_X\,\cdot\,\nabla \eta\quad\HH^n\text{-a.e.}
				\end{equation}

				\item 	
				For every $\eta\in\mathrm{TestForms}_{n-1}(\XX)$ with compact support,
				\begin{equation}\label{acdsac}
					\int_\XX \omega_X\,\cdot\,d\eta\,\dd\HH^n=0.
				\end{equation}
			\end{enumerate}
		\end{thm}

	   Our second result shows the converse implication: the existence of a volume form implies the existence of an open subset which is an orientable manifold, with sufficiently small complement.

			\begin{thm}[Orientability vs Volume Form II]
			\label{vfedacadc}
			Let $(\XX,d,\HH^n)$ be a non-collapsed $\RCD(-(n-1),n)$ space without boundary.  Assume that there exists $\omega\in L^\infty (\Lambda^n T^* \XX)$ 
			such that one of the following holds:
			\begin{enumerate}[label=(\arabic*)]
				\item  $|\omega|=1\ \HH^n$-a.e.\ and, for every  $\eta\in\mathrm{TestForms}_{n}(\XX)$ with compact support, $\omega\,\cdot\,\eta\in W^{1,2}(\XX)$;
				\item $\omega\ne 0$ and, for every  $\eta\in\mathrm{TestForms}_{n-1}(\XX)$ with compact support,
				\begin{equation}
					\int_\XX \omega\,\cdot\,d\eta\,\dd\HH^n=0.
				\end{equation}
			\end{enumerate}
		     Then, for $\epsilon<\epsilon(n)$, $A_\epsilon(\XX)$ is orientable.
		\end{thm}
            
		  \begin{rem}\label{remmodulo}
      Actually, we are going to prove the claim of Theorem \ref{vfedacadc} with ${(1)}$ possibly replaced by the slightly weaker
			\begin{enumerate}
   \item [(1')] There exists $c>0$ such that $|\omega|\ge c\ \HH^n$-a.e.\ and, for every  $\eta\in\mathrm{TestForms}_{n}(\XX)$ with compact support, $\omega\,\cdot\,\eta\in W^{1,2}(\XX)$.
			\end{enumerate}

  \end{rem}
		
		Finally, we show that the volume form is unique up to changing the sign.

		\begin{thm}[Uniqueness of Volume Form]
			\label{uniqueness}
			Let $(\XX,d,\HH^n)$ be a non-collapsed $\RCD(-(n-1),n)$ space without boundary.  Assume that $\omega_1,\omega_2\in L^\infty(\Lambda^n T^{*}\XX)$ satisfy the assumptions of Theorem \ref{vfedacadc}.
			Then there exists a constant $c\in \RR\setminus\{0\}$ such that $\omega_1=c\omega_2$.
		\end{thm}
		
		After having recalled some preliminary results, the remaining part of this section will be dedicated to the proof of the theorems stated above. We begin in Section \ref{volumeregularball}, where we study volume forms locally on $\delta$-regular balls. In Section \ref{sectcompat}, we begin gluing local forms by checking the compatibility of local volume forms on the intersection of $\delta$-regular balls.

	\subsection{Preliminaries}\label{sec:Preliminaries}

     In this section, we collect technical results on Sobolev functions and $\delta$-splitting maps in $\RCD$ spaces. We assume the reader to be familiar with the standard terminology of the $\RCD$ theory, referring the reader to \cite{Gigli14} or \cite{GP19} and references therein.
\subsubsection{Sobolev calculus}
	\begin{lem}\label{divconst}
		Let $(\XX,d,\mass)$ be an $\RCD(K,N)$ space. Let $B\subseteq\XX$ be open and connected, and assume that $f\in L^{2}(B)$ is such that  
		\begin{equation}\label{brvfadcas}
			\int_\XX f \mathrm{div}Z\,\dd\mass=0\quad\text{for every $Z\in \mathrm{TestV}(\XX)$ with $\supp(Z)\Subset B$}.
		\end{equation}
		Then $f$ is constant on $B$.
	\end{lem}
	
	\begin{proof}
		By \eqref{brvfadcas}, we see that we conclude if we can show $f\in W^{1,2}(B')$, for every $B'\Subset B$ ball. 
		Now, take  any $\psi\in \mathrm{TestF}(\XX)$ with $\supp\psi\Subset B$ and such that $\psi=1\ \mass$-a.e.\ on $B'$ (see \cite[Lemma 3.1]{Mondino-Naber14}). Then we know that
		\begin{equation}
			\int_\XX f\psi \mathrm{div}Z\,\dd\mass=-\int f\nabla \psi\,\cdot\,Z\,\dd\mass\quad\text{for every $Z\in \mathrm{TestV}(\XX)$}.
		\end{equation}
		We plug in $Z\defeq \nabla h_{2t} (f\psi)$ and we obtain, through an easy approximation argument, that
		\begin{equation}
			\int_\XX  f\psi\mathrm{div}(\nabla h_{2t}(f\psi))\,\dd\mass=-\int f\nabla \psi\,\cdot\,\nabla h_{2t}(f\psi)\,\dd\mass,
		\end{equation}
		which, by the properties of $h_{H,t}$ recalled in \cite[Section 1.4]{bru2019rectifiability} (in particular,  \cite[(1.26), Lemma 1.37 and Proposition 1.38]{bru2019rectifiability}), implies
		\begin{equation}
			-\int_\XX |\nabla h_t(f\psi)|^2\,\dd\mass=-\int_\XX h_{H,t}(f\nabla \psi)\,\cdot\,\nabla h_t(f\psi)\,\dd\mass.
		\end{equation}
		Therefore,
		\begin{equation}
			\begin{split}
				\int_\XX |\nabla h_t(f\psi)|^2\,\dd\mass&\le\bigg( \int_\XX |h_{H,t}(f\nabla \psi)|^2\,\dd\mass\bigg)^{1/2} \bigg(\int_\XX |\nabla h_t (f\psi)|^2\,\dd\mass\bigg)^{1/2}\\&\le e^{-2K t}\bigg( \int_\XX |f\nabla \psi|^2\,\dd\mass\bigg)^{1/2} \bigg(\int_\XX |\nabla h_t (f\psi)|^2\,\dd\mass\bigg)^{1/2},
			\end{split}
		\end{equation}
		where we used Holder's inequality and  \cite[Proposition 3.6.10]{Gigli14}.
		If follows that $f\psi\in W^{1,2}(\XX)$, whence the conclusion.
	\end{proof}

		\begin{lem}\label{sobolevext}
		Let $(\XX,d,\HH^n)$ be a non-collapsed $\RCD(-(n-1),n)$ space and let $f\in L^\infty\cap L^2(\XX)$,  $g\in L^2(\XX)$. Assume that  $C\subseteq\XX$ satisfies $\HH^{n-1}(C)=0$, and that for every $x\in\XX\setminus C$ there is $r_x\in (0,1)$ such that $f\in W^{1,2}(B_{r_x}(x))$ with $|\nabla f|\le g\ \HH^n$-a.e.\ on $B_{r_x}(x)$. 
		Then $f\in W^{1,2}(\XX)$.
	\end{lem}
	\begin{proof} 
     Set $U\defeq \bigcup_{x\notin C} B_{r_x}(x)\supseteq X\setminus C$.  Now fix $B_R(o)\subseteq\XX$ and notice that $\bar B_R(o)\setminus U$ is a compact subset of $ C$. For any $\sigma>0$, we can cover $\bar B_R(o)\setminus U$ with finitely many $B_{r_i}(x_i)$, such that $\sum_i r_i^{n-1}<\sigma$. Consider, for every $i$, the $1/r_i$-Lipschitz function $\varphi_i$ which is $1$ on $B_{r_i}(x_i)$ and is supported in $\bar B_{2r_i}(x_i)$. If we set $\varphi_\sigma \defeq (1- \sum_i\varphi_i)^+$, we can easily compute that 
		\begin{equation}\label{vfsdaca}
			\lim_{\sigma\searrow 0}\int_\XX [(1-\varphi_\sigma) +|\nabla \varphi_\sigma|]\,\dd\mass=0.
		\end{equation}
		
		We notice also that, for every $\sigma\in (0,1)$, $f\varphi_\sigma\in W^{1,2}(B_R(o))$ with 
		\begin{equation}
			|\nabla (f\varphi_\sigma)|\le |f||\nabla \varphi_\sigma|+g\varphi_\sigma\quad\HH^n\text{-a.e.\ on $B_R(o)$}.
		\end{equation}
		If we let $\sigma\searrow 0$, we deduce that  $f\in \BV(B_R(o))$. Notice that,  as $\HH^{n-1}(C)=0$, we have $|D f|(C)=0$,  by \cite[Theorem 3.4]{ABPrank} and the coarea formula. Then, by locality, $|D f|\le g\ \HH^n$-a.e.\ on $B_R(o)$. Thus, by \cite[Remark 3.5]{GigliHan14}, we get $f\in W^{1,2}(B_R(o))$, and it is now immediate to deduce $f\in W^{1,2}(\XX)$.
	\end{proof}

   \subsubsection{Codifferential and forms}	
   We assume the reader to be familiar with the standard notation regarding differential forms and (co)tangent modules over $\RCD$ spaces, referring to \cite{Gigli14}.

   \begin{lem}\label{vefdaa}
   	Let $(\XX, d,\HH^n)$ be a non-collapsed $\RCD(-(n-1) ,n)$ space and let $\omega\in\mathrm{TestForm}_k(\XX)$. 
   	Then $\omega\in D(\delta)$, namely $\omega$ belongs to the domain of the codifferential $\delta$, with
   	\begin{equation}\label{vfeadsc}
   		\delta\omega=-\sum_{i=1}^n \nabla_{e_i}\omega\mres e_i\quad\HH^n\text{-a.e.},
   	\end{equation}
   	where $(e_i)_{i=1,\dots,n}$ is an orthonormal basis for $L^2(T\XX)$.
   \end{lem}
   \begin{proof}
   	By linearity, it is enough to prove the statement for $\omega=f_0\,df_1\wedge \cdots\wedge df_k =f_0\tilde{\omega}$, for $f_0,\dots,f_k\in\mathrm{TestF}(\XX)$.
   	Take any $\eta=g_0 \, dg_1\wedge\dots \wedge dg_{k-1}\in\mathrm{TestForm}_{k-1}(\XX)$, with $g_0,\dots,g_{k-1}\in\mathrm{TestF}(\XX)$.
   	We can compute
   	\begin{equation}
   		\begin{split}
   			\int \omega\,\cdot\, d\eta\,\dd\HH^n&= \int df_1\wedge \dots \wedge df_k\,\cdot\, f_0 \, d g_0\wedge \dots \wedge dg_{k-1}\,\dd\HH^n
   			\\&= \int \tilde{\omega}\,\cdot\, d (f_0 g_0)\wedge dg_1\wedge \dots \wedge dg_{n-1}\,\dd\HH^n-\int \tilde\omega\,\cdot\, g_0 \, d f_0\wedge dg_1\wedge \dots \wedge dg_{n-1} \,\dd\HH^n
   			\\&=\int [f_0\delta\tilde{\omega}\,\cdot\, \eta- \tilde{\omega}\,\cdot\, df_0\wedge \eta]\,\dd\HH^n.
   		\end{split}
   	\end{equation}
   	Notice that, by linearity, the above holds for every $\eta\in \mathrm{TestForm}_{k-1}(\XX)$. This means that $\omega\in D(\delta)$, with
   	$$
   	\delta\omega	=f_0\delta\tilde\omega- \tilde{\omega}\mres d f_0.
   	$$
   	Therefore, by \cite[Proposition 3.5.12]{Gigli14},
   	\begin{equation}
   		\begin{split}
   			\delta\omega&= f_0 \sum_{a=1}^k (-1)^a \Delta f_a \, df_1\wedge \cdots \wedge \widehat{df_a}\wedge \cdots \wedge df_k
   			\\&\qquad+f_0\sum_{1\le b<c\le k}(-1)^{b+c} [\hess f_c(df_b,\,\cdot\,)-\hess f_b(df_c,\,\cdot\,)] \wedge \cdots\wedge\widehat {df_b}\wedge\cdots  \wedge\widehat{df_c}\wedge\cdots
   			\\&\qquad - \tilde\omega\mres df_0\\&\eqdef{\bf A}+{\bf B}+{\bf C}.
   		\end{split}
   	\end{equation}
   	Notice that  
   	\begin{equation}
   		\begin{split}
   			{\bf B}&=\sum_{1\le b<c\le k}(-1)^{b+c} [\hess f_c(df_b,\,\cdot\,)-\hess f_b(df_c,\,\cdot\,)] \wedge \cdots\wedge\widehat {df_b}\wedge\cdots\wedge \widehat {df_c}\wedge\cdots
   			\\&=\sum_{i=1}^n	\sum_{1\le b<c\le n}(-1)^{b+c} \hess f_c(e_i,\,\cdot\,) \wedge \cdots\wedge  d f_b( e_i)\wedge\cdots \wedge\widehat {df_c}\wedge\cdots
   			\\&\qquad-\sum_{i=1}^n 	\sum_{1\le b<c\le k}(-1)^{b+c} \hess f_b(e_i,\,\cdot\,) \wedge \cdots\wedge \widehat {df_b}\wedge\cdots \wedge  df_c(e_i) \wedge\cdots
   			\\&=\sum_{i=1}^n (-1)^b\sum_{1\le a, b \le k,\, a\ne b} df_1\wedge \cdots \wedge\hess f_a(e_i,\,\cdot\,) \wedge \cdots\wedge d f_b( e_i) \wedge\cdots.
   		\end{split}
   	\end{equation}
   	Also, 
   	\begin{equation}
   		\begin{split}
   			\sum_{i=1}^n	\nabla_{e_i}\omega\mres e_i &=\sum_{i=1}^n (\nabla f_0\,\cdot\,e_i) \tilde\omega \mres e_i+\sum_{i=1}^n  \sum_{a=1}^k f_0 \, [df_1\wedge\cdots\wedge \hess(f_a)(e_i,\,\cdot\,)\wedge \cdots]\mres e_i
   			\\&=\tilde\omega\mres d f_0+ \sum_{i=1}^n\sum_{a=1}^k [f_0 \, df_1\wedge\cdots\wedge \hess f_a(e_i,\,\cdot\,)\wedge \cdots]\mres e_i.
   		\end{split}
   	\end{equation}
   	
   	Now notice that, for every $a=1,\dots,k$ and $i=1,\dots,n$, we have
   	\begin{equation}
   		\begin{split}
   			&[df_1\wedge\cdots\wedge \hess f_a(e_i,\,\cdot\,)\wedge \cdots]\mres e_i\\
            &= -\sum_{ 1\le b\le k,\,b\ne a } (-1)^b df_1\wedge \cdots \wedge\hess f_a(e_i,\,\cdot\,) \wedge \cdots\wedge d f_b(e_i)\wedge\cdots\\&\qquad-(-1)^a
   			df_1\wedge \cdots \wedge \hess f_a(e_i,e_i)\wedge \cdots.
   		\end{split}
   	\end{equation}
   	Therefore, using that the Laplacian is the trace of the Hessian on non-collapsed $\RCD$ spaces, we obtain
   	\begin{equation}
   		\sum_{i=1}^n	\nabla_{e_i}\omega\mres e_i=-{\bf C}-{\bf B}-{\bf A}=-\delta\omega,
   	\end{equation}	
    as desired.
   \end{proof}

	\subsubsection{Splitting maps}

Let $(\XX, d,\HH^n)$ be a non-collapsed $\RCD(-(n-1)\delta,n)$ space. Assume that, for some $p\in\XX$,
\begin{equation}\label{eq:ghclose}
	d_{GH}(B_4(p), B_4(0^n))<\delta.
\end{equation}
We recall that, by volume monotonicity, the following holds. For every $\epsilon>0$, if $\delta\le \delta(\epsilon, n)$, each $B_r(x)\subseteq B_3(p)$ satisfies
\begin{equation}\label{stupidmass}
	d_{GH}(B_r(x), B_r(0^n))<\epsilon r\quad\text{and}\quad 1-\epsilon < \frac{\HH^n(B_r(x))}{\omega_n r^n} < 1+\epsilon .
\end{equation}
We refer the reader to \cite[Theorem 4.3]{CJNrect} and \cite[Theorem 1.3]{DPG17}.

	\begin{thm}[Regular Balls]
		\label{prel2}
        Assume \eqref{eq:ghclose}.
		If $\epsilon \le \epsilon(n)$ and $\delta \le \delta(\epsilon,n)$, then the following hold.
		\begin{enumerate}
			\item There exists an $\epsilon$-splitting map \begin{equation}
				u=(u_1,\dots,u_n):B_3(p)\rightarrow\RR^n.
			\end{equation}
			\item For every $x\in B_2(p)$ and $r \in (0,1)$, there exists an invertible matrix $T_{x,r}$, satisfying 
			\begin{equation}\label{T.distortion}
				|T_{x,r}|+|T_{x,r}^{-1}|\le r^{-\epsilon}
			\end{equation}
			and such that $T_{x,r}u:B_r(x)\rightarrow\RR^n$ is an $\epsilon$-splitting map and $(1+\epsilon)$-Lipschitz. In particular, 
			\begin{equation}
				|d u_1\wedge\cdots \wedge d u_n|>0\quad \HH^n\text{-a.e.\ on }B_2(p).
			\end{equation}

			\item	It holds that $u:B_1(p)\rightarrow\RR^n$ is bi-H\"older onto its image: more precisely,
			\begin{equation}\label{canreifeq}
				(1-\epsilon) d(x,y)^{1+\epsilon}\le |u(x)-u(y)|\le (1+\epsilon) d(x,y)\quad\text{for every $x,y\in B_1(p)$}
			\end{equation} 
		and $u(B_1(p))$ is open in $\RR^n$.
			
			\item The set $G$,  defined as in  \cite[Proposition 1.6]{BruPasSem20} by 
			\begin{equation}
				\begin{split}
					G\defeq \bigg\{x\in B_1(p)\,:\,&\sup_{r\in (0,1)}\dashint_{B_r(x)} |\hess u_a|^2\,\dd\HH^n<\sqrt{\epsilon}\text{ for every $a=1,\dots,n$ and } \\
					&\sup_{r\in (0,1)}\dashint_{B_r(x)}|\nabla u_a\,\cdot\,\nabla u_b-\delta_{a,b}|\,\dd\HH^n< \sqrt{\epsilon} \text{ for every $a,b=1,\dots,n$}\bigg\},
				\end{split}
			\end{equation}
			satisfies
			\begin{equation}\label{brvafedcs0}
				\HH^n(B_1(p)\setminus  G)\le C(n) \sqrt{\epsilon} \HH^n(B_1(p))
			\end{equation} and is such that 	
			\begin{equation}\label{cinquepuntotre}
				\big||u(x)-u(y)|-d(x,y)\big|\le \eta d(x,y)\quad\text{for every $x,y\in B_1(p)$ with $ x\in G$},
			\end{equation}
   provided that $\epsilon<\epsilon(\eta,n)$, for $\eta>0$.
		\end{enumerate}
	\end{thm}
	\begin{proof}
	    (1) and the first conclusion of (2) are nowadays standard and follow from the existence result for $\delta$-splitting maps and the transformation theorem \cite{Cheeger-Colding96,CJNrect}; see also \cite{BruPasSem20,bru2021constancy} for the proof of these results in the $\RCD$ setting.
	    To prove the second part of item (2), take any $x\in B_2(p)$ and $r\in (0,1)$ and notice that, if $w=w_{x,r}\defeq T_{x,r}u$, then 
		\begin{equation}
			|dw_1\wedge\dots \wedge dw_n|=|\det(T_{x,r})| |du_1\wedge\dots\wedge d u_n|\quad\HH^n\text{-a.e.\ on }B_{r}(x).
		\end{equation}
		In particular, $\HH^n$-a.e.\ on $B_r(x)$ it holds that $|dw_1\wedge\cdots \wedge dw_n|=0$ if and only if $|du_1\wedge\cdots\wedge du_n|=0$. Notice that, for a constant $c(n)$ depending only upon $n$,
		\begin{equation}
			|dw_1\wedge \cdots \wedge dw_n|>1/2\quad\HH^n\text{-a.e.\ on }B_r(x)\cap\{|\nabla w_a\,\cdot\,\nabla w_b-\delta_{a,b}|\le c(n)\ \forall\,   a,b=1,\dots,n\}.
		\end{equation}
		In particular, we deduce that
		\begin{equation}
			\begin{split}
				n^2\epsilon&\ge \sum_{a,b}\dashint_{B_r(x)}|\nabla w_a\,\cdot\,\nabla w_b-\delta_{a,b}|\dd\HH^n\ge c(n)\frac{\HH^n(\{|dw_1\wedge\cdots\wedge dw_n|=0\})}{\HH^n(B_r(x))}
				\\&=c(n)\frac{\HH^n(\{|du_1\wedge\cdots\wedge du_n|=0\})}{\HH^n(B_r(x))},
			\end{split}
		\end{equation}
		which, for $\epsilon< c(n)/n^2$ (which we can always assume), implies a contradiction at density points of $\{|du_1\wedge\cdots\wedge d u_n|=0\}$, showing that the latter has measure zero.

		
		The first part of (3) is a standard conclusion of the geometric transformation theorem \cite{CJNrect,brueboundary}. To prove that $u(B_1(p))$ is open in $\RR^n$, we fix $x\in B_1(p)$ 
        and let $v\defeq T_{x,r}u$, with $r\defeq 1-d(p,x)>0$.
        It suffices to show that $v(B_r(x))$ includes the ball $B_{r/2}(v(x))$.

        We argue by contradiction and assume that $z\in B_{r/2}(v(x))$ but $z\not\in v(B_r(x))$.
        Since $v$ provides a GH-equivalence between $B_r(x)$ and $B_r(v(x))$ (cf.\ \cite[Remark 3.10]{brueboundary}), assuming $\epsilon<\epsilon(n)$ there exists $y\in B_{3r/4}(x)$
        such that $v(y)$ has minimal distance $0<r'<r/10$ from $z$.
        We now study the $\epsilon$-splitting map $T_{y,4s} u : B_{4s}(y) \to \RR^n$ for $0<s<r'/4$. It turns out that
		\begin{itemize}
			\item[(i)] $T_{y,4s} u (B_{s}(y)) \subseteq B_{(1+\epsilon)s}(T_{y,4s}u(y))$ since $T_{y,4s} u$ is $(1+\epsilon)$-Lipschitz,
			
			\item[(ii)] $\HH^n(T_{y,4s} u (B_{s}(y))) \le \frac{3}{4} \omega_n s^n$ for $s$ small enough.
		\end{itemize}
		Here, (ii) follows from the fact that the image of $v(B_s(y))$ is  contained in the $C(n)s^2$-neighborhood of a half-space centered at $v(y)$ when $s$ is sufficiently small, as $v(B_s(x))$ does not intersect $B_{|z - v(y)|}(z)$. By the bounds \eqref{T.distortion}, $T_{y,4s} u (B_s(y))$ is contained in the $C(n)s^{2-\epsilon}$-neighborhood a half-space centered at $T_{y,4s}u(y)$. Hence, $T_{y,4s} u (B_s(y))$ cannot intersect a big portion (almost half) of $B_{(1+\epsilon)s}(T_{y,4s}u(y))$, providing the volume bound (ii).
		
		Finally, we show that (ii) contradicts the fact that $T_{y,4s} u : B_s(y) \to \RR^n$ is an $\epsilon$-splitting map,
        and in particular the bound \eqref{cinquepuntotre} proved below, for $\epsilon$ small enough.
        Indeed, let $w_\epsilon:=T_{y,4s} u$ and let $G_\epsilon\subseteq B_s(y)$ be the good set for the map $w_\epsilon$
        (defined as $G$ but in the scale-invariant way). Note that the proof of \eqref{cinquepuntotre} (for any given $\eta$)
        does not depend on \eqref{eq:ghclose}, but just on \eqref{stupidmass}, which trivially still hold when we replace $B_4(p)$
        with $B_{4s}(y)$.        
        Then, by \eqref{stupidmass} and \eqref{brvafedcs0}, we have
        $$\HH^n(w_\epsilon(G_\epsilon))\ge (1-C(n)\sqrt\epsilon)\omega_n s^n.$$
        By \eqref{cinquepuntotre}, the restriction $w_\epsilon:G_\epsilon\to w_\epsilon(G_\epsilon)$ is injective,
        with $(1-\eta)^{-1}$-Lipschitz inverse, giving
        $$\HH^n(G_\epsilon)\ge(1-\eta)^{-n}(1-C(n)\sqrt\epsilon)\omega_n s^n.$$
        This contradicts (ii) as soon as $\epsilon$, and thus $\eta$, are small enough.
        
		
		Now we turn to item (4). By \cite[Proposition 1.6]{BruPasSem20}, \eqref{brvafedcs0} follows, as well as the fact that $u:B_{s}(x)\rightarrow\RR^n$ is a $\sqrt{\epsilon}$-splitting map for every $x\in G$ and $s\in (0,1)$.
		
		Now, \eqref{cinquepuntotre}  follows from the same argument of the proof of \cite[Proposition 2.8]{BruPasSem20}, by using \cite[Theorem 3.4 (iii)]{brueboundary} and the argument of \cite{bru2019rectifiability}. We give anyway the details below. 
		
		Fix $x,y$ as in the statement.  Notice that, if $d( x, y)\ge 1/2$, the claim follows from the first part of the theorem, provided that $\epsilon<\epsilon(\eta,n)$, so that we will assume that $d( x, y)\le 1/2$.   		Hence, as $x\in G$, $u:B_{2d ( x, y)}( x)\rightarrow\RR^n$ is an $\sqrt{\epsilon}$-splitting map. By \eqref{stupidmass}, provided that $\epsilon<\epsilon(\eta,n)$, \cite[Remark 3.10]{brueboundary}  implies that  $u:B_{2d ( x, y)}( x)\rightarrow\RR^n$ provides us with a GH-equivalence witnessing the bound
		\begin{equation}
			d_{GH}(B_{2d (x, y)}(x),B_{2d ( x, y)}(0^n) )<\eta d ( x,y),
		\end{equation}
		from which the conclusion follows.
	\end{proof}

\subsection{Volume Form on Regular Balls}\label{volumeregularball}

	We consider an $\RCD(-(n-1)\delta,n)$ space $(\XX, d,\HH^n)$ and a regular ball $B_4(p) \subseteq \XX$, i.e.
	\begin{equation}
	d_{GH}(B_4(p), B_4(0^n))<\delta.
	\end{equation}
	From Theorem \ref{prel2}, if $\delta \le \delta(\epsilon,n)$ we can construct $\epsilon$-splitting maps $u: B_3(p) \to \RR^n$. It is then natural to consider the pullback form
	\begin{equation}
		\nu \defeq  du_1 \wedge \ldots \wedge du_n .
	\end{equation}
	We already know (see Theorem \ref{prel2}) that $|\nu|$ is nonzero $\HH^n$-a.e.\ and $\epsilon$-close to $1$ in an integral sense. The next result provides sharp Sobolev estimates on $|\nu|^{-1}$.

	\begin{prop}\label{integrability}
			Fix $r\in (1,\infty)$ and $q\in (1,2)$. If $\epsilon \le \epsilon(p,q,n)$ and $\delta \le \delta(\epsilon,n)$ then 
			\begin{equation}\label{eq:nu-1}
			|\nu|^{-1} \in L^r(B_{1/4}(p)) \cap W^{1,q}(B_{1/4}(p)) .
			\end{equation}
		    Also, there exists a lower-triangular matrix field $A,A^{-1}\in L^r\cap W^{1,q}(B_{1/4}(p); \RR^{n\times n})$, with $\det(A(x))>0$ for $\HH^n$-a.e. $x\in B_{1/4}(p)$, such that letting
      \begin{equation}\label{bvrascfds}
            e_i\defeq A_{i,j}\nabla u_j
      \end{equation}
      we obtain an orthonormal basis $(e_1,\dots,e_n)$ on $B_{1/4}(p)$.
	\end{prop}

   \begin{rem}
    The sharpness of Proposition \ref{integrability} can be verified on the two-dimensional ice-cream cone $C(S^1)$, where $\operatorname{diam}(S^1) = \frac{\pi}{1+\epsilon}$. If $\epsilon > 0$ is sufficiently small, then $u_1 = r^{1+\epsilon} \cos((1+\epsilon)\theta)$ and $u_2 = r^{1+\epsilon} \sin((1+\epsilon)\theta)$ constitute a $\delta$-splitting map, and $|\nu| = |du_1 \wedge du_2| =(1+\epsilon)^2 r^{1+2\epsilon}|dr\wedge d\theta|= (1+\epsilon)^2 r^{2\epsilon}$. Clearly, \eqref{eq:nu-1} is satisfied only for $q < \frac{2}{1+2\epsilon}$.
\end{rem}

    In view of Proposition \ref{integrability}, it is natural to define a {\it local volume form} by normalizing $\nu$:
    \begin{equation}
    	\omega_{B_{1/4}(p)} \defeq \frac{d u_1 \wedge \ldots \wedge d u_n}{|d u_1 \wedge \ldots \wedge d u_n|}=\frac{\nu}{|\nu|} ,
    	\quad \text{on $B_{1/4}(p)$} .
    \end{equation}

    \begin{prop}\label{sdfdascasdca}
       If $\delta<\delta(n)$, the following holds.
        \begin{enumerate}
            \item For every  $\eta\in\mathrm{TestForms}_{n}(\XX)$, $\omega_{B_{1/4}(p)}\,\cdot\,\eta\in W^{1,2}(B_{1/4}(p))$ with 				\begin{equation}\label{bssgfdvcac}
					\nabla (\omega_{B_{1/4}(p)}\,\cdot\,\eta)=\omega_{B_{1/4}(p)}\,\cdot\,\nabla \eta\quad\HH^n\text{-a.e.\ on $B_{1/4}(p)$}.
				\end{equation}

            \item For every  $\eta\in\mathrm{TestForms}_{n-1}(\XX)$ with $\supp(\eta)\subseteq B_{1/4}(p)$, \begin{equation}
    			\int_\XX \omega_{B_{1/4}(p)}\,\cdot\,d\eta\,\dd\HH^n=0.
    		\end{equation}
        \end{enumerate}
    \end{prop}
   The next result shows the opposite implication: any form $ \omega \in L^\infty(\Lambda^n T^* X)$ with constant modulus and weak Sobolev regularity coincides with $\omega_{B_{1/4}(p)}$, up to scaling by a constant.

    \begin{prop}\label{bgfvsdca}
    	Let $\omega\in L^\infty(\Lambda^n T^{*}\XX)$ be such that one of the following holds:
    	\begin{enumerate}
    		\item $| \omega(x)|=1$ for $\HH^n$-a.e. $x\in B_{1/4}(p)$ and $ \omega\,\cdot\,\eta\in W^{1,2}(B_{1/4}(p))$  for every $\eta\in\mathrm{TestForms}_{n}(\XX)$ with $\supp\eta\subseteq B_{1/4}(x)$;
    		\item $\omega\ne 0$ on $B_{1/4}(p)$ and, for every  $\eta\in\mathrm{TestForms}_{n-1}(\XX)$ with $\supp\eta\subseteq B_{1/4}(p)$, 	\begin{equation}
    			\int_\XX \omega\,\cdot\,d\eta\,\dd\HH^n=0.
    		\end{equation}
    	\end{enumerate}
    	Then, if $\delta<\delta(n)$, there exists a constant $c\in \RR\setminus\{0\}$ such that $ \omega = c\omega_{B_{1/4}(p)}$, $\HH^n$-a.e.\ on $B_{1/4}(p)$.
    \end{prop}

	\begin{proof}[Proof of Proposition \ref{integrability}]
 We begin by proving the regularity for $|\nu|^{-1}$.
		Notice first that if we show that $|\nu|^{-1} \in L^r(B_{1/4}(p))$ for all $r\in(1,\infty)$ then $ |\nu|^{-1} \in W^{1,q}(B_{1/4}(p))$ for all $q\in(1,2)$. Indeed, by direct computation, $|\nu|\in W^{1,2}$ with $|\nabla|\nu||\le C(\delta,n)\sum_a|\hess u_a|\in L^2$ and also, for any $\sigma\in (0,1)$, $\frac{1}{|\nu|+\sigma}\in W^{1,q}$ with 
		\begin{equation}
		\bigg|	\nabla \frac{1}{|\nu|+\sigma}\bigg|\le  \frac{|\nabla|\nu||}{(|\nu|+\sigma)^2}.
		\end{equation}
		Hence, the claim follows by H\"older's inequality,  given the integrability proved for  $|\nu|^{-1}$.
		
        Let $\alpha\in (0,2)$ be fixed (actually we are going to need only the case $\alpha=1$).
        We split the argument into two steps, the second of which involves iterating the first one.
		
		\textbf{Step 1}. Let $B_{2s}(x)\subseteq B_{1/2}(p)$. We claim that there exists a countable collection $(B_i)_i$ with $B_i=B_{r_i}(x_i)\subseteq B_{2r_i}(x_i)\subseteq B_{2s}(x)$  for every $i$, such that:
		\begin{itemize}
			\item[(i)] $|\nu|\ge c(n) s^{\epsilon n}\ \HH^n$-a.e.\ on $B_s(x)\setminus \bigcup_i B_i$;
			\item[(ii)]  $\sum_i r_i^{n-\alpha}\le \frac12 s^{n-\alpha}$.
		\end{itemize}
		
				Fix a ball $B_{2s}(x)$ as in the statement of the claim and consider $w=w_{x,2s}\defeq T_{x,2s}u$, which is an $\epsilon$-splitting map on $B_{2s}(x)$, by Theorem \ref{prel2}. Notice that $\HH^n$-a.e.\ we have
		 \begin{equation}
		\begin{split}	|dw_1\wedge\cdots\wedge dw_n| & =|\det(T_{x,2s})du_1\wedge\cdots\wedge d u_n|
        \\& \le C(n) |T_{x,2s}|^n|du_1\wedge\dots\wedge d u_n|
        \\&
        \le C(n) {(2s)}^{-\epsilon n} |\nu|,
        \end{split}
		\end{equation}
		where the last inequality is due to \eqref{T.distortion}. Hence, instead of item (i), it suffices to show
		\begin{itemize}
			\item [(i')] $|dw_1\wedge\cdots\wedge dw_n| \ge 1/2 \ \HH^n$-a.e.\ on $B_s(x)\setminus \bigcup_i B_i$.
		\end{itemize}
		Now notice that items (i') and (ii) are scale-invariant, so that we proceed in the proof considering the rescaled space $(\XX, s^{-1} d,\HH^n)$, which is still a non-collapsed $\RCD(-(n-1)\delta ,n)$ space. Hence, we will work on $B_1(x)$ in the rescaled space, and we have at our disposal an $\epsilon$-splitting map that we still call  $w:B_2(x)\rightarrow\RR^n$.

		Set for brevity $|\hess w|^2\defeq \sum_{a} |\hess w_a|^2\in L^1(B_{1/4}(p))$ and define, for $z\in B_1(x)$,
\begin{equation}
			Mw(z)\defeq \sup_{r\in (0,1)} r^\alpha \dashint_{B_r(z)}|\hess w|^2\,\dd\HH^n.
\end{equation}
		 It is enough to follow the  proof \cite[Proposition 3.11]{bru2019rectifiability}, so that we just sketch the details. Set $G\defeq \{z\in B_1(x)\,:\,Mw(z)\le \sqrt{\epsilon} \}$.
		 
	Taking $z\in G$, as in the proof of \cite[Proposition 3.11]{bru2019rectifiability} we have that 
	\begin{equation}
		\dashint_{B_r(z)}|\nabla w_a\,\cdot\,\nabla w_b-\delta_{i,j}|\le C(n,\alpha)\epsilon^{1/4}\quad\text{for every }r\in (0,1)\text{ and }a,b=1,\dots, n.
	\end{equation}
	Hence,  with a Lebesgue point argument, we see that 
\begin{equation}
		|\nabla w_a\,\cdot\,\nabla w_b-\delta_{a,b}|\le C(n,\alpha)\epsilon^{1/4}\quad\text{for every $a,b=1,\dots,n$}
\end{equation}
$\HH^n$-a.e.\ on $G$
so that, if $\epsilon$ is small enough, we obtain
\begin{equation}
	|dw_1\wedge\cdots\wedge dw_n|\ge 1/2\quad\HH^n\text{-a.e.\ on $G$},
\end{equation}
by simple linear algebra. 

Now it remains to cover $B_1(x)\setminus G$ by suitable balls. Take any $z\in B_1(x)\setminus G$, so that there exists $ r_z\in (0,1)$ with
\begin{equation}\label{vewafdsc}
\sqrt{\epsilon}\le  r_z^\alpha\dashint_{B_{r_z}(z)}|\hess w|^2\le C(n) r_z^{\alpha-n}\int_{B_2(x)}|\hess w|^2\le C(n) r_z^{\alpha-n}  \epsilon,
\end{equation}
where we used also \eqref{stupidmass} and the fact that $w:B_2(x)\rightarrow\RR^n$ is an $\epsilon$-splitting map. Notice in passing that this implies $r_z\le 1/10$, provided that $\epsilon$ is small enough.
By using Vitali’s covering lemma, we extract a sub-cover from $(B_{r_z}(z))_{z\in B_1(x)\setminus G}$, say $(\widehat B_i)_i$, such that $\bigcup_i 5\widehat B_i\supseteq  B_1(x)\setminus G$ and $(\widehat B_i)_i$ are pairwise disjoint, for $\widehat B_i\defeq B_{r_{z_i}}(z_i)$. Then, setting $r_i:=r_{z_i}$, by \eqref{vewafdsc} and \eqref{stupidmass} we have
\begin{equation}
\begin{split}
		\sum_i r_i^{n-\alpha}&\le C(n)\sum_i r_i^{-\alpha}{\HH^n(B_{r_i}(z_i))}\le C(n)\epsilon^{-1/2}\sum_i \int_{B_{r_i}(x_i)}|\hess w|^2\,\dd\HH^n\\
	&\le  C(n)\epsilon^{-1/2}\int_{B_{2}(x)}|\hess w|^2\,\dd\HH^n\le C(n) \sqrt{\epsilon}\le \frac12,
\end{split}
\end{equation}
provided that $\epsilon$ is small enough. Hence, it is now enough to define $B_i\defeq 5\widehat B_i$.

\textbf{Step 2}. 
	We iterate \textbf{Step 1}. We start from $B_{2s}(x)=B_{1/2}(p)$ and we obtain, thanks to \textbf{Step 1}, a collection of balls $(B_{i}^1=B_{r_{1,i}}(x_{1,i}))_{i}$. Then, at step $k$, for every ball $B_{l}^k$, we apply \textbf{Step 1} to $2B_{l}^k$ and we obtain a collection of balls  $(B_{l,m}^{k+1})_m$, and we set $(B_i^{k+1}=B_{r_{k+1,i}}(x_{k+1,i}))_i$ to be the union of all the balls $(B_{l,m}^{k+1})_{l,m}$.  Now, by \textbf{Step 1}, for every $k\in\NN$,
		\begin{equation}\label{vfeads}
			\sum_{i} r_{k,i}^{n-\alpha}\le \frac{1}{2} \sum_{i} r_{k-1,i}^{n-\alpha}\le\dots\le \frac{1}{2^{k}} (1/4)^{n-\alpha}.
		\end{equation}
	Hence, if we set ${\mathcal{B}}_k\defeq\bigcup_{i} B_{k,i}$, we see that $\HH^n(\mathcal{B}_k)\rightarrow 0$ as $k\rightarrow\infty$, so that we see that it is enough to find a bound for $\int_{B_{1/4}(p)\setminus\mathcal{B}_k} |\nu|^{-p}\dd\HH^n$ that is independent of $k$. To this aim, notice that
	\begin{equation}
		\int_{B_{1/4}(p)\setminus\mathcal{B}_k} |\nu|^{-r}\,\dd\HH^n\le\sum_{j= 0}^{k-1} \int_{\mathcal{B}_j\setminus\mathcal{B}_{j+1}} |\nu|^{-r}\,\dd\HH^n,
	\end{equation}
	where we set  $\mathcal{B}_0\defeq B_{1/4}(p)$. We conclude by computing, for $j\ge 0$, exploiting \textbf{Step 1} and \eqref{stupidmass}, that
\begin{equation}
		\int_{\mathcal{B}_j\setminus\mathcal{B}_{j+1}}|\nu|^{-r}\le \sum_{i}\int_{B^j_i\setminus \mathcal{B}_{j+1}}|\nu|^{-r}
		\le  \sum_{i}C(n) r_{j,i}^n C(n)^p r_{j,i}^{-\epsilon rn}\le C(n,r)\sum_{i} r_{j,i}^{n-\alpha}\le \frac{C(n,r)}{2^j},
\end{equation}
	provided that $\epsilon<\epsilon(r,n,\alpha)$. Summing over $j$, we obtain the bound in $L^r$.


		The proof of the second part follows from the classical Gram--Schmidt procedure. However, we have to take care of the regularity of the coefficients. Let $r',q'\in (1,\infty)$ to be determined later, depending upon $r,q,n$. 
		Define $\nu_a=d u_1\wedge\cdots\wedge d u_a$ for $a=1,\dots,n$. Notice that $|\nu_a|\ge  c(n)|\nu_n|$, so that, by what we proved above for $|\nu|^{-1}$, we have that
		\begin{equation}\label{befvdacas}
			|\nu_a|^{-1}\in L^{r'}\cap W^{1,q'}(B_{1/4}(p))\quad\text{for every }a=1,\dots,n,
		\end{equation}  
		provided that $\delta<\delta(\epsilon,r',q',n)$ ($r',q'$ are yet to be fixed, but depending only upon $r,q,n$).

		We start by defining on $B_{1/4}(p)$
		\begin{equation}
			\tilde e_1\defeq d u_1,\quad m_1\defeq |\tilde e_1|,\quad e_1=\frac{\tilde e_1}{m_1},
		\end{equation}
		and, for $a=2,\dots,n$, we define inductively
		\begin{equation}
			\tilde e_a\defeq d u_a-\sum_{b=1}^{a-1}(d u_a\,\cdot\,e_b )e_b,\quad m_a\defeq |\tilde e_a|,\quad e_a=\frac{\tilde e_a}{m_a}.
		\end{equation}
		Notice that this definition is well-posed, as $|du_1\wedge\cdots\wedge d u_n|>0\ \HH^n$-a.e.\ on $B_{1/4}(p)$ (which holds true by Theorem \ref{prel2}) implies that $m_a>0\ \HH^n$-a.e.\ on $B_{1/4}(p)$ for every $a=1,\dots,n$. 
		As the definition of $e_a$ involves terms $d u_b$ only with $b=1,\dots,a$,  we naturally have the lower-triangular matrix field $A:B_{1/4}(p)\rightarrow\RR^{n\times n}$. Also, $A$ has positive entries on the diagonal. Moreover, by construction, $(e_1,\dots,e_n)$ are orthonormal.
		We now need to show integrability and Sobolev regularity for $A$.
		
		A simple computation, exploiting the fact that  $(e_1,\dots,e_n)$ are orthonormal, yields
		\begin{equation}
		\begin{split}
				m_a&=\bigg| d u_a-\sum_{b=1}^{a-1} (du_a\,\cdot\,e_b )e_b\bigg|=\bigg| e_1\wedge\cdots\wedge e_{a-1}\wedge \bigg( d u_a-\sum_{b=1}^{a-1}(d u_a\,\cdot\,e_b) e_b\Big)\bigg|=|e_1\wedge\cdots \wedge e_{a-1}\wedge d u_a|
			\\&=\frac{1}{m_1\cdots m_{a-1}}|\tilde e_1\wedge \cdots\wedge \tilde e_{a-1}\wedge d u_a|=\frac{1}{m_1\cdots m_{a-1}}|d u_1\wedge\cdots\wedge d u_a|,
		\end{split}
		\end{equation}
		so that, for every $a=2,\dots,n$, we have
		\begin{equation}\label{bgrsdf}
			m_a=\frac{|\nu_a|}{|\nu_{a-1}|}\quad\HH^n\text{-a.e.\ on }B_{1/4}(p).
		\end{equation}

		Now notice that 
		\begin{equation}
\begin{split}
				e_a&=\frac{1}{m_a}\bigg(  d u_a-\sum_{b=1}^{a-1}\Big( d u_a\,\cdot\, \sum_{h=1}^b A_{b,h}\,d u_h\Big)\Big( \sum_{k=1}^b A_{b,k}\,d u_k\Big) \bigg)
				\\&=\frac{1}{m_a}d u_a-\frac{1}{m_a}\sum_{b=1}^{a-1}\sum_{h=1}^b\sum_{k=1}^b A_{b,h}A_{b,k} (d u_a\,\cdot\,d u_h) \, d u_k,
\end{split}
		\end{equation}
		so that, for $k=1,\dots,a-1$,
		\begin{equation}\label{brvfdaa}
			A_{a,k}=-\frac{1}{m_a}\sum_{b=k}^{a-1}\sum_{h=1}^b A_{b,h}A_{b,k} (d u_a\,\cdot\,d u_h) \quad\HH^n\text{-a.e.\ on }B_{1/4}(p).
		\end{equation}
		Now, notice that the regularity for $A_{a,a}=\frac{1}{m_a}$ follows from \eqref{befvdacas} and \eqref{bgrsdf} (for $a\ge 2$).
		Hence, we can easily prove that $A_{a,b}\in L^{r}\cap W^{1,q}(B_{1/4}(p))^{n\times n}$ for every $b=1,\dots,a-1$ by induction or recursion on $a$, given \eqref{brvfdaa}, exploiting \eqref{bgrsdf}, and keeping track of the integrability of $A_{a,b}$. Here the choice of $r'$ and $q'$, depending upon $r,q,n$, enters into play.
			\end{proof}

\begin{proof}[Proof of Proposition \ref{sdfdascasdca}]
    We first show (1). Taking $\eta\in \mathrm{TestForms}_n(\XX)$, recall by Proposition \ref{integrability} that $|\nu|^{-1}\in L^4\cap W^{1,4/3}(B_{1/4}(p))$, if $\epsilon<\epsilon(n)$ and $\delta<\delta(\epsilon,n)$. In particular $\omega_{B_{1/4}(p)} \,\cdot\,\eta\in L^4\cap W^{1,4/3}(B_{1/4}(p))$, with 
\begin{equation}\label{vrfdacaa}
\nabla (\omega_{B_{1/4}(p)}\,\cdot\,\eta)
=\bigg(\nabla \frac{1}{|\nu|} \nu 
 + \frac{1}{|\nu|}\nabla \nu \bigg)\,\cdot\,\eta
 + \omega_{B_{1/4}(p)} \,\cdot\,\nabla \eta\quad\HH^n\text{-a.e.\ on }B_{1/4}(p).
\end{equation}
Now, notice that a similar equation as above holds also for $\eta = \omega_{B_{1/4}(p)}$, in which case, by $|\omega_{B_{1/4}(p)}|=1$, we see that 
		\begin{equation}
			0=\nabla \bigg(\frac{1}{|\nu|} \nu\,\cdot\, \frac{1}{|\nu|}\nu\bigg)
            =2\bigg(\nabla \frac{1}{|\nu|}\nu + \frac{1}{\nu}\nabla \nu \bigg)\,\cdot\,\frac{\nu}{|\nu|}\quad\HH^n\text{-a.e.\ on }B_{1/4}(p).
		\end{equation}
Since $\omega_{B_{1/4}(p)}=\frac{\nu}{|\nu|}$ is a nowhere vanishing top form on $B_{1/4}(p)$, we deduce that 
		\begin{equation}\label{bsvfdaa}
			\nabla \frac{1}{|v|}v+\frac{1}{|v|}\nabla v=0 \quad\HH^n\text{-a.e.\ on }B_{1/4}(p),
		\end{equation}
		so that, by \eqref{vrfdacaa},  we have proved \eqref{bssgfdvcac}, and hence $\omega_{B_{1/4}(p)}\,\cdot\,\eta\in W^{1,2}(B_{1/4}(p))$, by \cite[Theorem 3.4]{GigliHan14}.
		
		We now prove (2). Fix $\eta\in \mathrm{TestForms}_{n-1}(\XX)$ with  $\supp(\eta)\subseteq B_{1/4}(p)$, say $\supp(\eta)\subseteq B\Subset B_{1/4}(p)$.  
Let $f_k\in\mathrm{TestF}(\XX)$ with $f_k\rightarrow \frac{1}{|\nu|}$ in $L^4\cap W^{1,4/3}(B)$, and $\supp(f_k)\subseteq B_{1/4}(p)$. By Lemma \ref{vefdaa},  for every $k$ it holds
\begin{equation}
\begin{split}
\int_\XX f_k \nu \,\cdot\,d\eta\,\dd\HH^n
&= -\int_\XX \sum_{i=1}^n \nabla_{e_i} (f_k \nu )\mres e_i\,\cdot\,\eta\,\dd\HH^n\\
&=-\int_\XX \sum_{i=1}^n\big(( \nabla f_k\,\cdot\,e_i) \nu + f_k \nabla_{e_i} \nu\big)\mres e_i\,\cdot\,\eta\,\dd\HH^n.
\end{split}
\end{equation}
If we let $k\rightarrow \infty$, we have
\begin{equation}
\int_\XX \omega_{B_{1/4}(p)}\,\cdot\,d\eta\,\dd\HH^n=-\int_\XX \sum_{i=1}^n\bigg(\bigg(\nabla\frac{1}{|\nu|}\,\cdot\,e_i\bigg)  \nu + \frac{1}{| \nu |} \nabla_{e_i}\nu\bigg)\mres e_i\,\cdot\,\eta\,\dd\HH^n=0,
\end{equation}
having used \eqref{bsvfdaa} for the last equality.
\end{proof}

\begin{proof}[Proof of Proposition \ref{bgfvsdca}]
      Remember our notation $\nu = du_1 \wedge \ldots \wedge du_n$ and $\omega_{B_{1/4}(p)} = \frac{\nu}{|\nu|}$.
   We can then find $f\in L^\infty(B_{1/4}(p))$ such that 
   	\begin{equation}
   		\omega= f \omega_{B_{1/4}(p)}\quad\HH^n\text{-a.e.\ on }B_{1/4}(p).
   	\end{equation}
   	Fix a ball $B\Subset B_{1/4}(p)$. It is enough to prove that $f= \omega\,\cdot\, \omega_{B_{1/4}(p)}$ is constant on $B$.
    
\textbf{Step 1}. Assume first that (1) holds.  By assumption, $\omega\,\cdot\, \nu \in L^\infty\cap W^{1,2}(B)$. Moreover, $\frac{1}{|\nu|}\in L^4\cap W^{1,3/2}(B_{1/4}(p))$, provided that $\delta<\delta(n)$, by Proposition \ref{integrability}.  Thus, we get $f\in W^{1,4/3}(B)$,  but also $|f|=1$ $\HH^n$-a.e.\ on $B_{1/4}(p)$ (as $|\omega|=1$ $\HH^n$-a.e.\ here), so that $f$ must be constant on $B$.

    \textbf{Step 2}. Assume instead that (2) holds. Let  $\varphi\in\mathrm{TestF}(\XX)$ with $\supp\varphi\Subset B$. Define also $(e_1,\dots,e_n)$ as in Proposition \ref{integrability}, for $p=6n$, $q=3/2$, assuming that $\epsilon<\epsilon(p,q,n)$ and $\delta<\delta(\epsilon,n)$. If $A$ denotes the matrix as in Proposition \ref{integrability}, let $A^{k}\in \mathrm{TestF}(\XX)^{n\times n}$ be  matrix fields such that, for every $a,b=1,\dots,n$, we have $A^k_{a,b}\rightarrow A_{a,b} $ in $L^{6n}\cap W^{1,3/2}(B)$, as $k\rightarrow\infty$. As in \eqref{bvrascfds}, this defines naturally $e_1^k,\dots,e_n^k$,
    which we view as 1-forms.
   	Hence, by assumption,
   	\begin{equation}\label{vfsdaa}
   		\int_\XX f e_1\wedge\cdots\wedge e_n\,\cdot\, d(\varphi e_2^k\wedge\cdots \wedge e_n^k)\,\dd\HH^n=	\int_\XX \omega\,\cdot\, d(\varphi e_2^k\wedge\cdots \wedge e_n^k)\,\dd\HH^n=0.
   	\end{equation}
   	Notice that for every $a=1,\dots,n$, $de_a^k= d\sum_{b=1}^n A_{a,b}^k\,d u_b=\sum_{b=1}^n dA_{a,b}^k\wedge du_b$, and this means that
   	\begin{equation}\label{vfsdaa1}
   		\|d(\varphi e_2^k\wedge\cdots \wedge e_n^k)\|_{L^{6/5}(B)}\le M\quad\text{for every $a=1,\dots,n$ and $k\ge 1$},
   	\end{equation}
   	where $M$ depends on the norms $\|A_{a,b}\|_{L^{6n}\cap W^{1,3/2}(B_{1/4}(p))}$, $\|\varphi\|_{L^\infty}$, $\|\nabla\varphi\|_{L^\infty}$, and $n$.
   	
   	Thus, take $(f_h)_h\subseteq\mathrm{TestF}(\XX)$ with $\supp f_h\subseteq B_{1/4}(p)$ and $f_h\rightarrow f$ in $L^{6}(B)$, with moreover $(f_h)_h\subseteq L^\infty(B_{1/4}(p))$ bounded.
   	By \eqref{vfsdaa} and \eqref{vfsdaa1}, we have that 
   	\begin{equation}
   		\bigg|	\int_\XX f_he_1\wedge\cdots\wedge e_n\,\cdot\, d(\varphi e_2^k\wedge\cdots \wedge e_n^k)\,\dd\HH^n\bigg|\le \| f-f_h\|_{L^6(B)} M\quad\text{for every }k,h.
   	\end{equation}
   	We then recall that $e_1\wedge\cdots\wedge e_n=\omega_{B_{1/4}(p)}$ on $B_{1/4}(p)$ and use Lemma \ref{sdfdascasdca} to integrate by parts  to deduce, letting  $k\rightarrow\infty$ and using the convergence $e^k_a\rightarrow e_a$ in $L^{6n}(B)$, that
   	\begin{equation}
   		\bigg|	\int_\XX\varphi e_1\wedge\cdots\wedge e_n\,\cdot\, d f_h\wedge  e_2\wedge\cdots \wedge e_n\,\dd\HH^n\bigg|\le \| f-f_h\|_{L^6(B)} M\quad\text{for every }h.
   	\end{equation}
   	This implies $\int_\XX \varphi d f_h\,\cdot\, e_1\,\dd\HH^n\to 0$. Repeating the same argument for different indices, we obtain
   	\begin{equation}
   		\lim_{h\to\infty} \int_\XX \varphi\, d f_h\,\cdot\, e_a\,\dd\HH^n=0\quad\text{for every }a=1,\dots,n.
   	\end{equation}
   	Hence, approximating again $e_a$ with $e_a^k$, we see that, for every $a=1,\dots,n$,
   	\begin{equation}
   		\begin{split}
   			0&=\lim_{h\to\infty}\lim_{k\to\infty} \int_\XX \varphi\, d f_h\,\cdot\, e_a^k\,\dd\HH^n\\
   			&=-\lim_{h\to\infty}\lim_{k\to\infty} \int_\XX f_h \bigg(d\varphi\,\cdot\, e_a^k+\varphi\sum_{b=1}^n d A^k_{a,b}\,\cdot\, d u_b\bigg)\,\dd\HH^n\\
   			&=-\int_\XX f\bigg(d\varphi\,\cdot\, e_a+\varphi\sum_{b=1}^n d A_{a,b}\,\cdot\, d u_b\bigg)\,\dd\HH^n,
   		\end{split}
   	\end{equation}
    where we used the fact that $u$ is harmonic.
   	Here we are considering $dA_{a,b}\in L^{0}(T\XX)$, with $|dA_{a,b}|\in L^{3/2}(B)$, limit of $dA_{a,b}^k$, and similarly we will do for differentials of other $W^{1,3/2}(B)$ functions, which is meaningful thanks to \cite[Theorem 3.4]{GigliHan14}.
   	By approximation, we see that the above continues to hold even for $\varphi\in W^{1,3/2}\cap L^\infty(B)$.
   	
   	Now, take  $Z\in \mathrm{TestV}(\XX)$ with $\supp(Z)\Subset B$. The above holds for $\varphi_a\defeq e_a(Z)$, so that
   	\begin{equation}\label{vfadasdc}
   		\int_\XX f\sum_{a=1}^n\bigg(d\varphi_a\,\cdot\, e_a+\varphi_a\sum_{b=1}^n d A_{a,b}\,\cdot\, d u_b\bigg)\,\dd\HH^n=0.
   	\end{equation}
   	We claim that the term in brackets  in \eqref{vfadasdc} equals $\mathrm{div}Z$, in which case we can conclude by Lemma~\ref{divconst}.
   	In order to verify this, take  $\psi\in \mathrm{TestF}(\XX)$ with $\supp\psi\Subset B$, and take $(\varphi_a^h)_h\subseteq \mathrm{TestF}(\XX)$ with $\supp\varphi_a^h\Subset B$ such that $\varphi_a^h\rightarrow\varphi_a$ in $W^{1,3/2}(B)$, with moreover $(\varphi_a^h)_h\subseteq L^\infty(B)$ bounded.
   	We compute
   	\begin{equation}
   		\begin{split}
   			-\int_\XX \psi\mathrm{div}(Z)\,\dd\HH^n&=\int_\XX d \psi\,\cdot\, \sum_{a=1}^n \varphi_a e_a\,\dd\HH^n=\lim_{h\to\infty}\lim_{k\to\infty} \int_\XX d \psi\,\cdot\, \sum_{a=1}^n \varphi_a^h e_a^k\,\dd\HH^n\\
   			&=-\lim_{h\to\infty}\lim_{k\to\infty} \sum_{a=1}^n\int_\XX \psi \bigg(d\varphi_a^h\,\cdot\, e_a^k+\varphi^h_a\sum_{b=1}^n d A_{a,b}^k\,\cdot\, d u_b\bigg)\,\dd\HH^n\\
   			&=- \sum_{a=1}^n\int_\XX \psi \bigg(d\varphi_a\,\cdot\, e_a+\varphi_a\sum_{b=1}^n d A_{a,b}\,\cdot\, d u_b\bigg)\,\dd\HH^n,
   		\end{split}
   	\end{equation}
    where we used again the harmonicity of $u$.
   	This is what we needed to verify.				
   \end{proof}

\subsection{Compatibility of Local Volume Forms}\label{sectcompat}

	Let $\epsilon,\delta\in(0,1)$. Let $(\XX, d,\HH^n)$ be a non-collapsed $\RCD(-(n-1)\delta,n)$ space. We consider two $\delta$-regular balls $B_{4r}(p), B_{4r'}(p') \subseteq \XX$, i.e.
	\begin{equation}
		d_{GH}(B_{4r}(p), B_{4r}(0^n))<\delta r,
		\quad
		d_{GH}(B_{4r'}(p'), B_{4r'}(0^n))<\delta r',
		\quad 0^n \in \RR^n.
	\end{equation}
    By Theorem \ref{prel2}, if $\delta \le \delta(\epsilon,n)$, we can construct $\epsilon$-splitting maps $u:B_{3r}(p)\rightarrow\RR^n$ and $u':B_{3r'}(q)\rightarrow\RR^n$. The next result shows that the local volume forms are compatible on $A\defeq B_{r}(p)\cap B_{r'}(p')$.

	        \begin{prop}\label{brgvfsd}
				For $\epsilon <\epsilon(n)$ and $\delta \le \delta(\epsilon,n)$, then
                we can find $f\in L^0(A)$ such that
		\begin{equation}\label{vfdcsas}
						du_{1}\wedge\cdots \wedge du_{n}=f\,du'_{1}\wedge\cdots \wedge du'_{n}\quad\HH^n\text{-a.e.\ on $A$}.
				\end{equation}
					Moreover,  if $\deg(u'\circ u^{-1})=1$ on a connected component $u(A')$ of $u(A)$, then $f>0$ $\HH^n$-a.e.\ on $A'$, whereas, if $\deg(u'\circ u^{-1})=-1$ here, then $f<0$ $\HH^n$-a.e.\ on $A'$.
            	\end{prop}

				\begin{proof}
First, recall that, by Theorem \ref{prel2}, $|du_1\wedge\cdots\wedge du_n|>0$ $\HH^n$-a.e.\ on $ B_{r}(p)$ and $|du'_1\wedge\cdots\wedge du'_n|>0$ $\HH^n$-a.e.\ on $ B_{r'}(p')$. 

\textbf{Step  1}.  Set $\nu \defeq du_1\wedge\cdots \wedge du_n$ and $\omega_{B_r(p)} \defeq \frac{\nu}{|\nu|}$, and similarly for $\nu'$ and $\omega_{B_{r'}(p')}$. By what we just remarked, we have an $f$ as in \eqref{vfdcsas}. 
Defining $g$ by $\omega_{B_r(p)} =g \omega_{B_{r'}(p')}$ $\HH^n$-a.e.\ on $A$, we now claim that $g$ is locally constant,
implying that $f$ has locally constant sign.
Notice that, as $|g|=1\ \HH^n$-a.e., it is enough to prove that $g\in W^{1,q}(A)$, for some $q>1$,
as this gives $\nabla g=0$ $\HH^n$-a.e.\ on $\{g=1\}$ and on $\{g=-1\}$. We have
\begin{equation}
g
=\omega_{B_r(p)}\,\cdot\, \omega_{B_{r'}(p')}
={\nu\,\cdot\,\nu'}\frac{1}{|\nu|}\frac{1}{|\nu'|}\quad\HH^n\text{-a.e.\ on $A$}.
\end{equation}
Since $\nu\,\cdot\,\nu'\in W^{1,2}\cap L^\infty$ and $\frac{1}{|\nu|},\frac{1}{|\nu'|}\in W^{1,3/2}\cap L^{5}$  by Proposition \ref{integrability}, then $g\in W^{1,q}$ for some $q>1$, provided that $\epsilon<\epsilon(n)$ and $\delta<\delta(\epsilon,n)$.
				
Up to changing the sign of $u'_1$, we will assume, in the following, that $f>0$ $\HH^n$-a.e.\ on a connected component $A'\subseteq A$.

\textbf{Step 2}. Let $G,G'$ be the sets associated with $u,u'$, as in Theorem \ref{prel2}. We want to show that we can assume with no loss of generality that $\HH^n(G\cap G')>0$. 

Take any point $\hat x\in A'$. By Theorem \ref{prel2}, if $s$ is small enough, we have two invertible matrices $\widehat T,\widehat T'$ such that $\hat u\defeq \widehat T u,\hat u'\defeq \widehat T'u':B_{2s}(\hat x)\rightarrow\RR^n$ are $\epsilon$-splitting maps. If we define the sets $\widehat G,\widehat G'$ as in Theorem \ref{prel2}, we have $\HH^n(B_s(\hat x)\setminus \widehat G) + \HH^n(B_s(\hat x)\setminus \widehat G')\le C(n) \sqrt{\epsilon}\HH^n(B_s(\hat x))$. Thus, $\HH^n(\widehat G\cap \widehat G')>0$ if $\epsilon<\epsilon(n)$. Assuming that we have proved the result for $\hat u,\hat u'$, we claim that this implies the conclusion of the proof. Indeed,  set $\hat\nu \defeq d\hat u_1\wedge\cdots\wedge d\hat u_n =(\det\widehat T)\nu$ and define similarly $\hat \nu'$, so that 
\begin{equation}
\hat \nu= \frac{\det \widehat T}{\det \widehat T'}f\hat \nu'\quad\HH^n\text{-a.e.\ on }B_s(\hat x).
\end{equation}
Hence, by the result for $\hat u,\hat u'$ and the fact that $f>0$, we have that 
\begin{equation}
\deg(\hat u'\circ\hat u^{-1})=\sign\bigg(\frac{\det \widehat T}{\det \widehat T'}\bigg)=\sign(\det \widehat T')\sign(\det \widehat T),
\end{equation}
so that 
\begin{equation}
		\deg( u'\circ u^{-1})=\deg (\widehat T'\circ \hat u'\circ \hat u^{-1}\circ \widehat T^{-1})=\sign(\det \widehat T')\deg(\hat u'\circ \hat u^{-1})\sign(\det \widehat T)=1.
\end{equation}
Hence, in what follows, we replace $u,u'$ with $\hat u,\hat u'$, but we go back to the original notation $u,u'$.

\textbf{Step 3}. 
First, notice that $u'\circ u^{-1}$ is $\LL^n$-differentiable on $u(G\cap G')$. Indeed, by  \eqref{cinquepuntotre}, $u:G\rightarrow\RR^n$ is bi-Lipschitz onto its image, and the same holds for $u'$. Hence, setting $F$ to be any Lipschitz extension of $u'\circ u^{-1}:u(G\cap G')\rightarrow \RR^n$,  then $F$ is differentiable $\LL^n$-a.e.\ on $u(G\cap G')$. Take $\bar x$ a density point for $u(G\cap G')$ at which $F$ is differentiable, so that
\begin{equation}
\lim_{u(G\cap G')\ni x\rightarrow \bar x}\frac{|u'\circ u^{-1}(x)-u'\circ u^{-1}(\bar x)-\nabla F(\bar x)\,\cdot\,( x-\bar x)|}{|\bar x-x|}=0.
\end{equation}
For $\eta\in (0,1)$, if $s>0$ is small enough, for every $x\in B_s(\bar x)$ we can find $x'\in B_s(\bar x)\cap u(G\cap G')$ with $|x-x'|\le\eta |x-\bar x|$, by density reasons. We estimate, by \eqref{canreifeq} and \eqref{cinquepuntotre},
				\begin{equation}
\begin{split}
						|u'\circ u^{-1}(x)-u'\circ u^{-1}( x')-\nabla F(\bar x)\,\cdot\,( x-x')|&\le 2|u^{-1}(x)-u^{-1}(x')|+|\nabla F(\bar x)| |x-x'|\\&\le (4+|\nabla F(\bar x)| )|x-x'|\le (4+|\nabla F(\bar x)| )\eta |x-\bar x|.
\end{split}
				\end{equation}
				Hence,
				\begin{equation}
						\frac{|u'\circ u^{-1}(x)-u'\circ u^{-1}(\bar x)-\nabla F(\bar x)\,\cdot\,(\bar x-x)|}{| x-\bar x|} \le (4+|\nabla F(\bar x)|)\eta+o(|x'-\bar x|),
				\end{equation}
				whence the claim, as $\eta$ was arbitrary.
                
                \textbf{Step 4}. Now we make some standard reductions, which amount to negligible modifications of $G,G'$. First, we can assume that $|\nabla u_a\,\cdot\,\nabla u_b-\delta_{a,b}|<\sqrt{\epsilon}\ \HH^n$-a.e.\ on $G$, for every $a,b=1,\dots,n$, and similarly for $u'$. This follows from the definition of the sets $G,G'$, taking Lebesgue points. Moreover, we can also assume that every point of $G\cap G'$ is a Lebesgue point for $\nabla u_a\,\cdot\,\nabla u_b$, $\nabla u_a\,\cdot\,\nabla u_b'$, and $\nabla u_a'\,\cdot\,\nabla u_b'$, for every $a,b=1,\dots,n$. Also, we can assume that $G\cap G'\subseteq\mathcal{R}$. Fix then $\bar x\in G\cap G'$ a density point, such that $u^{-1}$ is differentiable at $u(x)$,  which is a density point for $u(G\cap G')$. Up to translations in $\RR^n$, we assume that $u(\bar x)=u'(\bar x)=0^n$.
				
				Notice also that, by Theorem \ref{prel2}, $u'\circ u^{-1}$ is an homeomorphism onto its open image. Hence, by the differentiability property at $0$, it is enough to show
				\begin{equation}\label{claim}
					\det(d (u'\circ u^{-1})(0))>0.
				\end{equation}
				
				Take a sequence of radii $(r_k)_k$ with $r_k\rightarrow 0$ and, up to subsequences, assume that $(\XX,r_k^{-1}d,\HH^n,\bar x)\rightarrow (\RR^n, d_e,\LL^n,0)$ in the pmGH sense. Fix a realization $(\ZZ,d_\ZZ)$ of such convergence and coordinates on $\RR^n$. This is possible as  $\bar x\in\mathcal{R}$.
				Up to taking a further subsequence, we can assume that, for every $a=1,\dots,n$, $u_{a,k}\defeq r_k^{-1}u_a\rightarrow L_a$ locally uniformly for some $L_a:\RR^n\rightarrow \RR$ Lipschitz and harmonic, for the realization as above, with $|\nabla u_{k,a}|^2\rightarrow |\nabla L_a|^2$ locally strongly in $L^1$.  This comes from  the convergence results of \cite{Ambrosio2017local}: see e.g.\ \cite[Section 1.2.3]{bru2019rectifiability}.
                
                As $\bar x$ is a Lebesgue point for $|\nabla u_a|^2$,  we can use the results of \cite{AH16, Ambrosio2017local} (see in particular  \cite[Proposition 1.27]{bru2019rectifiability}) to see that $|\nabla L_a|$ is $\LL^n$-a.e.\ constant on $\RR^n$. Finally, by the Bochner inequality (in the form of, e.g., \cite[(1.22)]{bru2019rectifiability}), we can actually prove that $(L_1,\dots,L_n)$ are linear functions, which we write, with a slight abuse, as $x\mapsto L x$. Of course, we assume that the analogous result holds for $u'$, so that we obtain a linear function $L'=(L'_1,\dots,L_n'):\RR^n\rightarrow\RR^n$. Again by the convergence $|\nabla u_{k,a}|^2\rightarrow | L_a|^2$, and by polarization, taking the Lebesgue representatives, we see that, for every $a,b=1,\dots,n$, $ L_a\,\cdot\, L_b=\nabla u_a\,\cdot\,\nabla u_b (\bar x)$,  $ L_a\,\cdot\, L_b'=\nabla u_a\,\cdot\,\nabla u_b' (\bar x)$.
				
				By \eqref{vfdcsas} and simple linear algebra, we see that $\det(\nabla u_a\,\cdot\,\nabla u_b'(\bar x))_{a,b}>0$, and we read this as $\det(L\circ (L')^T)=\det(L_a\,\cdot\, L_b')_{a,b}>0$. This means that $\det L$ and $\det L'$ have the same sign, so that we immediately obtain 
				\begin{equation}\label{venjoa}
					\sign\big(\det(L'\circ L^{-1})\big)=1.
				\end{equation}
                
                \textbf{Step 5}.
				Now, fix $a=1,\dots,n$. We take $(y_k)_k\subseteq u(G)$, with
				\begin{equation}
					\frac{|y_k|}{r_k}\to1\quad\text{and}\quad\frac{y_k}{|y_k|}\rightarrow e_a,
				\end{equation} 
			 which we can do  as $0$ is a density point for $u(G)$. Setting $\bar y_k\defeq u^{-1}(y_k)$, by \eqref{cinquepuntotre} we have 
			 \begin{equation}
			 	0<\liminf_k\frac{d(\bar y_k,\bar x)}{r_k}\le\limsup_k\frac{d(\bar y_k,\bar x)}{r_k} <\infty,
			 \end{equation}
			 so that we can extract a (not relabeled) subsequence with $\bar y_k\rightarrow \bar y\in\RR^n\setminus\{0\}$, in $(\ZZ,d_\ZZ)$.
			 We then see that 
			 \begin{equation}
			 	\lim_{k\to\infty} r_k^{-1} u'\circ u ^{-1}(y_k)=\lim_{k\to\infty} r_k^{-1} u'(\bar y_k)=L' \bar y=L'\circ L^{-1}(L\bar y).
			 \end{equation}
			 But, similarly,
			 \begin{equation}
			 L\bar y=	\lim_{k\to\infty} r_k^{-1} u(\bar y_k)=\lim_{k\to\infty} r_k^{-1} y_k=e_a,
			 \end{equation}
			 so that $d (u'\circ u^{-1})=L'\circ L^{-1}$,  and we conclude \eqref{claim} by \eqref{venjoa}.
				\end{proof}

\subsection{Proof of Theorem \ref{brfvaed}}

Let $\delta<\delta(n)$ to be fixed later. Notice that we can replace $A$ by $A\cap A_{\delta/4}(\XX)$, getting an open set which is an orientable topological manifold and whose complement has vanishing $\HH^{n-1}$-measure. By definition of $A_\delta(\XX)$,  for every $p\in A$, there exists $r_p\in (0,\delta/4)$ with $d_{GH}(B_{4r_p}(p),B_{4r_p}(0^n))<\delta r_p$.
	
 Assume that $\delta<\delta(n)$ to have the claims of Proposition \ref{brgvfsd} (in particular the scale-invariant version of Theorem \ref{prel2}). More precisely, these statements depend also on a parameters $\epsilon,\eta$, which we fix accordingly, depending only upon $n$.
 
 For every $p\in A$, consider  $u:B_{3r_p}(p)\rightarrow\RR^n$  the $\epsilon$-splitting maps given by (the scale-invariant version of) Theorem \ref{prel2}. Consider the $n$-forms $\nu \defeq du_1\wedge\cdots\wedge d u_n:B_{r_p}(p)\rightarrow\RR^n$ and $\omega_{B_{r_p/4}(p)} \defeq \frac{\nu}{|\nu|}$ (which is well defined by Theorem \ref{prel2}). As $A$ is an oriented topological manifold, up to replacing $u_1$ with $-u_1$, we can assume $\deg(u)=1$. Then, define the  form $\omega_X$ on $X$
 by letting $\omega_X\defeq\omega_{B_{r_p/4}(p)}$ $\HH^n$-a.e.\ on $B_{r_p/4}(p)$. We have to prove that this definition is well-posed. 
Take $B_{r'}(p')$ with $p'\in A$ and $r'\defeq r_{p'}\in (0,\delta/4)$ as above. Then, consider the form $\omega_{B_{r'/4}(p')}$ as above. We need to check that $\omega_{B_{r_p/4}(p)} = \omega_{B_{r'/4}(p')}$ $\HH^n$-a.e.\ on $B_{r_p/4}(p)\cap B_{r'/4}(p')$. This, however, follows from Proposition \ref{brgvfsd}, as $\deg(u'\circ u^{-1})=1$.

Now we want to prove regularity for $\omega_\XX$. Item (1) follows from Proposition \ref{sdfdascasdca} with Lemma \ref{sobolevext}. We turn to item (2). Fix $\eta\in \mathrm{TestForms}_{n-1}(\XX)$ with compact support, say $\supp \eta\subseteq B_R(o)$. For $\sigma>0$, consider $\varphi_\sigma$ as in the proof of Lemma \ref{sobolevext}. Notice that still $\varphi_\sigma\eta\in D(d)$, with $d(\varphi_\sigma \eta)=d \varphi_\sigma \wedge\eta + \varphi_\sigma \, d\eta$. Hence, by \eqref{vfsdaca}, it suffices to prove 
		\begin{equation}
			\int_\XX \omega_X\,\cdot\, d(\varphi_\sigma\eta)\,\dd\HH^n=0\quad\text{for every }\sigma\in (0,1).
		\end{equation}
  Then, by a partition of unity argument, we see that there is no loss of generality in assuming that $\supp\eta\subseteq B_{r_p/4}(p)$, for some $B_{r_p}(p)$ as above,  so that the conclusion is due to Proposition \ref{sdfdascasdca} with an approximation argument.\qed	

	\subsection{Proof of Theorem \ref{vfedacadc}}
		Take $\epsilon < \epsilon(n)$ and $\delta < \delta(\epsilon,n)$, to be fixed later, hence $\delta<\delta(n)$.
		As already recalled in the proof of Theorem \ref{brfvaed}, for every  $x\in A_{\delta/4}(\XX)$, there exists $r_x\in (0,\delta/4)$ such that $d_{GH}(B_{4r_x}(x), B_{4r_x}(0^n))<\delta r_x$.
		
		We have to prove that $A_{\delta/4}(\XX)$ is orientable. Take $x\in A_{\delta/4}(\XX)$ and $r_x\in (0,\delta/4)$ as above, so that by Theorem \ref{prel2} we have an $\epsilon$-splitting map $u=(u_1,\dots,u_n):B_{3r_x}(x)\rightarrow\RR^n$. By Theorem \ref{prel2}, $B_{r_x}(x)$ is orientable.  By Proposition \ref{bgfvsdca}, if $\epsilon<\epsilon(n)$ and $\delta<\delta(\epsilon,n)$,
		\begin{equation}\label{bbfdssadc}
		\omega = c \frac{\nu}{|\nu|}=c\omega_{B_{r_x/4}(x)}\quad\HH^n\text{-a.e.\ on $B_{r_x/4}(x)$},
		\end{equation}
		where $\nu\defeq du_1\wedge\cdots\wedge d u_n$, and where $c$ might a priori depend on $x$.  But this implies that $|\omega|$ is locally constant on the connected set $A_{\delta/4}(\XX)$, so that $|\omega|=c$ holds $\HH^n$-a.e.\ on $\XX$.
		Up to replacing  $u_1$ with $-u_1$, we can assume that $c> 0$.
		Consider then, as chart for $B_{r_x/4}(x)$,  the map $u$. 
		Take $x,y\in A_{\delta/4}$ and $r_x,r_y$, and assume that $B_{r_x}(x)\cap B_{r_y}(y)\ne\emptyset$.   Let $u_x,u_y$ be the maps as above. By Proposition \ref{brgvfsd}, thanks to the choice of the sign of $u_{x,1}, u_{y,1}$, we have that $\deg(u_y\circ u_x^{-1})=1$, so that we see that the atlas for $A_{\delta/4}$ is oriented.
	\qed

\subsection{Proof of Remark \ref{remmodulo}}
We sketch how we need to modify the arguments of our proofs in order to obtain the claim. First, we replace (1) of Proposition \ref{bgfvsdca} with 
\begin{itemize}
    \item [(1')] $|\omega(x)|\ge c>0$ for $\HH^n$-a.e. $x\in B_{1/4}(p)$. Moreover, $ \omega\,\cdot\,\eta\in W^{1,2}(B_{1/4}(p))$  for every $\eta\in\mathrm{TestForms}_{n}(\XX)$ with $\supp\eta\subseteq B_{1/4}(x)$. 
\end{itemize}
The conclusion is then that \begin{equation}
    \omega = f\frac{du_1\wedge\cdots\wedge du_n}{|du_1\wedge\cdots\wedge du_n|}=f\omega_{B_{1/4}(p)}\quad\HH^n\text{-a.e.\ on }B_{1/4}(p),
\end{equation}
    where either $f\ge c$ $\HH^n$-a.e.\ on $B_{1/4}(p)$ or $f\le-c$ $\HH^n$-a.e.\ on $B_{1/4}(p)$.
This is proved similarly as for \textbf{Step 1} of the proof of Proposition \ref{bgfvsdca}: we still have that $f\in W^{1,4/3}(B)$, for any ball $B\Subset B_{1/4}(p)$, as well as $|f|\ge c>0$ $\HH^n$.a.e.\ on $B$.
This forces 
the claim, by well-known properties of Sobolev functions. Indeed, considering $g:=\max\{f,-c\}$ and $h:=\min\{f,c\}$,
since $h\in W^{1,4/3}(B)$ and $h$ has only finitely many values (namely, $\pm c$), we see that $h$ must be constant.

Then, the proof of Theorem \ref{vfedacadc} can be followed almost verbatim. Indeed, keeping the same notation, we have that \eqref{bbfdssadc} has to be replaced by 	
\begin{equation}
\omega=f\frac{\nu}{|\nu|}=\omega_{B_{r_x/4}(x)}\quad\HH^n\text{-a.e.\ on $B_{r_x/4}(x)$},
\end{equation}
where either $f\ge c$ $\HH^n$-a.e.\ or $f\le -c$ $\HH^n$-a.e. Then, up to replacing $u_1$ with $-u_1$, we see that we can assume that $f\ge c$ holds $\HH^n$-a.e.\ on $B_{r_x/4}(x)$, so that, as in the proof of Theorem \ref{vfedacadc}, the conclusion follows from Proposition \ref{brgvfsd}.
\qed    

\subsection{Proof of Theorem \ref{uniqueness}}
Let $\delta<\delta(n)$ to be fixed later.
By Theorem \ref{vfedacadc}, $A_{\delta/4}(\XX)$ is orientable, so that we can apply Theorem \ref{brfvaed} to obtain  an orientation form $\omega_\XX\in L^\infty(\Lambda^n T^{*}\XX)$. 
As in the proof of Theorem \ref{vfedacadc}, for every $x\in A_{\delta/4}(\XX)$ there exists $r_x\in (0,\delta/4)$ such that $d_{GH}(B_{4r_x}(x), B_{4r_x}(0^n))<\delta r_x$. Again, as in the proof of Theorem \ref{vfedacadc}, we obtain an $\epsilon$-splitting map $u=(u_1,\dots,u_n):B_{3r_x}(x)\rightarrow\RR^n$, by Theorem \ref{prel2}. By Proposition \ref{bgfvsdca}, if $\epsilon<\epsilon(n)$ and $\delta<\delta(\epsilon,n)$, for $\nu \defeq du_1\wedge\cdots\wedge d u_n$,
\begin{equation}
\omega_\XX = c_\XX\frac{\nu}{|\nu|}=c_X\omega_{B_{r_x/4}(x)}\quad\HH^n\text{-a.e.\ on $B_{r_x/4}(x)$},
\end{equation}
and, for $i=1,2$,
\begin{equation}\label{bbfdssadc2}
\omega_i=c_i\frac{\nu}{|\nu|}=c_i\omega_{B_{r_x/4}(x)}\quad\HH^n\text{-a.e.\ on $B_{r_x/4}(x)$},
\end{equation}
for $c_\XX,c_1,c_2\in\RR\setminus\{0\}$ which may depend on $x$. 
  As $A_{\delta/4}(\XX)$ is connected, we have
  \begin{equation}
      \omega_i=\frac{c_i}{c_\XX}\omega_\XX\quad\HH^n\text{-a.e.\ on $B_{r_x/4}(x)$},
  \end{equation}
  and the last ratio must be constant on $A_{\delta/4}$ (as it is locally constant).
  This concludes the proof.
\qed

				\section{Metric Currents}\label{sectioncurrents}
				In this section, we are going to work with an $\RCD(K,N)$ space $(\XX,d,\mass)$. We will need the machinery of currents on metric spaces.

	As $\mass$ can be not finite, we consider local metric currents (but we still denote them as metric currents), as defined in \cite{lan11}, by modification of \cite{AmbrosioKirchheim00}.
	We then recall that $\mathcal{D}^k(\XX)\defeq \LIP_c(\XX)\times\LIP_{\mathrm{loc}}(\XX)^k$, for $k\in\NN$,
    is in some sense their pre-dual. As $(\XX,d)$ is proper, in practice it will be equivalent to consider $\mathcal{D}^k\defeq \LIP_c(\XX)^{k+1}$.
				
	For $\omega\in L^0(\Lambda^k T^*\XX)$, we will consider both the usual pointwise norm $|\omega|$ and the pointwise co-mass norm $|\omega|_c$, defined as
				\begin{equation}
					|\omega|_c\defeq\esssup\{\omega\,\cdot\,\tau_1\wedge\dots \wedge \tau_k \mid \tau_1,\dots,\tau_k\in L^0( T^*\XX) \text{ and }|\tau_1\wedge\dots\wedge  \tau_k|\le 1\ \mass\text{-a.e.}\}.
				\end{equation} 
    
    \begin{lem}\label{lemmanormaforme}
                    With the notation above, we have that
            \begin{equation}\label{lemmanormaforme1}
					|\omega|_c\le |\omega|\le {n \choose k} |\omega|_c\qquad\mass\text{-a.e.}
				\end{equation}
                and\begin{equation}\label{lemmanormaforme2}
					|\omega|_c=\esssup\{\omega\,\cdot\,d\pi_1\wedge\cdots \wedge d\pi_k \mid \pi_1,\dots, \pi_k\text{ are $1$-Lipschitz}\}\quad\mass\text{-a.e.}
				\end{equation} 
    \end{lem}
				
	The next statement establishes a correspondence between forms and currents of bounded weight.
    
			\begin{prop}\label{formsandcurrents}
					Let $(\XX, d,\mass)$ be an $\RCD(K,N)$ space of essential dimension $n$ and let $k\in\{1,\dots,n\}$. 			
					Every $\omega \in L^\infty(\Lambda^k T^*\XX)$  induces a metric $k$-current $T_\omega$  by
					\begin{equation}\label{vefadcsa}
						\mathcal{D}^k(\XX)\ni (f,\pi_1,\dots,\pi_k)\mapsto \int f\omega\,\cdot\, d\pi_1\wedge\cdots\wedge d\pi_k,
					\end{equation}
					and $|T_\omega|=|\omega|_c$ $\mass$-a.e. 
					Conversely, every $k$-metric current $T$ with $|T|\in L^\infty$ is induced by a (unique) $k$-form $\omega_T \in L^\infty(\Lambda^k T^*\XX)$ through \eqref{vefadcsa}.
				\end{prop}	

        Finally, the next two results show that the form is co-closed precisely when the current has no boundary.

    	\begin{thm}\label{maincurr}
					Let $(\XX,d,\mass)$ be an $\RCD(K,N)$ space of essential dimension $n$. Let $\omega\in L^\infty(\Lambda^kT^*\XX)$ satisfying
		\begin{equation}
			\int_\XX \omega\,\cdot\,d\eta\,\dd\mass=0\quad\text{for every  $\eta\in\mathrm{TestForms}_{k-1}(\XX)$ with compact support}.
		\end{equation}
		
		Then the current $T_\omega$ induced as in Proposition \ref{formsandcurrents} has no boundary.
				\end{thm}

    \begin{thm}\label{bgrvfdcassc}
							Let $(\XX,d,\mass)$ be an $\RCD(K,N)$ space of essential dimension $n$. Let $T$ be a metric $k$-current with $|T|\in L^\infty$ and assume that $T$ has no boundary. Then,  if $\omega_T\in L^\infty(\Lambda^k T^*\XX)$ is the $k$-form induced as in Proposition \ref{formsandcurrents}, we have
									\begin{equation}
								\int_\XX \omega\,\cdot\,d\eta\,\dd\mass=0\quad\text{for every  $\eta\in\mathrm{TestForms}_{k-1}(\XX)$ with compact support}.
							\end{equation}
				\end{thm}

				\subsection{Proof of Lemma \ref{lemmanormaforme}}
                
                First, \eqref{lemmanormaforme1} follows from the existence of local bases for $L^0(T^*\XX)$.  The inequality $(\ge)$ in \eqref{lemmanormaforme2} is trivial, so we only have to prove $(\le)$.
                Notice first that a partitioning argument shows that 
                \begin{equation}
					|\omega|_c=\esssup\{\chi_B\omega\,\cdot\,d\pi_1\wedge\cdots \wedge d\pi_k\mid\pi_1,\dots ,\pi_k\in\LIP(\XX)\text{ with $|d\pi_1|,\dots,|d\pi_k|\le 1\ \mass$-a.e.\ on $B$}\}.
				\end{equation}  
                Fix then $\pi_1,\dots,\pi_k\in \LIP(\XX)$ and let $B$ be such that $|d\pi_1|,\dots,|d\pi_k|\le 1\ \mass$-a.e.\ on $B$. We want to show that $\chi_B\omega\,\cdot\,d\pi_1\wedge\cdots \wedge d\pi_k$ is bounded above by the right-hand-side of \eqref{lemmanormaforme2}.
                
                We can assume that $B$ is compact and that $\pi_1,\dots,\pi_k$ have compact support.
                Define $\pi_i^t\defeq h_t \pi_i$ and recall that $|d \pi_i^t|\le e^{-Kt} h_t|d \pi_i|\ \mass$-a.e.\ for every $i=1,\dots, k$. Also, for any $\eta\in (0,1)$, there exists $\eta'\in (0,1)$ and a compact  subset $B'\subseteq B$ with $\mass(B\setminus B')<\eta$ such that, for $t\in (0,\eta')$, we have
                $$e^{-Kt}h_t| d\pi_i|\le (1+\eta)\quad \text{$\mass$-a.e.\ on $B'$},$$
                for every $i=1,\dots,k$.
                For  $\mass$-a.e.\ $x\in B'$, if $t\in (0,\eta')$ and if $r$ is small enough, $\pi_i^t$ is $(1+\eta)$-Lipschitz on $B_r(x)$, for $i=1,\dots,k$. Indeed, we recall that $h_t |d \pi_i|$ is a continuous function, and that the local version of the Sobolev-to-Lipschitz property holds on $\RCD(K,N)$ spaces. This easily yields the conclusion.
                
				\subsection{Proof of Proposition \ref{formsandcurrents}}
					Take $\omega$ as in the statement. We verify that $T$ is a current as in \cite{lan11}. Multilinearity is clear, as well as locality. We thus have to check only the continuity axiom. Notice first that $|T_\omega|= |\omega|_c\in L^\infty$, by the previous lemma.

					Fix then $(f^j)\subseteq\LIP_c(\XX)$, $f\in \LIP_c(\XX)$ with $f^j\rightarrow f$ uniformly and such that $\supp f^j\subseteq K$ for a common compact set $K$, and $k$ sequences $(\pi_i^j)_j$, for $i=1,\dots,k$, where $\pi^j_i$ is $C$-Lipschitz (for some fixed $C$) for every $j\in\NN$ and $i=1,\dots,k$, and $\pi_i^j\rightarrow\pi_i$ pointwise on $\XX$.
					We have to prove that
					\begin{equation}
						\mathcal{E}^j\defeq  \int f\omega\,\cdot\, d\pi_1\wedge\cdots\wedge d\pi_k-\int f^j\omega\,\cdot\, d\pi_1^j\wedge\cdots\wedge d\pi_k^j\rightarrow 0.
					\end{equation}  
					As 
					\begin{equation}
						\Big|\int f^j\omega\,\cdot\, d\pi_1^j\wedge\cdots\wedge d\pi_k^j-\int f\omega\,\cdot\, d\pi_1^j\wedge\cdots\wedge d\pi_k^j\Big|\le C^k\|f-f^j\|_{L^1(|\omega|_c\mass\mressmall K)},
					\end{equation}
					we see that there is no loss of generality in assuming $f^j=f$ for every $j$.
					
					We notice that $	\mathcal{E}^j=	\mathcal {T}^j_1+\cdots+	\mathcal {T}^j_k$, where		
					\begin{equation}
						\mathcal {T}^j_i\defeq\int f\omega\,\cdot\, d\pi_1^j\wedge\cdots\wedge d\pi_{i-1}^j\wedge d(\pi_i -\pi_i^j)\wedge d\pi_{i+1}\wedge \cdots\wedge d\pi_k,
					\end{equation}
					so that we focus on a fixed $i=1,\dots,k$ and we prove that $	\mathcal {T}^j_i\rightarrow 0$.
					We now take $\tilde \omega\in \mathrm{TestForm}_k(\XX)$ with compact support such that $\|f\omega-\tilde\omega\|_{L^1}<\epsilon$. 
					We then use an immediate approximation argument to integrate by parts in
					\begin{equation}
						\begin{split}
							{\mathcal{\tilde T}}^j_i&\defeq	\int \tilde \omega \,\cdot\, d\pi_1^j\wedge\cdots\wedge d\pi_{i-1}^j\wedge d(\pi_i -\pi_i^j)\wedge d\pi_{i+1}\wedge \cdots\wedge d\pi_k \\&=(-1)^{i}\int \delta\tilde\omega\,\cdot\,( \pi_i-\pi_i^j) \, d\pi_1^j\wedge\cdots\wedge d\pi_{i-1}^j\wedge d\pi_{i+1}\wedge \cdots\wedge d\pi_k,
						\end{split}
					\end{equation}
					and we notice that the right-hand side converges to $0$ as $j\rightarrow\infty$.
					Hence, we have
                    $$\limsup_{j\to\infty} |\mathcal{T}_i^j|\le 		\limsup_{j\to\infty} |\mathcal{T}_i^j-	\mathcal{\tilde T}_i^j|+	\limsup_{j\to\infty} |	\mathcal{\tilde T}_i^j|\le \epsilon\cdot C^k,$$
                    so that, as $\epsilon$ was arbitrary, we conclude the first part.
                    Now we prove the other implication. We divide the proof in three steps.
                    
                    \textbf{Step 1}.
					Letting $(f,\pi_1,\dots,\pi_k)\in\mathcal{D}^k(\XX)$, we show that 
					\begin{equation}
						T(f,\pi_1,\dots,\pi_k)\le \int |f| |d\pi_1| \cdots |d\pi_k| \, \dd|T|.
					\end{equation}
					By the locality axiom for currents, we see that we can assume that $\pi_1,\dots,\pi_k$ have compact support. Define $\pi_i^t\defeq h_t \pi_i$ and recall that $|d \pi_i^t|\le e^{-Kt} h_t|d \pi_i|$ holds $\mass$-a.e.\ for every $i=1,\dots, k$.
					Notice that it is enough to show, for $t\in (0,1)$, that
					\begin{equation}\label{vefadcs}
						T(f,\pi_1^t,\dots,\pi_k^t)\le e^{-Kkt} \int_K |f| (h_t|d \pi_1|) \cdots (h_t|d \pi_k|) \dd|T|,
					\end{equation}	 
					and then let $t\searrow 0$,	using the continuity axiom. 
					
					Therefore, we focus on \eqref{vefadcs}. 
					First, notice that 
					$
					T(\,\cdot\,,\pi_1^t,\dots,\pi_k^t)
					$ induces an $L^\infty$ function $g$ with  $g\le |T|\in L^\infty$.  Take a Lebesgue point $x$ for $g$ and $|T|$, and fix $\epsilon\in (0,1)$. Now, if $r$ is small enough, $\pi_i^t$ is $[\epsilon+e^{-Kt}(h_t |d \pi_i|)(x)]$-Lipschitz on $B_r(x)$, for $i=1,\dots,k$. Indeed, we recall that $h_t |d \pi_i|$ is a continuous function, and that the local version of the Sobolev-to-Lipschitz property holds on $\RCD(K,N)$ spaces. By the locality axiom for currents, we have
					$$
					\Big|\int_{B_r(x)}g\,\dd\mass\Big|=|T(\chi_{B_r(x)},\pi_1^t,\dots,\pi_k^t)|\le \int_{B_r(x)} \prod_{i=1}^k (\epsilon+ e^{-Kt}(h_t |d \pi_i|)(x))\,\dd |T|,
					$$
					so that, once we divide both sides by $\mass(B_r(x))$,  let $r\searrow 0$, and then $\epsilon\searrow 0$, we see that
					\begin{equation}
						|g|(x)\le |T|(x)e^{-Kkt}\prod_{i=1}^k (h_t |d \pi_i|)(x),
					\end{equation} 
					whence the claim, by integration.
                    
					\textbf{Step 2}. Now take  $u_1,\dots,u_n\in\LIP(\XX)$ Lipschitz and $(c_i^j)_{i=1,\dots,k;\, j=1,\dots,n}\subseteq\RR$. For every choice of functions $(f,\pi_1,\dots,\pi_k)\in\mathcal{D}^k(\XX)$, we can write
					\begin{equation}
						T(f,\pi_1,\dots,\pi_k)=T\Big(f,\sum_{j=1}^n c_1^j u_j,\dots,\sum_{j=1}^n c_k^j u_j\Big)+\mathcal{R}_1+\cdots+\mathcal{R}_k, 
					\end{equation}
					where, for $i=1,\dots,k$,
					\begin{equation}\label{brgvfsad}
						\mathcal{R}_i=
						T\Big(f,\sum_{j=1}^n c_1^j u_j,\dots, \sum_{j=1}^n c_{i-1}^j u_j,\Big( \pi_i-\sum_{j=1}^n c_{i}^j u_j \Big),\pi_{i+1},\dots,\pi_k\Big)
					\end{equation}
					and hence
					\begin{equation}
						|		\mathcal{R}_i|\le 	\int |f| \Big |\sum_{j=1}^n c_1^j d u_j\Big| \cdots \Big|d \pi_i-\sum_{j=1}^n c_{i}^j d u_j\Big|\cdots |d \pi_k|\,\dd|T|,
					\end{equation}
					thanks to \textbf{Step 1}.
					Now, by multilinearity, 
					$$
					T\Big(f,\sum_{j=1}^n c_1^j u_j,\dots,\sum_{j=1}^n c_k^j  u_j\Big)= \sum_{1\le i_1,\dots,i_k\le n} T(f c_1^{i_1}\cdots c_k^{i_k}  ,u_{i_1},\dots,u_{i_k}).
					$$
					All in all, we have
					$$
					\Big|	T(f,\pi_1,\dots,\pi_k)- \sum_{1\le i_1,\dots,i_k\le n} T(f c_1^{i_1}\cdots c_k^{i_k}, u_{i_1},\dots,u_{i_k})\Big|\le |\mathcal{R}_1|+\cdots  |\mathcal{R}_k|.
					$$
					By a partitioning and approximation argument, we see that the above continues to holds even in the case in which $(c_i^j)_{i=1,\dots,k; \, j=1,\dots,n}\subseteq L^\infty$. 
					
					Now, take $(f_l,\pi_{l,1},\dots,\pi_{l,k})_{l}\subseteq\mathcal{D}^k(\XX)$ and $(c_{l,i}^j)_{i=1,\dots,k;\, j=1,\dots,n;\,l}\subseteq L^\infty$ be finitely many elements indexed by $l$. Summing the above equation,
                    we obtain
					\begin{equation*}
						\Big| \sum_{l} 	T(f_l,\pi_{l,1},\dots,\pi_{l,k})-\sum_l  \sum_{1\le i_1,\dots,i_k\le n} T(f_lc_{l,1}^{i_1}\cdots c_{l,k}^{i_k},u_{i_1},\dots,u_{i_k} )\Big|\le  \sum_l [|\mathcal{R}_{l,1}|+\cdots+|\mathcal{R}_{l,1}|]\eqdef\mathcal{R},
					\end{equation*}		
					with the obvious meaning for $\mathcal{R}_{l,1},\dots, \mathcal{R}_{l,k}$, as in \eqref{brgvfsad}. Now, let $I_1,\dots I_{c_{k,n}}$ be the increasing multi-indices for $k$ elements over $n$ (hence $c_{k,n}={n\choose k}$), so that
					\begin{equation}
						\Big| \sum_{l} 	T(f_l,\pi_{l,1},\dots,\pi_{l,k})-\sum_l \sum_{I_h}\sum_{\{i_1,\dots,i_k\}=I_h} T(f_l \sign((i_1,\dots,i_k), I_h)c_{l,1}^{i_1}\cdots c_{l,k}^{i_k},u_{I_h} )\Big|\le \mathcal{R}
					\end{equation}
					where, if $(\sigma(i_1),\dots,\sigma(i_k))=I_h$, $u_{I_h}$ stands for $u_{\sigma(i_1)},\dots,u_{\sigma(i_k)}$ and $\sign((i_1,\dots,i_k), I_h)$ means the sign of the permutation $\sigma$.  
					By recalling \textbf{Step 1}, we see that
					\begin{equation*}
						\Big| \sum_{l} 	T(f_l,\pi_{l,1},\dots,\pi_{l,k})\Big|\le  \sum_{I_h}\sum_{\{i_1,\dots,i_k\}=I_h} \int \Big|\sum_l f_l \sign((i_1,\dots,i_k), I_h)c_{l,1}^{i_1}\cdots c_{l,k}^{i_k}\Big|\prod_{j\in I_h}|d u_j|\,\dd|T|+ \mathcal{R}.
					\end{equation*}
					
					Now we consider the following form with compact support, say $\supp\eta\subseteq K$, with $K$
                    independent of $l$:
					\begin{equation}
						\eta\defeq \sum_l f_l \, d\pi_{l,1}\wedge \cdots \wedge d\pi_{l,k}.
					\end{equation} 
					By a further partitioning and approximation argument, we see that we can take, in place of the vector fields $(d u_j)_j$, an orthonormal basis $(e_j)_j$.
                    Taking
                    $$c_{l,i}^j\defeq d \pi_{l,i}\,\cdot\,e_j\in L^\infty$$
                    for every $l,i,j$, 
                    we see that in this case the remainder $\mathcal{R}$ vanishes.
					Hence,
					\begin{equation}\label{vfdacs}
						\Big| \sum_{l} 	T(f_l,\pi_{l,1},\dots,\pi_{l,k})\Big|\le  \sum_{I_h}\int \Big|\eta\,\cdot\, \bigwedge_{j\in I_h}e_j\Big|\,\dd|T|\le c_{n,k}\int |\eta|\,\dd|T|\le c_{n,k}\sqrt{|T|(K)}\|\eta\|_{L^2(|T|)}.
					\end{equation}
                    
					\textbf{Step 3}. By \eqref{vfdacs}, we have that, for every $K\subseteq\XX$ compact, the assignment 
					\begin{equation}
						\big\{\eta\in\mathrm{TestForms}_k(\XX)\,:\,\supp\eta\subseteq K\big\}\ni \sum_l f_l \, d\pi_{l,1}\wedge \cdots \wedge d\pi_{l,k}\mapsto \sum_l T(f_l\chi_K,\pi_{l,1},\dots,\pi_{l,k})
					\end{equation}
					is represented by $\omega_{K,T}\in L^2(\Lambda^k T^*\XX)$.  It is also easy to show that $|\omega_{K,T}|_c=|T|$ holds $\mass$-a.e.\ on $K$. We can now obtain $\omega_T$ by considering $\omega_{\bar B_R(o),T}$, for some fixed $o\in\XX$, and letting $R\rightarrow\infty$.
				\qed

			\subsection{Proof of Theorem \ref{maincurr}}
Letting $(f,\pi_1,\dots,\pi_{k-1})\in\mathcal{D}^{k-1}(\XX)$, if $\sigma\in \LIP_c(\XX)$ is such that $\sigma=1$ on $\supp f$,
by definition of boundary we then have
\begin{equation}
\begin{split}
		\partial T_\omega(f,\pi_1,\dots,\pi_{k-1})&=T_\omega(\sigma,f,\pi_1,\dots,\pi_{k-1})=\int_\XX \omega\,\cdot\,\sigma \, df\wedge d\pi_1\wedge\cdots\wedge d\pi_{k-1}\,\dd\mass\\
		&=\int_\XX \omega\,\cdot\, d(f  d\pi_1\wedge\cdots\wedge d\pi_{k-1})\,\dd\mass,
\end{split}
\end{equation}
which vanishes, thanks to the assumption, by an approximation argument. 
			\qed

				\subsection{Proof of Theorem \ref{bgrvfdcassc}}
					By linearity, we can just take $\eta= f_0\,df_1\wedge\cdots\wedge df_{k-1}$ with $f_0,\dots,f_{k-1}\in\mathrm{TestF}(\XX)$ with compact support, so that, if $\sigma\in \LIP_c(\XX)$ is such that $\sigma=1$ on $\supp f_0$, we get
					\begin{equation}
						\int_\XX \omega\,\cdot\,d\eta\,\dd\mass=	\int_\XX \omega\,\cdot\,\sigma \, df_0\wedge df_1\wedge\cdots\wedge df_{k-1}\,\dd\mass=T_\omega(\sigma,f_0,\dots,f_{k-1}).
					\end{equation}
                    Since the latter equals $\partial T(f_0,f_1,\dots,f_{k-1})=0$, this
					concludes the proof.
				\qed

	 \bibliographystyle{abbrv}

\begin{thebibliography}{10}

\bibitem{AH16}
L.~Ambrosio and S.~Honda.
\newblock New stability results for sequences of metric measure spaces with uniform {Ricci} bounds from below.
\newblock In {\em Measure theory in non-smooth spaces}, pages 1--51. De Gruyter Open, Warsaw, 2017.

\bibitem{Ambrosio2017local}
L.~Ambrosio and S.~Honda.
\newblock Local spectral convergence in {${\rm RCD}^*(K,N)$} spaces.
\newblock {\em Nonlinear Anal.}, 177(part A):1--23, 2018.

\bibitem{AmbrosioKirchheim00}
L.~Ambrosio and B.~Kirchheim.
\newblock Currents in metric spaces.
\newblock {\em Acta Math.}, 185(1):1--80, 2000.

\bibitem{ABPrank}
G.~Antonelli, C.~Brena, and E.~Pasqualetto.
\newblock The rank-one theorem on $\mathrm{RCD}$ spaces.
\newblock {\em Analysis \& PDE}, 17(8):2797--2840, 2024.

\bibitem{brueboundary}
E.~Bru{\`e}, A.~Naber, and D.~Semola.
\newblock {Boundary regularity and stability for spaces with Ricci bounded below}.
\newblock {\em Invent. Math.}, 228(2):777--891, 2022.

\bibitem{BrueNaberSemola23Milnor}
E.~Bru{\`e}, A.~Naber, and D.~Semola.
\newblock Fundamental groups and the {M}ilnor conjecture.
\newblock To appear on Annals of Mathematics, 2023.

\bibitem{BruPasSem20}
E.~Bru{\`e}, E.~Pasqualetto, and D.~Semola.
\newblock Rectifiability of {$\rm{RCD}(K,N)$} spaces via {$\delta$}-splitting maps.
\newblock {\em Ann. Fenn. Math.}, 46(1):465--482, 2021.

\bibitem{bru2021constancy}
E.~Bru{\`e}, E.~Pasqualetto, and D.~Semola.
\newblock {Constancy of the dimension in codimension one and locality of the unit normal on $\mathrm{RCD}(K, N)$ spaces}.
\newblock {\em Ann. Sc. Norm. Super. Pisa Cl. Sci. (5)}, 24(3):1765--1816, 2023.

\bibitem{bru2019rectifiability}
E.~Bru\`e, E.~Pasqualetto, and D.~Semola.
\newblock Rectifiability of the reduced boundary for sets of finite perimeter over {${\rm RCD}(K,N)$} spaces.
\newblock {\em J. Eur. Math. Soc. (JEMS)}, 25(2):413--465, 2023.

\bibitem{brupisem}
E.~Bru{\`e}, A.~Pigati, and D.~Semola.
\newblock Topological regularity and stability of noncollapsed spaces with {R}icci curvature bounded below.
\newblock Preprint, arXiv: 2405.03839, 2024.

\bibitem{Cheeger-Colding96}
J.~Cheeger and T.~H. Colding.
\newblock Lower bounds on {R}icci curvature and the almost rigidity of warped products.
\newblock {\em Ann. of Math. (2)}, 144(1):189--237, 1996.

\bibitem{Cheeger-Colding97I}
J.~Cheeger and T.~H. Colding.
\newblock On the structure of spaces with {R}icci curvature bounded below. {I}.
\newblock {\em J. Differential Geom.}, 46(3):406--480, 1997.

\bibitem{Cheeger-Colding97II}
J.~Cheeger and T.~H. Colding.
\newblock On the structure of spaces with {R}icci curvature bounded below. {II}.
\newblock {\em J. Differential Geom.}, 54(1):13--35, 2000.

\bibitem{CJNrect}
J.~Cheeger, W.~Jiang, and A.~Naber.
\newblock {Rectifiability of singular sets of noncollapsed limit spaces with Ricci curvature bounded below}.
\newblock {\em Ann. of Math. (2)}, 193(2):407 -- 538, 2021.

\bibitem{DancerWang03}
A.~Dancer and M.~Y. Wang.
\newblock Integrability and the {E}instein equations.
\newblock In {\em Symplectic and contact topology: interactions and perspectives ({T}oronto, {ON}/{M}ontreal, {QC}, 2001)}, volume~35 of {\em Fields Inst. Commun.}, pages 89--101. Amer. Math. Soc., Providence, RI, 2003.

\bibitem{DPG17}
G.~De~Philippis and N.~Gigli.
\newblock Non-collapsed spaces with {R}icci curvature bounded from below.
\newblock {\em J. \'{E}c. Polytech. Math.}, 5:613--650, 2018.

\bibitem{Gigli14}
N.~Gigli.
\newblock Nonsmooth differential geometry -- an approach tailored for spaces with {R}icci curvature bounded from below.
\newblock {\em Mem. Amer. Math. Soc.}, 251(1196):v+161, 2018.

\bibitem{GigliHan14}
N.~Gigli and B.-X. Han.
\newblock Independence on {$p$} of weak upper gradients on {$RCD$} spaces.
\newblock {\em J. Funct. Anal.}, 271(1):1--11, 2016.

\bibitem{GP19}
N.~Gigli and E.~Pasqualetto.
\newblock {\em Lectures on nonsmooth differential geometry}, volume~2 of {\em SISSA Springer Series}.
\newblock Springer, Cham, 2020.

\bibitem{HarveySearle17}
J.~Harvey and C.~Searle.
\newblock Orientation and symmetries of {A}lexandrov spaces with applications in positive curvature.
\newblock {\em J. Geom. Anal.}, 27(2):1636--1666, 2017.

\bibitem{Hatcher}
A.~Hatcher.
\newblock {\em Algebraic topology}.
\newblock Cambridge University Press, 2002.

\bibitem{HondaOrient}
S.~Honda.
\newblock Ricci curvature and orientability.
\newblock {\em Calc. Var. Partial Differential Equations}, 56(6):Paper No. 174, 47, 2017.

\bibitem{Kapovitch07}
V.~Kapovitch.
\newblock Perelman's stability theorem.
\newblock In {\em Surveys in differential geometry. {V}ol. {XI}}, volume~11 of {\em Surv. Differ. Geom.}, pages 103--136. Int. Press, Somerville, MA, 2007.

\bibitem{KapMon19}
V.~Kapovitch and A.~Mondino.
\newblock On the topology and the boundary of {$N$}-dimensional {$\rm{RCD}(K,N)$} spaces.
\newblock {\em Geom. Topol.}, 25(1):445--495, 2021.

\bibitem{lan11}
U.~Lang.
\newblock Local currents in metric spaces.
\newblock {\em J. Geom. Anal.}, 21(3):683--742, 2011.

\bibitem{Liu13}
G.~Liu.
\newblock 3-manifolds with nonnegative {R}icci curvature.
\newblock {\em Invent. Math.}, 193(2):367--375, 2013.

\bibitem{LytSta18}
A.~Lytchak and S.~Stadler.
\newblock Conformal deformations of ${CAT}(0)$ spaces.
\newblock {\em Mathematische Annalen}, 373(1--2):155--163, 2017.

\bibitem{MatPor}
R.~Matveev and J.~W. Portegies.
\newblock Intrinsic flat and {Gromov}--{Hausdorff} convergence of manifolds with {Ricci} curvature bounded below.
\newblock {\em J. Geom. Anal.}, 27(3):1855--1873, 2017.

\bibitem{Mondino-Naber14}
A.~Mondino and A.~Naber.
\newblock Structure theory of metric measure spaces with lower {R}icci curvature bounds.
\newblock {\em J. Eur. Math. Soc. (JEMS)}, 21(6):1809--1854, 2019.

\bibitem{Otsu91}
Y.~Otsu.
\newblock On manifolds of positive {R}icci curvature with large diameter.
\newblock {\em Math. Z.}, 206(2):255--264, 1991.

\bibitem{Rajala12-2}
T.~Rajala.
\newblock Interpolated measures with bounded density in metric spaces satisfying the curvature-dimension conditions of {S}turm.
\newblock {\em J. Funct. Anal.}, 263(4):896--924, 2012.

\bibitem{SchoenYau82}
R.~Schoen and S.~T. Yau.
\newblock Complete three-dimensional manifolds with positive {R}icci curvature and scalar curvature.
\newblock In {\em Seminar on {D}ifferential {G}eometry}, volume 102 of {\em Ann. of Math. Stud.}, pages 209--228. Princeton Univ. Press, Princeton, NJ, 1982.

\bibitem{Zhu93}
S.-H. Zhu.
\newblock A finiteness theorem for {R}icci curvature in dimension three.
\newblock {\em J. Differential Geom.}, 37(3):711--727, 1993.

\end{thebibliography}

\end{document}